\documentclass[12pt, twoside]{article}
\usepackage{amsmath,amsthm,amssymb}
\usepackage{times}
\usepackage{enumerate}

\def\titlerunning#1{\gdef\titrun{#1}}
\makeatletter
\def\author#1{\gdef\autrun{\def\and{\unskip, }#1}\gdef\@author{#1}}
\def\address#1{{\def\and{\\\hspace*{18pt}}\renewcommand{\thefootnote}{}%
\footnote {#1}}%
\markboth{\autrun}{\titrun}}
\makeatother

\def\subjclass#1{{\renewcommand{\thefootnote}{}%
\footnote{\emph{Mathematics Subject Classification (2010):} #1}}}
\def\keywords#1{\par\medskip
\noindent\textbf{Keywords.} #1}

\usepackage{graphicx}
\usepackage{mathtools}

\usepackage[colorlinks=true, pdfstartview=FitV, linkcolor=blue, citecolor=blue, urlcolor=blue,pagebackref=false]{hyperref}

\usepackage{xcolor}

\hypersetup{
    colorlinks,%
    citecolor=black,%
    filecolor=black,%
    linkcolor=black,%
    urlcolor=black
}

\usepackage{amsmath}
\usepackage{amssymb}
\usepackage{amsthm}
\newcommand{\mathbbm}{}

\usepackage{atbeginend}

\usepackage{psfrag}

\usepackage{enumerate}

\newtheorem{theorem}{Theorem}[section]
\newtheorem{corollary}[theorem]{Corollary}
\newtheorem{proposition}[theorem]{Proposition}
\newtheorem{lemma}[theorem]{Lemma}

\numberwithin{equation}{section}

\theoremstyle{definition}
\newtheorem{definition}[theorem]{Definition}
\theoremstyle{remark}
\newtheorem{remark}[theorem]{Remark}

\DeclareMathOperator\capacity{cap}
\DeclareMathOperator\Range{Range}
\DeclareMathOperator\supp{supp}
\DeclareMathOperator\dist{d}

\DeclareMathOperator\Cov{Cov}
\DeclareMathOperator\diam{diam}
\def\d{\mathrm{d}}

\newcommand{\mcup}{\textstyle \bigcup\limits}
\newcommand{\vvviiiggg}{\Big}

\newcommand{\bally}{B}
\newcommand{\m}{\eta}
\newcommand{\M}{L}
\newcommand{\Exp}{\mathop{\mathrm{Exp}}}

\newcommand{\eps}{\varepsilon}
\newcommand{\Z}{\mathbb{Z}}
\newcommand{\R}{\mathbb{R}}

\newcommand{\V}{V}

\newcommand{\1}[1]{{\mathbbm{1}}_{#1}}

\usepackage{courier}

\usepackage{array}
\newcolumntype{e}{>{\displaystyle}r @{\,} >{\displaystyle}c @{\,} >{\displaystyle}l}

  \newcounter{constant}
  \newcommand{\newconstant}[1]{\refstepcounter{constant}\label{#1}}
  \newcommand{\useconstant}[1]{c_{\textnormal{\tiny \ref{#1}}}}
  \newcounter{iconst}
  \newcommand{\newiconst}[1]{\refstepcounter{iconst}\label{#1}}
  \newcommand{\useiconst}[1]{\gamma_{\textnormal{\tiny \ref{#1}}}}
\def\clap#1{\hbox to 0pt{\hss#1\hss}}

\def\mathclap{\mathpalette\mathclapinternal}

\def\mathclapinternal#1#2{\clap{$\mathsurround=0pt#1{#2}$}}

\usepackage{calc}
\def\partext#1{\text{\parbox[t]{\textwidth - 2cm}{#1}}}

\def\arraypar#1{\parbox[c]{\textwidth - 2cm}{\centering #1}}

\usepackage{environ}
\NewEnviron{display}{
\begin{equation}\begin{array}{c}
  \arraypar{\BODY}
\end{array}\end{equation}
}

\def\urltilda{\kern -.15em\lower .7ex\hbox{\~{}}\kern .04em}

\allowdisplaybreaks

\begin{document}


\baselineskip=17pt


\titlerunning{Soft local times and decoupling of interlacements}

\title{Soft local times and decoupling of random interlacements}

\author{Serguei Popov \and Augusto Teixeira}

\date{\today}

\maketitle

\address{S. Popov: Dep. of Statistics, Institute of Mathematics, Statistics and Scientific Computation, University of Campinas -- UNICAMP, rua S\'ergio Buarque de Holanda 651, 13083--859, Campinas SP, Brazil\\
e-mail: {\itshape \texttt{popov@ime.unicamp.br}} \qquad
website: {\href{http://www.ime.unicamp.br/~popov/}{\itshape \texttt{www.ime.unicamp.br/\urltilda popov/}}}
\and
A.Teixeira: Instituto Nacional de Matem\'atica Pura e Aplicada -- IMPA, estrada Dona Castorina 110, 22460--320, Rio de Janeiro RJ, Brazil\\
e-mail: {\itshape \texttt{augusto@impa.br}} \qquad
website: {\href{http://www.impa.br/~augusto/}{\itshape \texttt{w3.impa.br/\urltilda augusto/}}}
}


\newiconst{c:main'}
\newiconst{c:main''}
\begin{abstract}
In this paper we establish a decoupling feature of the random interlacement process $\mathcal{I}^u \subset \Z^d$ at level $u$, $d \geq 3$.
Roughly speaking, we show that observations of $\mathcal{I}^u$ restricted to two disjoint subsets~$A_1$ and~$A_2$ of~$\Z^d$ are approximately independent, once we add a sprinkling to the process $\mathcal{I}^u$ by slightly increasing the parameter $u$.
Our results differ from previous ones in that we allow the mutual distance between the sets $A_1$ and $A_2$ to be much smaller than their diameters.
We then provide an important application of this decoupling for which such flexibility is crucial.
More precisely, we prove that, above a certain critical threshold~$u_{**}$, the probability of having long paths that avoid $\mathcal{I}^u$ is exponentially small, with logarithmic corrections for $d=3$.

To obtain the above decoupling, we first develop a general method for comparing the trace left by two Markov chains on the same state space.
This method is based in what we call the soft local time of a chain.
In another crucial step towards our main result, we also prove that any discrete set can be ``smoothened'' into a slightly enlarged discrete set, for which its equilibrium measure behaves in a regular way.
Both these auxiliary results are interesting in themselves and are presented independently from the rest of the paper.

.
\end{abstract}

\section{Introduction and results}
\label{s_intro}

This work is mainly concerned with the decoupling of the random interlacements model introduced by A.S. Sznitman in \cite{Szn09}.
In other words, we show that the restrictions of the interlacement set~$\mathcal{I}^u$ to two disjoint subsets~$A_1$ and~$A_2$ of $\Z^d$ are approximately independent in a certain sense.
To this aim, we first develop a general method, based on what we call \emph{soft local times}, to obtain an approximate stochastic domination between the ranges of two general Markov chains on the same state space.

To apply this coupling method to the model of random interlacements, we first need to modify the sets $A_1$ and $A_2$ through a procedure we call \emph{smoothening}.
This consists of enclosing a discrete set $A \subset \Z^d$ into a slightly enlarged set $A'$, whose equilibrium distribution behaves ``regularly'', resembling what happens for a large discrete ball.

Finally, as an application of our decoupling result, we obtain upper bounds for the connectivity function of the vacant set $\mathcal{V}^u = \Z^d \setminus \mathcal{I}^u$, for intensities~$u$ above a critical threshold $u_{**}$.
These bounds are considerably sharp, presenting a behavior very similar to that of their corresponding lower bounds.

We believe that these four results are interesting in their own. Therefore, we structured the article in a way so they can be read independently from each other.
Below we give a more detailed description of each of these results.

\subsection{Decoupling of random interlacements}
\label{s:decoupl_intro}
The primary interest of this work lies in the study of the random interlacements process, recently introduced by A.-S. Sznitman in~\cite{Szn09}.
The construction of random interlacements was originally motivated by the analysis of the trace left by simple random walk on large graphs, for instance a large discrete torus or a thick discrete cylinder.
Intuitively speaking, this model describes the texture in the bulk left by these trajectories, when the random walk is let run up to specific time scales.

Recently, a great effort has been spent in the study of this model \cite{SS09}, \cite{SS10}, \cite{T10}, \cite{Szn09c}, \cite{Szn11}, \cite{RS12}, and \cite{CP12} as well as in establishing rigorously the relation between random interlacements and the trace left by random walks on large graphs, see \cite{S09b}, \cite{Win08}, \cite{TW10} and \cite{CTW10}.
Recent works have also shown a connection between: random interlacements, the Gaussian free field \cite{Szn12a}, \cite{Szn12b} and cover times of random walks \cite{Bel12}.

Roughly speaking, the model of random interlacements can be described as an Poissonian cloud of doubly infinite random walk trajectories on $\Z^d$, $d \geq 3$. The density of this cloud is governed by an intensity parameter $u > 0$ so that, as $u$ increases, more and more trajectories enter the picture. We denote by $\mathcal{I}^u$ the so called interlacement set, given by the union of the range of these random walk trajectories. Regarding $\mathcal{I}^u$ as a random subset of $\Z^d$, its law $\mathcal{Q}^u$ can be characterized as the only distribution in $\{0,1\}^{\Z^d}$ such that
\begin{equation}
 \label{e:charact}
  \mathcal{Q}^u[K \cap \mathcal{I}^u = \varnothing]
= \exp\{-u \capacity (K)\}, \text{ for every finite $K \subset \Z^d$},
\end{equation}
where $\capacity(K)$ stands for the capacity of the set~$K$ defined in~\eqref{e:cap}, see Proposition~1.5 of~\cite{Szn09} for the characterization~\eqref{e:charact}.



The main difficulty in understanding properties of $\mathcal{I}^u$ is related to its long range dependence.
Let us note for instance that
\begin{equation}
 \label{e:cov}
 \Cov(\1{x \in \mathcal{I}^u}, \1{y \in \mathcal{I}^u})
  \sim \frac{c_{d}u}{\|x-y\|^{d-2}} \quad\text{ as } \quad \| x-y \| \to \infty,
\end{equation}
see \cite{Szn09}, (1.68).
Such a slow decay of correlations imposes several obstacles to the analysis of random interlacements, especially in low dimensions.
Various methods have been developed in order to circumvent this dependence, some of which we briefly summarize below.

\vspace{3mm}

Let us explain what is the type of statement we are after. Consider two subsets~$A_1$ and~$A_2$
of~$\Z^d$ with diameters smaller or equal to~$r$ and within distance at least~$s \geq 1$ from each other. Suppose also that we are given two functions $f_1:\{0,1\}^{A_1} \to [0,1]$ and $f_2:\{0,1\}^{A_2} \to [0,1]$ that depend
only on the configuration of the random interlacements inside the sets~$A_1$ and~$A_2$ respectively. In~\cite{Szn09}, (2.15) it was established that
\begin{equation}
 \label{e:basic}
 \Cov(f_1,f_2) \leq c_{d} u
  \frac{\capacity(A_1) \capacity(A_2)}{s^{d-2}} \leq c'_d u \Big(\frac {r^2}s \Big)^{d-2},
\end{equation}
see also Lemma~2.1 of~\cite{Bel11}. Although the above inequality retains the slow polynomial decay observed in~\eqref{e:cov}, it has been useful in various situations, see for instance Theorem~4.3 of~\cite{Szn09} and Theorem~0.1 of~\cite{Bel11}.

A first improvement on~\eqref{e:basic} appeared already in the pioneer work~\cite{Szn09}, where the author considers what he calls `sprinkling' of the law $\mathcal{I}^u$, see Section~3. In the sprinkling procedure, ``independent paths are thrown in, so as to dominate long range dependence''
of~$\mathcal{I}^u$.

Given two functions $f_1$ and $f_2$ as above, which are non-increasing in~$\mathcal{I}^u$,
the technique of Section~3 of~\cite{Szn09}
allows one to conclude that, roughly speaking,
\begin{equation}
 \label{e:sprink}
 \mathcal{Q}^u[f_1 f_2] \leq \mathcal{Q}^{u(1 + \delta)}[f_1]
\mathcal{Q}^{u(1 + \delta)}[f_2] + c_{d, \alpha} \Big( \frac rs \Big)^{\alpha},
\end{equation}
where $\alpha$ is arbitrary and the sprinkling parameter $\delta$ goes to zero as a polynomial of $(r/s)$. Note that the above represents a big improvement over~\eqref{e:basic}: in exchange to  restricting ourselves to non-increasing functions and introducing a sprinkling term, we obtain a much faster decay in the error term.
Since its introduction, the sprinkling technique has been useful for several problems on random interlacements, see~\cite{Szn09b}, \cite{Szn12}, and~\cite{TW10}.

The most recent result on decoupling bounds for interlacements can be found in~\cite{Szn12} and stands out for several reasons.
First, it generalizes the ideas behind~\cite{SS10} and~\cite{T10} for random interlacements on quite general classes of graphs (besides~$\Z^d$), as long as they satisfy certain heat kernel estimates.
Secondly, the tools developed in~\cite{Szn12} work both to show existence and absence of percolation through a unified framework and give novel results even in the particular case of~$\Z^d$, see also the beautiful applications in~\cite{RS12b} and~\cite{DRS12}.

On the other hand, the results in~\cite{Szn12} were designed having a renormalization scheme in mind. Thus, their use is restricted to bounding the so-called `cascading events', which behave in a certain hierarchical way, see the details in Section~3 of~\cite{Szn12}.

Although the results in~\eqref{e:basic}, \eqref{e:sprink} and \cite{Szn12} complement each other, they suffer from the same drawback, as they implicitly assume that
\begin{equation}
 \label{e:sgeqr}
 \text{the distance between $A_1$ and $A_2$ is at least of the same order as their diameters.}
\end{equation}
This can be a major obstruction in some applications, such as the one we present in Section~\ref{s_appl} on the decay of connectivity.

\vspace{3mm}

Let us now state the main theorem of the present paper, which can be regarded as an improvement on~\eqref{e:sprink}.
Later we will describe precisely how it differs quantitatively from previously known results.

Below, $\useiconst{c:main'}$ and $\useiconst{c:main''}$ are positive constants depending only on the dimension $d$.

\begin{theorem}
\label{t:coroll}
Let $A_1,A_2$ be two non intersecting subsets of~$\Z^d$, with at least one of them being finite. Let $s$ be the distance between~$A_1$ and $A_2$, and $r$ be the minimum of their diameters.
 Then, for all $u>0$ and $\eps\in(0,1)$ we have
\begin{itemize}
 \item[(i)] for any increasing functions $f_1 : \{0,1\}^{A_1}
\to [0,1]$ and $f_2 : \{0,1\}^{A_2} \to [0,1]$,
\begin{equation}
\label{main_incr}
 \mathcal{Q}^u[f_1 f_2] \leq \mathcal{Q}^{(1+\eps)u}[f_1]
\mathcal{Q}^{(1+\eps)u}[f_2]
    + \useiconst{c:main'} (r+s)^d \exp(-\useiconst{c:main''} \eps^2us^{d-2});
\end{equation}
 \item[(ii)] for any decreasing functions $f_1 : \{0,1\}^{A_1} \to [0,1]$ and $f_2 : \{0,1\}^{A_2} \to [0,1]$,
\begin{equation}
\label{main_decr}
 \mathcal{Q}^u [f_1 f_2] \leq \mathcal{Q}^{(1-\eps)u} [f_1]
 \mathcal{Q}^{(1-\eps)u} [f_2]
    + \useiconst{c:main'} (r+s)^d \exp(-\useiconst{c:main''}\eps^2us^{d-2}).
\end{equation}
\end{itemize}
We of course assume the above functions~$f_1$ and~$f_2$ to be measurable (recall that one of the sets $A_1$ or $A_2$ may be infinite).
\end{theorem}

The above theorem is a direct consequence of the slightly more general Theorem~\ref{t_main}. Note that the opposite inequalities to \eqref{main_incr} and \eqref{main_decr} follow without error terms (and with $\eps = 0$) by the FKG inequality, which was proved for random interlacements in~\cite{Tei09}, Theorem~3.1.

Let us now stress what are the main improvements offered by the above bounds over previously known results.
First, there is no requirement that~$s$ should be larger than~$r$ as in~\eqref{e:sgeqr} (and again, one of the sets may even be infinite).
Moreover, these error bounds feature an explicit and fast decay on~$s$, even as~$\eps=\eps(s,r)$ goes (not too rapidly) to zero.
We include in Remark~\ref{r:optimal} some observations on how close to optimal one can expect \eqref{main_incr} and \eqref{main_decr} to be.

\subsection{Connectivity decay}
\label{s:connectivity_intro}
As an application of Theorem~\ref{t:coroll}, we establish a result on the decay of connectivity on the vacant set $\mathcal{V}^u = \Z^d \setminus \mathcal{I}^u$. More precisely, for $u$ large enough (see Theorem~\ref{t:connect} for details), for $d \geq 4$,
\begin{equation}
  \label{e:d4_display}
  \mathcal{Q}^u[0 \xleftrightarrow{\mathcal{V}^u} x] \leq \useiconst{c:connect}
  \exp \{ -\useiconst{c:connect2} \|x\|\},
  \text{ for every $x \in \Z^d$.}
\end{equation}
where $\useiconst{c:connect}$ and $\useiconst{c:connect2}$ depend only on $d$. If $d = 3$ and $u$
is large enough, then for any $b>1$ there exist $\useiconst{c:connect_dim3} = \useiconst{c:connect_dim3}(u,b)$ and
$\useiconst{c:connect2_dim3} = \useiconst{c:connect2_dim3}(u,b)$ such that
\begin{equation}
\label{e:d3_display}
 \mathcal{Q}^u[0 \xleftrightarrow{\mathcal{V}^u} x] \leq \useiconst{c:connect_dim3} \exp \Big\{ -\useiconst{c:connect2_dim3} \frac{\|x\|}{\log^{3b}\|x\|} \Big\}, \text{ for every $x \in \Z^3$.}
\end{equation}
see Theorem~\ref{t:connect} and Remark~\ref{r:decay} for more details.

Let us stress that the above bounds greatly improve on the previously known results, proved in Theorem~0.1 of \cite{SS10}. There, the authors establish similar bounds but with $\| x \|$ replaced by $\|x\|^\rho$ for some unknown exponent $\rho \in (0,1)$.
Our bounds on the other hand are considerably sharp, as they closely resemble the corresponding lower bounds, see Remark~\ref{r:decay} for details.

Note that the exponential decay in~\eqref{e:d4_display} is also observed in independent percolation models, see for instance Theorem~(5.4) of \cite{Gri99}, p.88 and \cite{Men86}.
However, due to the strong dependence present in $\mathcal{V}^u$, its validity was at first not obvious to the authors.
For one reason, it is known that the logarithmic factor in~\eqref{e:d3_display} cannot be dropped, see Remark~3.2 below.
Similar types of non exponential decays in dependent percolation models can be found for instance in~(1.65) and~(2.21) of~\cite{Szn09} and Remark~3.7 2) of~\cite{T10}.

Finally we would like to stress that our proof of \eqref{e:d4_display}--\eqref{e:d3_display} is general enough in the sense that it could be adapted for other dependent percolation models, as long as they satisfy a suitable decoupling inequality.
See the discussion in Remark~\ref{r:general_perc}.

\subsection{Soft local times}
In Section~\ref{s:simul} we develop a technique to prove approximate stochastic domination of the trace left by a Markov chain on a metric space.
This is an important ingredient in proving Theorem~\ref{t:coroll} and we also expect it to be useful in future applications. To illustrate this technique, consider an irreducible Markov chain $(Z_i)_{i \geq 1}$ on a finite state space~$\Sigma$ having~$\pi$ as its unique stationary measure.

A typical model to keep in mind is a random walk on a torus that jumps from~$z$ to a uniformly chosen point in the ball centered in $z$ with radius $k$. By transitiveness, the uniform distribution $\pi$ is clearly invariant. Intuitively speaking, if we let this Markov chain run for a long time~$t$, we expect the law of covered set $\{Z_1, \dots, Z_t\}$ to be ``reasonably close'' to that of a collection $\{W_1, \dots, W_t\}$ of i.i.d. points in~$\Sigma$ distributed according to~$\pi$.
This is made precise in the following result, which is a consequence of Corollary~\ref{c:coupleZ}.

\begin{proposition}
\label{p:domainMarkov}
Let $(Z_i)_{i \geq 1}$ be a Markov chain on a finite set $\Sigma$, with transition probabilities $p(z,z')$, initial distribution~$\pi_0$, 
and stationary measure~$\pi$.
Then we can find a coupling~$\mathbb{Q}$ between~$(Z_i)$ and an i.i.d.\ collection~$(W_i)$ (with law~$\pi$), in such a way that for any $\lambda > 0$ and $t \geq 0$,
\begin{equation}
  \label{e:domainMarkov}
  \begin{split}
    \mathbb{Q}\big[\{Z_1, \dots, Z_t\} & \subset \{W_1, \dots, W_R\}\big]\\
    & \geq \mathbb{Q}\Big[\xi_0\pi_0(z)+ \sum_{j = 1}^{t-1} \xi_j p(Z_j,z) \leq\lambda \pi(z), \text{ for all $z \in \Sigma$}\Big],
  \end{split}
\end{equation}
where $\xi_i$ are i.i.d.\ $\mathrm{Exp}(1)$ random variables, independent of $R$, a $\mathrm{Poisson}(\lambda)$-distributed random variable.
\end{proposition}

Observe that the above statement can have interesting consequences in bounding the hitting time of a given subset of $\Sigma$, see \eqref{e:hittingtime} for a precise definition.

We call the sum $\sum_j \xi_j p(Z_j,z)$ the \emph{soft local time} of the chain $Z_j$. To justify this notation, observe that instead of counting the number of visits to a fixed site (which corresponds to the usual notion of local time), we are summing up the \emph{chances} of visiting such site, multiplied by i.i.d.\ mean-one positive factors. See also Theorem~\ref{t:expectsingle}.

In Remark~\ref{r:bumps} we describe the main advantages of Proposition~\ref{p:domainMarkov} over previous domination techniques and how it allows us to drop the assumption \eqref{e:sgeqr}.

Later in Section~\ref{s:simul}, we establish general estimates on the expectation, variance and exponential moments of the soft local time $\sum_j \xi_j p(Z_j,z)$.
These are based on regularity assumptions on the transition probabilities $p(\cdot, \cdot)$ and are valuable when estimating the right hand side of~\eqref{e:domainMarkov} by means of exponential Chebyshev's inequalities, see Theorems~\ref{t:expectsingle}, \ref{t:2ndmoment} and~\ref{t:expmoment}.

Now, we comment on the main method employed to prove results such as Proposition~\ref{p:domainMarkov} above.
One can better visualize the picture in a continuous space, so we use another example to illustrate the method:
 assume that we are given a sequence
of (not necessarily independent nor Markovian)
 random variables $S_1, S_2, S_3, \ldots$
taking values in the interval $[0,1]$, and let~$T$ be a finite stopping time.
As in~\eqref{e:domainMarkov}, we attempt to dominate this process by a sequence $U_1,\ldots,U_N$, where~$(U_k)$ are i.i.d.\ Uniform$[0,1]$ random variables, and~$N$ is a Poisson random variable independent of $(U_k)$.
More precisely, we want to construct a coupling between the two sequences in such a way that
\begin{equation}
\label{eq:softlocal}
 \{S_1, \ldots, S_T\} \subseteq \{U_1,\ldots,U_N\}
\end{equation}
with probability close to one. We assume that the law of~$S_k$
conditioned on $S_1,\ldots,S_{k-1}$ is absolutely continuous with respect to the Lebesgue measure on $[0,1]$, see \eqref{e:transitone}.

Our method for obtaining such a coupling is illustrated on Figure~\ref{f:softlocal}.
Consider a Poisson point process in $[0,1]\times \R_+$ with rate~$1$.
Then, one can obtain a realization of the $U$-sequence by simply retaining the first coordinate of the points lying below a given threshold (the dashed line in Figure~\ref{f:softlocal}) corresponding to the parameter of the Poisson random variable~$N$.

Now, in order to obtain a realization of the $S$-sequence using the same Poisson point process, one proceeds as follows:
\begin{itemize}
\item first, take the density $g(\cdot)$ of~$S_1$ and multiply it by
  the unique positive number $\xi_1$ so that there is exactly one
  point of the Poisson process lying on the graph of $\xi_1 g$ and
  nothing strictly below it;
\item then consider the conditional density $g(\cdot\mid S_1)$
  of~$S_2$ given~$S_1$ and find the smallest constant $\xi_2$ so that
  exactly two points lie underneath $\xi_2 g(\cdot\mid S_1) + \xi_1
  g(\cdot)$;
\item continue with $g(\cdot\mid S_1,S_2)$, and so on, up to time~$T$,
  as shown on Figure~\ref{f:softlocal}.
\end{itemize}

In Proposition~\ref{p:xiGi}, we show that the collection of points obtained through the above procedure has the same law as $(S_1, S_2, \dots)$ and is independent of the random variables $\xi_i$, which are i.i.d. with law Exp($1$).
We call the sum $\xi_1 g(\cdot) + \xi_2 g(\cdot\mid S_1) + \cdots$ the \emph{soft local time} of the process~$S$ (which coincides with the sum in the right-hand side of~\eqref{e:domainMarkov} in the Markovian case).
Clearly, if the soft local time (the gray area on the picture) is below the dashed line, then the domination in~\eqref{eq:softlocal} holds.
To obtain the probability of a successful coupling, one has to estimate the probability that the soft local time lies below the dashed line.
In several cases, this reduces to a large deviations estimate.
\begin{figure}
\centering \includegraphics{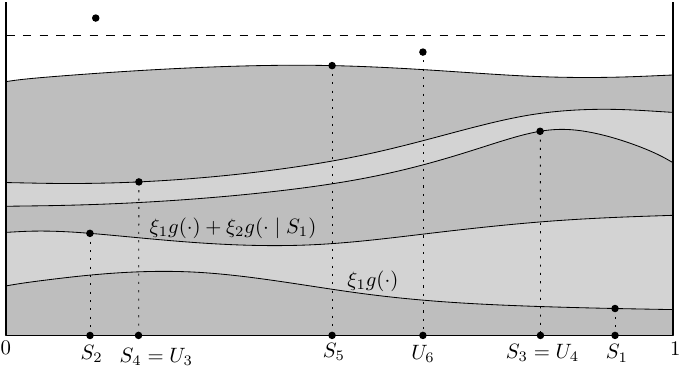}
  \caption{Soft local times: the construction of the process~$S$
  (here, $T=5$, $N=6$,
   $U_k=S_k$ for $k=1,2,5$);
it is very important to observe that the points of the
  two processes need not necessarily appear in the same order with respect to the vertical axis, see Remark~\ref{r:bumps}.}
  \label{f:softlocal}
\end{figure}

After developing a general version of this technique in Section~\ref{s:simul}, we adapt this theory to random interlacements in Section~\ref{s:altern_constr}.
More precisely, we present an alternative construction of the interlacement set $\mathcal{I}^u$ restricted to some $A \subset \Z^d$.
In this construction, we split each trajectory composing $\mathcal{I}^u$ into a collection of excursions in and out of $A$.
This induces a Markov chain on the space of excursions and the techniques of soft local times helps us control the range of such soup.

After the conclusion of this article, we learned that a technique similar to the soft-local times was introduced in the special case $\Sigma = (0,1) \subset \mathbb{R}$ in order to study local minima of the Brownian motion. see  Claim~1.5 of \cite{EJP309}.

We believe that the method of soft local time can be useful in other contexts besides random interlacements.
For example, when considering a random walk trajectory on a finite graph (such as a torus or a discrete cylinder), one can naturally be interested in the degree of independence in the pictures left by the walker on disjoint subsets of the graph.
The approach followed in this paper is likely to be successful in this situation as well.
We also believe this technique could give alternative proofs or generalize results on the coupling of systems of independently moving particles, see Proposition~5.1 of \cite{PSSS11} for an example of such a statement.

\subsection{Smoothening of discrete sets}
\label{s:smoothening_intro}
As we mentioned before, in order to estimate the probability of having a successful coupling using the soft local times technique, we need some regularity conditions on the transition densities of the Markov chain.
When applying this to the excursions composing the random interlacements, this translates into a condition on the regularity of the entrance distributions on the sets $A_1$ and $A_2$, which may not hold in general (picture for instance a set with sharp points).

To overcome this difficulty, we develop a technique to enlarge the original discrete sets $A_1$ and $A_2$ into slightly bigger discrete sets with ``sufficiently smooth'' boundaries, so that their entrance probabilities satisfy the required regularity conditions.

The exact result we are referring to is given in Proposition~\ref{p:fatA}, but we provide here a small preview of its statement.
There exist positive constants $c,c',c'',s_0$ (depending only on dimension) such that for any $s\geq s_0$ and any finite set $A\subset \Z^d$, there exist a set $A^{(s)}$ with $A \subseteq A^{(s)} \subseteq B(A,s)$ and
\begin{equation}
\label{eq:pre_smooth}
 P_x [X_{H} = y] \leq c P_x [X_{H} = y'],
\end{equation}
for all $y,y'\in\partial A^{(s)}$ with $\|y - y'\| \leq c''s$, and all $x$ such that $\|x-y\|\geq c's$.
Where $X$ is the simple random walk and~$H$ is the hitting time of the set~$A^{(s)}$.
That is, the entrance measure to the set~$A^{(s)}$ is ``comparable'' in close sites of the boundary, as long as the starting point of the random walk is sufficiently far away.

It is important to observe that for example a large (discrete) ball fulfills the above property, 
while a large box does not, since its entrance probabilities at the faces are typically much 
smaller than those at the corners (to see this, observe that
one can obtain using the arguments similar e.g.\ to the proof of Theorem~1.4 of~\cite{MP13}
that the harmonic measure at a corner of the box is at least $O(n^{-\gamma})$ for some 
$\gamma<1$, while for ``generic'' sites on the faces it is $O(n^{-1})$).

\subsection{Plan of the paper}
\label{s:plan_intro}
The paper is organized in the following way.
In Section~\ref{ss:ri} we formally define the model of random interlacements, and state our main decoupling result.
In Section~\ref{s_appl}, we state the connectivity decay stated in \eqref{main_decr} and \eqref{main_incr}, see Theorem~\ref{t:connect}.
In Section~\ref{s:simul} we present a general version of the method of soft local times.
Then, in Section~\ref{s:altern_constr} this method is used to introduce an alternative construction of random interlacements, which is better suited for decoupling configurations on disjoint sets.
In the same section we reduce the proof of our main Theorem~\ref{t_main} to a large deviations estimate for the soft local time of excursions.
In Section~\ref{s:proof}, we estimate the probability of these large deviation events and conclude the proof of Theorem~\ref{t_main} under a set of additional assumptions on the entrance measures of $A_{1,2}$.
While this set of assumptions may not be satisfied for arbitrary $A_{1,2}$, we show in Section~\ref{s:reduction} that this is not really an issue, as one can always enlarge slightly the sets of interest (with the procedure referred above as smoothening) so that these modified sets satisfy the necessary regularity assumptions.
Before going to (quite technical) Section~\ref{s:reduction}, 
in Section~\ref{s:connectivity} we prove the result on the decay of connectivity for the vacant set, corresponding to \eqref{e:d4_display} and \eqref{e:d3_display}.

\section{Random interlacements: formal definitions and main result}
\label{ss:ri}

In this paper, we use the following convention concerning constants:
 $c_1,c_2,c_3,\ldots$, as well as
$\gamma_1,\gamma_2,\gamma_3,\ldots$
denote strictly positive constants depending
only on dimension~$d$. Dependence of constants
on additional parameters appears in the notation. For example, $c_\alpha$
denotes a constant depending only on~$d$ and $\alpha$.
Also $c$-constants are ``local'' (used only in a small neighborhood
of the place of the first appearance) while $\gamma$-constants
are ``nonlocal'' (they appear in propositions and ``important'' formulas).

Let us now introduce some notation and describe the model of random interlacements. In addition, we recall some useful facts concerning the model.

For $a \in \mathbb{R}$, we write $\lfloor a \rfloor$ for the largest integer smaller or equal to $a$ and recall that
\begin{equation}
\label{e:floorconvex}
\lfloor t a+(1-t)b \rfloor \in [\min \{a,b\},\max \{a,b\}], \text{ for all } a,b \in \mathbb{Z} \text{ and } t \in [0,1].
\end{equation}

We say that two points $x,y \in \mathbb{Z}^d$ are neighbors if they are at Euclidean distance (denoted by $\|\cdot\|$) exactly~$1$ (we also write
$x \leftrightarrow y$ when~$x$ and~$y$ are neighbors).
This induces a graph structure and a notion of connectedness in $\mathbb{Z}^d$.

If $K \subset \mathbb{Z}^d$, we denote by $K^c$ its complement and by $\bally(K,r)$ the $r$-neighborhood of $K$ with respect to the Euclidean distance, i.e. the union of the balls $\bally(x,r)$ for $x \in K$. The diameter of $K$ (denoted by $\diam(K)$) is the supremum of $\| x - y \|_\infty$ with $x,y \in K$,
where $\|\cdot\|_\infty$ is the maximum norm.
Let us define its internal boundary
 $\partial K = \{x \in K; x \leftrightarrow y$ for some
$y \in K^c\}$.

In this article the term \textit{path} always denotes finite, nearest neighbor paths, i.e.\ some
$\mathcal{T}:\{0,\dots,n\} \rightarrow \mathbb{Z}^d$ such that $\mathcal{T}(l) \leftrightarrow \mathcal{T}(l+1)$ for $l = 0,\dots, n-1$.
In this case we say that the length of $\mathcal{T}$ is~$n$.

Let us denote by~$W_+$ and~$W$ the spaces of infinite, respectively doubly infinite, transient trajectories
\begin{equation}
 \label{e:WW}
 \begin{split}
  W_+ = \Big\{w:\mathbb{Z}_+ \rightarrow \mathbb{Z}^d; w(l) \leftrightarrow w(l+1), \text{ for each } l \geq 0 \text{ and } \| w(l) \| \xrightarrow[l \rightarrow \infty]{} \infty \Big\},\\
  W = \Big\{w:\mathbb{Z} \rightarrow \mathbb{Z}^d; w(l) \leftrightarrow w(l+1), \text{ for each } l \in \mathbb{Z} \text{ and } \| w(l) \| \xrightarrow[|l| \rightarrow \infty]{} \infty \Big\}.
 \end{split}
\end{equation}
We endow these spaces with the $\sigma$-algebras $\mathcal{W}_+$ and $\mathcal{W}$ generated by the coordinate maps $\{X_n\}_{n \in \mathbb{Z}_+}$ and $\{X_n\}_{n \in \mathbb{Z}}$.

Let us also introduce the entrance time of a finite set $K \subset \mathbb{Z}^d$
\begin{equation}
\label{e:entranceexit}
H_K(w) = \inf \{k ; X_k(w) \in K\}, \text{ for } w \in W_{(+)}, \\
\end{equation}
and for $w \in W_+$, we define the hitting time of $K$ as
\begin{equation}
\label{e:hittingtime}
\widetilde{H}_K(w) = \inf \{k \geq 1; X_k(w) \in K\}.
\end{equation}

Let $\theta_k:W \rightarrow W$ stand for the time shift given by $\theta(w)(\cdot) = w(\,\cdot + k)$ (where $k$ could also be a random time).

For $x \in \mathbb{Z}^d$, (recall that $d \geq 3$) we can define the law $P_x$ of a simple random walk starting at~$x$ on the space $(W_+,\mathcal{W}_+)$. If~$\rho$ is a measure on $\mathbb{Z}^d$, we write $P_\rho = \sum_{x\in \mathbb{Z}^d} \rho(x) P_x$.


Let us introduce, for a finite $K \subset \mathbb{Z}^d$, the equilibrium measure
\begin{equation}
e_K(x) = \1{x \in K} P_x[\widetilde H_K = \infty], \text{ for } x \in \mathbb{Z}^d,
\end{equation}
the capacity of $K$
\begin{equation}
\label{e:cap}
\text{cap}(K) = e_K(\mathbb{Z}^d)
\end{equation}
and the normalized equilibrium measure
\begin{equation}
\label{e:normalizedeK}
\overline e_K (x) = e_K(x)/ \text{cap}(K), \text{ for } x \in \mathbb{Z}^d.
\end{equation}

We mention the following bound on the capacity of a ball of radius $r\geq 1$
\newconstant{c:cap}
\newconstant{c:cap1}
\begin{equation}
\label{e:estimatecap}
\capacity(\bally(0,r)) \asymp r^{d-2}
\end{equation}
 see Proposition~6.5.2 of~\cite{LL10} (here and in the sequel we write
$f(r)\asymp g(r)$ when $\useconstant{c:cap} g(r)\leq f(r)
 \leq \useconstant{c:cap1} g(r)$ for strictly positive
constants $\useconstant{c:cap},\useconstant{c:cap1}$ depending only on the dimension).

Let $W^*$ stand for the space of doubly infinite trajectories in~$W$
modulo time shift,
\begin{equation}
W^* = W/\sim \mbox{, where } w \sim w' \mbox{ if } w(\cdot) = w'(k + \cdot), \mbox{ for some } k \in \mathbb{Z},
\end{equation}
endowed with the $\sigma$-algebra
\begin{equation}
\mathcal{W}^* = \{ A \subset W^*; (\pi^*)^{-1}(A) \in \mathcal{W}\},
\end{equation}
which is the largest $\sigma$-algebra making the canonical projection $\pi^*:W \rightarrow W^*$ measurable. For a finite set $K \subset \mathbb{Z}^d$, we denote as $W_K$ the set of trajectories in~$W$ which meet the set~$K$ and define $W^*_K = \pi^*(W_K)$.

Now we are able to describe the intensity measure of the Poisson point process which governs the random interlacements.

For a finite set $K \subset \mathbb{Z}^d$, we consider the measure $Q_K$ in $(W, \mathcal{W})$ supported in $W_K$ such that, given $A,B \in \mathcal{W}_+$ and $x \in K$,
\begin{equation}
\label{e:QK}
Q_K[(X_{-n})_{n \geq 0} \in A, X_0 = x, (X_n)_{n \geq 0} \in B] = P_x[A\mid\tilde H_K = \infty]P_x[B] e_K(x).
\end{equation}
Theorem~1.1 of \cite{Szn09} establishes the existence of a unique $\sigma$-finite measure $\nu$ in $W^*$ such that,
\begin{equation}
\label{e:nuQK}
1_{W_K^*}\cdot \nu = \pi^* \circ Q_K, \text{ for any finite set $K \subset \mathbb{Z}^d$.}
\end{equation}
The above equation is the main tool to perform calculations on random interlacements.

We then introduce the spaces of point measures on $W^* \times \mathbb{R}_+$ and $W_+ \times \mathbb{R}_+$
\begin{equation}
 \label{e:Omega}
 \begin{split}
  \Omega & = \bigg\{ \; \;\omega = \smash{\sum_{i\geq 1}}
\delta_{(w^*_i,u_i)} \Big| \; \;
  \begin{split}
   & w^*_i \in W^*, u_i \in \mathbb{R}_+ \mbox{ and }
\omega(W^*_K \times [0,u]) < \infty\\
   & \mbox{ for every finite } K \subset \mathbb{Z}^d \mbox{ and } u\geq 0.
  \end{split} \;\; \bigg\}
 \end{split}
\end{equation}
and endowed with the $\sigma$-algebra $\mathcal{A}$ generated by the evaluation maps $\omega \mapsto \omega(D)$ for $D \in \mathcal{W}^* \otimes \mathcal{B}(\mathbb{R}_+)$. Here $\mathcal{B}(\cdot)$ denotes the Borel $\sigma$-algebra.

We let $\mathbb{P}$ be the law of a Poisson point process on $\Omega$ with
intensity measure $\nu \otimes \d u$, where $\d u$ denotes the Lebesgue measure on $\mathbb{R}_+$.
Given $\omega = \sum_i\delta_{(w^*_i,u_i)} \in \Omega$, we define \textit{the interlacement} and
the \textit{vacant set} at level $u$ respectively as the random subsets of $\mathbb{Z}^d$:
\begin{gather}
\label{e:Iu}
\mathcal{I}^u (\omega) = \bigg\{ \bigcup_{i; u_i \leq u} \textnormal{Range}(w^*_i) \bigg\} \mbox{ and}\\
\label{e:Vu}
\mathcal{V}^u(\omega) = \mathbb{Z}^d \setminus \mathcal{I}^u(\omega).
\end{gather}
In~\cite{Szn09} (0.13), Sznitman introduced the critical value
\begin{equation}
\label{e:u*}
u_* = \inf \big\{u \geq 0; \mathbb{P}[\mathcal{V}^u
\text{ contains an infinite connected component}] = 0\big\},
\end{equation}
where the vacant set undergoes a phase transition in connectivity.
It is known that $0 < u_* < \infty$ for all $d \geq 3$,
see~\cite{Szn09}, Theorem~3.5 and~\cite{SS09}, Theorem~3.4.
Moreover, it is also proved that if existent, the infinite connected component of the vacant set must be unique, see~\cite{Tei09b}, Theorem~1.1.



It is important to mention also that, as shown in~\cite{Szn09},
\begin{equation}
 \label{e:ergodic}
 \begin{array}{c}
 \text{the law of the random set $\mathcal{I}^u$ is invariant and ergodic with respect to}\\
 \text{translations of the lattice $\Z^d$.}
 \end{array}
\end{equation}

\subsection{Decoupling: the main result}
\label{s_result}

We now state our main result on random interlacements. It provides us with a way to decouple the intersection of the interlacement set $\mathcal{I}^u$ with two disjoint subsets $A_1$ and $A_2$ of~$\Z^d$. Namely, we couple the original interlacement process~$\mathcal{I}^u$ with two \emph{independent} interlacements processes~$\mathcal{I}^u_1$ and $\mathcal{I}^u_2$ in such a way
that~$\mathcal{I}^u$ restricted on~$A_k$ is ``close'' to~$\mathcal{I}^u_k$, for $k = 1,2$, with probability rapidly going to~$1$ as the distance between the sets increases. This is formulated precisely in
\begin{theorem}
\label{t_main}
Let $A_1,A_2$ be two nonintersecting subsets of~$\Z^d$, with at least one of them being finite. Abbreviate $s=\dist(A_1,A_2)$ and $r= \min\{ \diam(A_1), \diam(A_2)\}$. Then, there are positive constants $\useiconst{c:main'}$ and $\useiconst{c:main''}$ (depending only on the dimension $d$) such that for all $u > 0$ and $\eps\in(0,1)$ there exists a coupling $\mathbb{Q}$ between $\mathcal{I}^u$ and two independent random interlacements processes, $(\mathcal{I}^u_1)_{u \geq 0}$ and $(\mathcal{I}^u_2)_{u \geq 0}$ such that
\begin{equation}
\label{eq_mainmain}
 \mathbb{Q} \big[ \mathcal{I}^{u(1-\eps)}_k \cap A_k \; \subseteq \; \mathcal{I}^u \cap A_k \; \subseteq \; \mathcal{I}^{u(1+\eps)}_k,
\text{ $k = 1, 2$} \big] \geq 1 - \useiconst{c:main'} (r+s)^d \exp(-\useiconst{c:main''}\eps^2us^{d-2}).
\end{equation}
\end{theorem}
It is straightforward to see that
the above theorem implies the
inequality on the covariance of increasing (or decreasing) functions depending
only on~$A_1$ and~$A_2$ stated previously in Theorem~\ref{t:coroll}.
%
Also, we mention that the factor $(r+s)^d$ before the exponential
in~\eqref{eq_mainmain} can usually be reduced,
see Remark~\ref{rem_boundary_size}.

\section{Discussion, open problems, and an application of the decoupling}
\label{s_appl}
 We start this section with the following application of our main result.
We are interested in
the probability $\mathbb{P}[0 \xleftrightarrow{\mathcal{V}^u} x]$ that two far away points are connected through the vacant set. In the sub-critical case, $u > u_*$, this probability clearly converges to zero as~$\|x\|$ goes to infinity. In what follows, we will be interested in the rate in which this convergence takes place.

In Proposition~3.1 of \cite{Szn09}, it was proven that $\mathbb{P}[0 \xleftrightarrow{\mathcal{V}^u} x]$ decays at least as a polynomial in~$\|x\|$
 if $u$ is chosen large enough.
Then~in \cite{SS10} this was considerably improved,
by showing that for~$u$ large enough,
there exist $c, c'$ and $\delta > 0$ (possibly depending on~$u$), such that
\begin{equation}
 \label{e:connect}
 \mathbb{P}[0 \xleftrightarrow{\mathcal{V}^u} x] \leq c \exp\{-c' \|x\|^\delta\}, \text{ for every $x \in \Z^d$.}
\end{equation}

To be more precise, the above statement was established for all intensities $u$ above the threshold
\begin{equation}
 \label{e:ustarstar}
 u_{**}(d) = \inf \Big\{ u > 0; \text{ for some $\alpha > 0$, $\lim_{L \to \infty} L^\alpha \mathbb{P} \big[ [-L,L]^d \xleftrightarrow{\mathcal{V}^u} \partial [-2L,2L]^d\big] = 0$} \Big\}.
\end{equation}
The above critical value is known to satisfy $u_* \leq u_{**} < \infty$ (see \cite{S09}, Lemma~1.4) and a very relevant question is whether $u_*$ and $u_{**}$ actually coincide.

In~\cite{Szn12}, an important class of decoupling inequalities was introduced, implying in particular that \eqref{e:ustarstar} can be written as
\begin{equation}
 \label{e:ustarstar2}
 u_{**} = \inf \Big\{ u > 0; \text{ $\lim_{L \to \infty} \mathbb{P} \big[ [-L,L]^d \xleftrightarrow{\mathcal{V}^u} \partial [-2L,2L]^d\big] = 0$} \Big\},
\end{equation}
potentially enhancing the validity of~\eqref{e:connect}.
The above result could perhaps be seen as a step in the direction of proving $u_* = u_{**}$.

Here, we further weaken the definition of $u_{**}$ but, more importantly, we improve on the bound \eqref{e:connect} for values of $u$ above $u_{**}$. The improved result we present gives the correct exponents in the decay of the connectivity function, although for $d = 3$ they could be off by logarithmic corrections, see Remark~\ref{r:decay} below.

\newiconst{c:connect} \newiconst{c:connect2}
\newiconst{c:connect_dim3} \newiconst{c:connect2_dim3}
\begin{theorem}
 \label{t:connect}
 For $d \geq 4$, given $u > u_{**}(d)$, there exist positive constants  $\useiconst{c:connect} = \useiconst{c:connect}(d,u)$ and $\useiconst{c:connect2} = \useiconst{c:connect2}(d,u)$ such that
 \begin{equation}
\label{e:connect_d4}
  \mathbb{P}[0 \xleftrightarrow{\mathcal{V}^u} x] \leq
\useiconst{c:connect} \exp \{ -\useiconst{c:connect2} \|x\|\},
\text{ for every $x \in \Z^d$.}
 \end{equation}
 If $d = 3$ and $u > u_{**}(3)$, then for any $b>1$ there exist
$\useiconst{c:connect_dim3} = \useiconst{c:connect_dim3}(u,b)$ and
$\useiconst{c:connect2_dim3} = \useiconst{c:connect2_dim3}(u,b)$
such that
 \begin{equation}
\label{e:connect_d3}
  \mathbb{P}[0 \xleftrightarrow{\mathcal{V}^u} x] \leq
\useiconst{c:connect_dim3} \exp \Big\{ -\useiconst{c:connect2_dim3}
\frac{\|x\|}{\log^{3b}\|x\|} \Big\}, \text{ for every $x \in \Z^d$.}
 \end{equation}

 Moreover, we show that \eqref{e:ustarstar} can be written as
 \begin{equation}
 \label{e:ustarstar3}
  u_{**} = \inf \Big\{ u > 0; \liminf_{L \to \infty}
\mathbb{P} \big[ [0,L]^d \xleftrightarrow{\mathcal{V}^u}
\partial [-L,2L]^d \big] < \frac{7}{2d\cdot 21^d} \Big\}.
 \end{equation}
\end{theorem}

\begin{remark}
\label{r:decay}
The probability that a straight segment of length~$n$ is vacant is exponentially small in~$n$ when $d\geq 4$, while for $d = 3$, this probability is at least $c\exp(-c'\frac{n}{\log n})$, which corresponds to the capacity of a line segment (this follows e.g.\ from Proposition~2.4.5
of~\cite{L91}). So, \eqref{e:connect_d4} is sharp (up to constants), but the situation with~\eqref{e:connect_d3} is less clear, since in~\eqref{e:connect_d3} the power of the logarithm in the denominator is at least~$3$. We believe, however, that~\eqref{e:connect_d3} can be improved (by decreasing the power of the logarithm).
\end{remark}

\begin{remark}
\label{r:optimal}
There is a general question about how sharp is the result in~\eqref{eq_mainmain} (also in~\eqref{main_incr} and~\eqref{main_decr}). One could for instance question whether the probability in~\eqref{eq_mainmain} can be exactly~$1$, thus achieving the equality in~\eqref{main_incr}--\eqref{main_decr} (so that we would have a ``perfect domination''). Interestingly enough, Theorem~\ref{t:connect} sheds some light on this question, at least in dimension~$d=3$. Indeed, in the proof of Theorem~\ref{t:connect} we use~\eqref{main_decr} with $\eps\simeq \log^{-b}s$ to obtain the subexponential decay of~\eqref{e:connect_d3}; however, if the error term could be dropped altogether,
or even if $s$ could be substituted by $s^{1+\delta}$ (for some $\delta>0$) in that term, then (compare with the proof for $d\geq 4$) one would obtain the exponential decay for $d=3$ as well, which contradicts the remarks of the previous paragraph. This is an indication that, in general, $s^{d-2}$ in the exponent in the error term could be sharp, at least if~$\eps$ is small enough. Also, we note that one cannot hope to achieve the perfect domination if $\eps\ll s^{-(d-2)}$ simply due to~\eqref{e:cov}.

It is less clear how small the parameter~$\eps$ can be made (say, in the situation when~$s$ does not exceed~$r$). Obviously, \eqref{eq_mainmain} stops working when~$\eps=O(s^{-\frac{d-2}{2}})$, but we are unsure about how much our main result can be improved in this direction. Also, it is
interesting to observe that, contrary to the bound~\eqref{e:basic},
our estimates become \emph{better} as the parameter~$u$ increases.
\end{remark}

\begin{remark}
\label{r:general_perc}
As mentioned in Section~\ref{s:decoupl_intro}, one can obtain the exponential decay
as in~\eqref{e:connect_d4} for any percolation model with suitable monotonicity
and decoupling properties. Namely, let $\mathcal{\tilde Q}^u$ be a family of measures
on $\{0,1\}^{\Z^d}$, $d\geq 2$, indexed by a parameter $u\in[0,\infty)$.
We assume that this family is monotone in the sense that $\mathcal{\tilde Q}^{u'}$
dominates $\mathcal{\tilde Q}^u$ if $u'<u$ (as happens for the vacant set in the
random interlacement model).
Also, assume that there are positive constants $b,c,M,\delta$ such that:
for any increasing events $A_1,A_2$ that depend on disjoint boxes of
size~$r$ within distance at least~$s$ from each other, we have for all $u>0$ and $\eps\in(0,1)$
\[
 \mathcal{\tilde Q}^u [A_1 A_2] \leq \mathcal{\tilde Q}^{(1-\eps)u} [A_1]
 \mathcal{\tilde Q}^{(1-\eps)u} [A_2]
    + c (r+s)^M \exp(-\useiconst{c:main''}\eps^{b} u s^{1+\delta}).
\]
Then for all $u>u^{**}$ (where $u^{**}$ is defined as in~\eqref{e:ustarstar3} with
obvious notational changes) we would obtain the exponential decay as in~\eqref{e:connect_d4}
(again, with obvious notational changes). The proof would go through practically unaltered.
\end{remark}

\section{Soft local times and simulations with Poisson processes}
\label{s:simul}
In this section we prove a result about simulating sequences of
random variables
using Poisson processes.
Besides being interesting on itself, this result will
be a major ingredient in order to couple various random
interlacements during the proof of Theorem~\ref{t_main}.

Let $\Sigma$ be a locally compact and Polish metric space. Suppose also that we are given a measure space $(\Sigma, \mathcal{B}, \mu)$ where $\mathcal{B}$ is the Borel $\sigma$-algebra on $\Sigma$ and $\mu$ is a Radon measure, i.e., every compact set has finite $\mu$-measure.

The above setup is standard for the construction of a Poisson point process on $\Sigma$. For this, we also consider the space of Radon point measures on $\Sigma \times \mathbb{R}_+$
\begin{equation}
 \label{e:M1}
 \M = \vvviiiggg\{\m = \sum_{\lambda \in \Lambda} \delta_{(z_\lambda,
 v_\lambda)}; z_\lambda \in \Sigma, v_\lambda \in \mathbb{R}_+
\text{ and } \m(K) < \infty \text{ for all compact $K$} \vvviiiggg\},
\end{equation}
endowed with $\sigma$-algebra $\mathcal{D}$ generated by the evaluation maps $\m \mapsto \m(S)$, $S \in \mathcal{B} \otimes \mathcal{B}(\R)$.

Note that the index set $\Lambda$ in the above sum has to be countable. However, we do not use $\Z_+$ for this indexing, because $(z_\lambda, v_\lambda)$ will be ordered later and only then we will endow them with an ordered indexing set.

One can now canonically construct a Poisson point process $\m$ on the space $(\M, \mathcal{D}, \mathbb{Q})$ with intensity given by $\mu \otimes \d v$, where $\d v$ is the Lebesgue measure on $\mathbb{R}_+$. For more details on this construction, see for instance~\cite{R08}, Proposition~3.6 on p.130.


The proposition below provides us with a way to simulate a random element of $\Sigma$ using the Poisson point process $\m$. Although this result is very simple and intuitive, we provide here its proof for the sake of completeness and the reader's convenience.

\begin{proposition}
\label{p:simulate}
Let $g:\Sigma \to \mathbb{R}_+$ be a measurable function with $\int g(z) \mu(\d z) = 1$. For $\m = \sum_{\lambda \in \Lambda} \delta_{(z_\lambda, v_\lambda)} \in \M$, we define
\begin{equation}
 \xi = \inf \{ t \geq 0; \text{ there exists $\lambda \in \Lambda$ such that $t g(z_\lambda) \geq v_\lambda$}\},
\end{equation}
see Figure~\ref{f:xi}. Then under the law $\mathbb{Q}$ of the Poisson point process~$\m$,
\begin{enumerate}[(i)]\addtolength{\itemsep}{2mm}\vspace{2mm}
 \item \label{e:io} there exists a.s.\
 a unique $\hat{\lambda} \in \Lambda$ such that
 $\xi g(z_{\hat{\lambda}}) = v_{\hat{\lambda}}$,
 \item \label{e:xiio} $(z_{\hat{\lambda}}, \xi)$ is distributed as $g(z) \mu(dz)
\otimes \Exp(1)$,
 \item \label{e:mprime} $\m' := \sum_{\lambda \neq \hat{\lambda}} \delta_{(z_\lambda,v_\lambda -
\xi g(z_\lambda))}$ has the same law as~$\m$ and is independent
of $(\xi, \hat{\lambda})$.
\end{enumerate}
\end{proposition}

As we have mentioned in the introduction, a statement similar to the above proposition has already been established in the special case of $\Sigma = (0,1) \subset \mathbb{R}$, in Claim~1.5 of \cite{EJP309}.

\begin{proof}
Let us first define, for any measurable $A \subset \Sigma$, the random variable
\begin{equation}
 \xi^A = \inf \{ t \geq 0; \text{ there exists $\lambda \in \Lambda$ such that $t \mathbbm{1}_A g(z_\lambda) \geq v_\lambda$}\}.
\end{equation}
Elementary properties of Poisson point processes
(see for instance~(a) and~(b) in~\cite{R08}, p.~130) yield that
\begin{equation}
\label{e:xiA}
\begin{array}{c}
 \text{$\xi^A$ is exponentially distributed
(with parameter $\int_A g(z) \mu(dz)$) and}\\
 \text{if $A$ and $B$ are disjoint, $\xi^A$ and $\xi^B$ are independent.}
\end{array}
\end{equation}

Property~(\ref{e:io}) now follows from~\eqref{e:xiA}, using that~$\Sigma$ is separable and the fact that two independent exponential random variables are almost surely distinct. Observe also that
\begin{equation}
 \mathbb{Q}[\xi \geq \alpha, z_{\hat{\lambda}} \in A] = \mathbb{Q}[\xi^{\Sigma \setminus A} > \xi^A \geq \alpha].
\end{equation}
Thus, using~\eqref{e:xiA} we can prove property~(\ref{e:xiio}) using simple properties of the minimum of independent exponential random variables.

Finally, let us establish property (\ref{e:mprime}).
We first claim that, given $\xi$, $\m'' := \sum_{\lambda \neq \hat{\lambda}} \delta_{(z_\lambda, v_\lambda)}$ is a Poisson point process, which is independent of $z_{\hat{\lambda}}$ and, conditioned on $\xi$, has intensity measure $\mathbbm{1}_{\{v > \xi g(z)\}} \cdot \mu(\d z) \otimes \d v$.

This is a consequence of the Strong Markov property for Poisson point processes and the fact that $\{(z,v) \in \Sigma \times \mathbb{R}_+; v \leq \xi g(z)\}$ is a stopping set, see Theorem~4 of~\cite{Ros82}.

To finish the proof, we observe that, given~$\xi$, $\m'$ is a mapping of~$\m''$ (in the sense of Proposition~3.7 of~\cite{R08}, p.~134). This mapping pulls back the measure $\mathbbm{1}_{\{v > \xi g(z)\}} \cdot \mu(\d z) \otimes \d v$ to $\mu(\d z) \otimes \d v$.
Noting that the latter distribution does not involve~$\xi$, we conclude the proof of~(\ref{e:mprime}) and therefore of the lemma.
\end{proof}

\begin{figure}
\centering \includegraphics[width = 0.75 \textwidth]{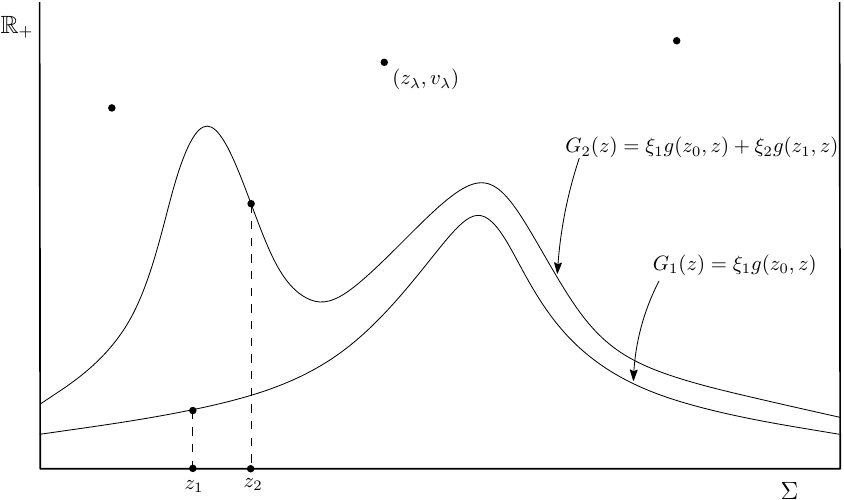}
  \caption{An example illustrating the definition of $\xi$ and $\hat{\lambda}$ in Proposition~\ref{p:simulate}. More generally, $\xi_1$, $z_1$ and $\xi_2$, $z_2$ as in \eqref{e:xis}}
  \label{f:xi}
\end{figure}

Let us now use the same Poisson point process~$\m$, to simulate not only a single random element of $\Sigma$, but a Markov chain $(Z_k)_{k \geq 1}$. For this, suppose that
in some probability space $(\M', \mathcal{D}', \mathcal{P})$
we are given a Markov chain $(Z_k)_{k \geq 1}$ on $\Sigma$ with transition densities
\begin{equation}
\label{e:transitone}
 \text{$\mathcal{P}[Z_{k+1} \in dz \mid Z_k] = g(Z_k, z) \mu(dz)$, for $k \geq 1$,}
\end{equation}
where $g(\cdot, \cdot)$ is $\mathcal{B}$-measurable in each of its coordinates and integrates
to one with respect to~$\mu$ on the second coordinate.

We moreover suppose that the starting distribution of the Markov chain is also absolutely continuous with respect to $\mu$. In fact, in order to simplify the notation, we suppose that
\begin{equation}
\label{e:startlaw}
 \text{$Z_1$ is distributed as $g(Z_0, z) \mu(\d z)$.}\\
\end{equation}
Observe that the Markov chain starts at time one, so that there is no element $Z_0$ in the chain. In fact, \eqref{e:startlaw} should be regarded as a notation for the distribution of $Z_1$, that is consistent with \eqref{e:transitone} for convenient indexing. This notation will be particularly useful in Theorem~\ref{t:2ndmoment} below.

\begin{remark}
 Observe that, in principle, $Z_k$ could be any process adapted to a filtration and the arguments of this section would still work, as long as their conditional distribution are absolutely continuous with respect to $\mu$.
 However, for simplicity we only deal with Markovian processes here, as the notations for general processes would be more complicated.
\end{remark}

Using
Proposition~\ref{p:simulate}, we introduce
\begin{equation}
\begin{aligned}
 & \xi_{1} := \inf \big\{ t \geq 0; \text{ there exists $\lambda \in \Lambda$ such that $t g(Z_0, z_\lambda) \geq v_\lambda$}\big\} \text{ and}\\
 & G_{1}(z) := \xi_{1} \; g(Z_0, z), \text{ for $z \in \Sigma$.}\\
 & (z_1, v_1) \text{ is the unique pair in $\{(z_\lambda, v_\lambda)\}_{\lambda \in \Lambda}$ with $\xi_1 G_1( z_{1}) = v_{1}$.}
\end{aligned}
\end{equation}
see Figure~\ref{f:xi}.

It is clear from Proposition~\ref{p:simulate} that $z_1$ is distributed as $Z_1$ and that the point process $\sum_{(z_\lambda,v_\lambda) \neq (z_1,v_1)} \delta_{(z_\lambda, v_\lambda - G_1(z_\lambda))}$ is distributed as $\eta$. In fact we can continue this construction starting with $\eta'$ to prove the following

\begin{proposition}
 \label{p:xiGi}
 We can proceed iteratively to define $\xi_n, G_n$ and $(z_n,v_n)$ as follows
 \begin{align}
  \label{e:xisingle}
  & \xi_{n} :=  \inf \big\{ t \geq 0; \text{ $\exists (z_\lambda, v_\lambda) \notin \{(z_k, v_k)\}_{k=1}^{n-1}$; $G_{n-1}(z_\lambda) + t g(z_{n-1}, z_\lambda) \geq v_\lambda$}\big\},\\[2mm]
  & G_{n}(z) = G_{n-1}(z) + \xi_{n} \; g(z_{n-1}, z),\\[1mm]
  & (z_n, v_n) \text{ is the unique pair $(z_\lambda, v_\lambda) \notin \{(z_k,v_k)\}_{k=1}^{n-1}$ with $G_n(z_\lambda) = v_{\lambda}$},\\
\label{e:zandZ}
  & (z_{1}, \dots, z_{n}) \overset d\sim (Z_1, \dots, Z_n) \text{ and they are independent from $\xi_1, \dots, \xi_n$,}\\
  & \sum_{\mathclap{(z_\lambda, v_\lambda) \notin \{(z_k,v_k)\}_{k=1}^{n}}} \quad \delta_{(z_\lambda, v_\lambda - G_{n}(z_\lambda))}
\text{ is distributed as $\eta$ and independent of the above.}
 \end{align}
 for all $n \geq 1$, see Figure~\ref{f:xi} for an illustration of
this iteration.
\end{proposition}

We call $G_n$ the \emph{soft local time} of the Markov chain, up to time $n$, with respect to the reference measure~$\mu$. We will justify the choice of this name in Theorem~\ref{t:expectsingle} below.

 From the above construction we have the following
\begin{corollary}
\label{c:coupleZ}
On the probability measure $\mathbb{Q}$ (where we defined the Poisson point process $\m$)  we can construct the Markov chain $(Z_k)_{k \geq 1}$, in such a way that for any measurable function $v: \Sigma \to \R_+$,
\begin{equation}
 \label{e:couplesingle}
 \mathbb{Q} \Big[ \{Z_1, \dots, Z_T\} \subseteq \{z_{\lambda}; v_{\lambda} \leq v(z_{\lambda})\} \Big] \geq \mathbb{Q} \big[ G_T(z) \leq v(z), \text{ for $\mu$-a.e. $z \in \Sigma$} \big],
\end{equation}
for any finite stopping time $T \geq 1$.
\end{corollary}

\begin{remark}
  \label{r:bumps}
  Let us now comment on how the above corollary compares with other
  techniques for approximate domination present in the literature.
  One such method is called ``Poissonization'' and is present in
  various works, see for instance \cite{Szn09}, \cite{S09} and
  \cite{TW10}.  Loosely speaking, the method of Poissonization
  attempts to compare the elements $Z_1$, $Z_2, \dots$ with $z_1$,
  $z_2, \dots$ one by one, so that one needs the transition densities
  $g(z,z')$ be close to one (in $L^1(\mu)$) uniformly over $z$.  Not
  having such requirement is the main contribution of our technique,
  which will be useful later when working with random interlacements.
\end{remark}

In order to estimate the right-hand side of~\eqref{e:couplesingle}, it is natural to resort to concentration inequalities or large deviations principles for the sum defining~$G_T$. For this it is first necessary to obtain the expectation of the soft local time~$G_T(z)$. The following proposition relates this with the expectation of the usual local time of the chain~$Z_k$ and that is the main reason why we call~$G_k$ a soft local time.

We define the local time measure of the chain $(Z_k)_{k \geq 1}$ up to time~$n$ by
\begin{equation}
 \mathcal{L}_n= \sum_{k \leq n} \delta_{Z_k}.
\end{equation}
Observe that in some examples, the probability that $z \in \Sigma$ is visited by the Markov chain could be zero for every $z \in \Sigma$ (for instance if $\mu$ is the Lebesgue measure). Therefore, we need to use a test function in order to define what we call the \emph{expected local time} of the chain. More precisely, we say that a measurable function $h: \Sigma \to \R_+$ is the \emph{expected local time density} of $(Z_k)_{k \leq n}$ with respect to $\mu$ if
\begin{equation}
 \label{e:expectlocal}
 E^{\mathcal{P}} \big( \mathcal{L}_n f \big) = \int_\Sigma f(z) h(z)\,
 \mu(\d z), \quad \text{ for every non-negative measurable $f$.}
\end{equation}
Here $n$ could also be replaced by a stopping time. An important special case occurs when $\Sigma$ is countable and $\mu$ is the counting measure. In this case, the expected local time density $h(z)$ is given simply by the expectation of the local time $\mathcal{L}_n$ at $z$:
\begin{equation}
 \label{e:Expectcount}
 E^{\mathcal{P}} \Big( \sum_{k=1}^n f(Z_k) \Big) = \sum_{k=1}^n \sum_z f(z)
\mathcal{P}[Z_k = z] = \sum_z f(z) E^{\mathcal{P}}\mathcal{L}_n (z).
\end{equation}

For what follows, we suppose that the state space $\Sigma$ contains a special element $\Delta$ which we refer to as the \emph{cemetery}. We assume that $\mu(\{\Delta\}) = 1$ and $g(\Delta,\cdot) = 1_{\{\Delta\}}(\cdot)$, or in other words, that the cemetery is an absorbing state. We write $T_{\Delta}$ for the hitting time of~$\Delta$ which is a \emph{killing time} for the chain in the sense of~\cite{FP99}, see~(2).
We will also assume that test functions~$f$ as in~\eqref{e:expectlocal}
are zero at the cemetery.

The next result relates the expected local time density with the expectation of the soft local time.
\begin{theorem}
 \label{t:expectsingle}
 Consider a state space $(\Sigma, \mathcal{B}, \mu)$ with a cemetery
state~$\Delta$ and a Markov chain $(Z_k)_{k \geq 1}$ satisfying~\eqref{e:startlaw} and~\eqref{e:transitone}. Then we have
 \begin{display}
  $E^{\mathbb{Q}}[G_{T_\Delta}(z)]$ is the expected local time density of $(Z_k)_{k \leq T_\Delta}$ as in~\eqref{e:expectlocal}.
 \end{display}
 The result is also true when $T_\Delta$ is replaced by a deterministic time.
\end{theorem}

\begin{proof}
 Given some $n \geq 1$, let us calculate
 \begin{equation}
  \begin{array}{e}
   E^{\mathcal{P}} \Big( {\sum_{k=1}^n} f(Z_k) \Big) &
= & E^{\mathcal{P}} f(Z_1) + E^{\mathcal{P}} \Big( \smash{\sum_{k=2}^n} E^{\mathcal{P}}_{Z_{k-1}} f(Z_1) \Big)\\
   & = & E^{\mathcal{P}} \Big( \sum_{k=1}^n \int f(z) g(Z_{k-1}, z)
\mu(\d z) \Big)\\
   & \overset{\eqref{e:zandZ}}= & E^{\mathbb{Q}} 
\int f(z) G_n(z) \mu(\d z)
= \int f(z) E^{\mathbb{Q}}G_n(z)\, \mu (\d z),
  \end{array}
 \end{equation}
 proving the validity of the proposition for the deterministic time~$n$.
We now let~$n$ go to infinity and the result follows from the monotone convergence theorem and the fact that~$f$ is zero at~$\Delta$.
\end{proof}
Let us remark that the above proof can be adapted to any \emph{killing time};
on the other hand, one cannot put an arbitrary \emph{stopping time} on the place
of~$T_\Delta$ in Theorem~\ref{t:expectsingle}.

Before stating the next result, let us discuss a bit further our convention on the starting distribution of the Markov chain. According to \eqref{e:startlaw}, $Z_1$ is distributed as $g(Z_0, z) \mu(\d z)$, but this was seen as a mere notation for convenient indexing and $Z_0$ had no meaning whatsoever on that equation. However, it is clear that given any $z_0 \in \Sigma$, we could plug it in the first coordinate of $g(\cdot, \cdot)$ as in \eqref{e:transitone} to define the density of $Z_1$. Then the whole construction of $\xi_k$, $G_k$ and $(z_k,v_k)$ in Proposition~\ref{p:xiGi} would depend on the specific choice of $z_0$. In the next proposition, we write $\mathbb{Q}_{z_0}$ for the measure~$\mathbb{Q}$, where the construction of $\xi_k$, $G_k$ and $(z_k, v_k)$ (recall \eqref{e:xisingle}), is obtained starting from the density $g(z_0,z)$. We also denote by $E^{\mathbb{Q}}_{z_0}$ the corresponding expectation.

\begin{remark}
\label{rem4.6}
Let us also observe that restricting the distribution of~$Z_1$ to be $g(z_0, z) \mu( \d z)$ for some $z_0 \in \Sigma$ does not represent any additional loss of generality, as~$z_0$ could be an artificial state introduced in~$\Sigma$, from which~$g(z_0, z)$ is any desired density for~$Z_1$.
\end{remark}

The next two theorems are useful in estimating the second and exponential moments of the soft local times. This will be useful in the proofs of Lemma~\ref{l_expect_Y} and Theorem~\ref{t_main}.

Next, besides calculating the expectation of~$G_k$, it is useful to estimate its second moment.
\begin{theorem}
 \label{t:2ndmoment}
 For any $z, z_0 \in \Sigma$,
 \begin{equation}
   E^{\mathbb{Q}}_{z_0} \big(G_{T_\Delta}(z) \big)^2
  \leq 4 E^{\mathbb{Q}}_{z_0} \big(G_{T_\Delta}(z)\big)
\sup_{z'_0} E^{\mathbb{Q}}_{z'_0} G_{T_\Delta}(z) .
 \end{equation}
 The result is also true with $T_\Delta$ replaced by a deterministic time.
\end{theorem}

\begin{proof}
 Given $z \in \Sigma\setminus\Delta$ and $n \geq 1$, we can write
(recall that the expectation of $(\Exp(1))^2$ equals $2$)
 \begin{align*}
   E^{\mathbb{Q}}_{z_0}& \big( G_n(z) \big)^2
= E^{\mathbb{Q}}_{z_0} \Big( \sum_{k=1}^n \xi_k g(z_{k-1},z) \Big)^2 \\
   & = E^{\mathbb{Q}}_{z_0} \Big( \sum_{k=1}^n \xi_k^2 g^2(z_{k-1},z) \Big)
+ E^{\mathbb{Q}}_{z_0} \Big(2 \sum_{k < k' \leq n}
\xi_k \xi_{k'} g(z_{k-1}, z) g(z_{k'-1},z) \Big)\\
   & \leq \sum_{k=1}^n E\xi_k^2 \sup_{z'} g(z',z) E^{\mathbb{Q}}_{z_0}
 g(z_{k-1},z) + 2 \sum_{k=1}^{n-1} \sum_{k'=k+1}^{n}
 E^{\mathbb{Q}}_{z_0} \big(g(z_{k-1},z) g(z_{k'-1},z) \big) \\
   & \leq 2 \sup_{z'} g(z',z) E^{\mathbb{Q}}_{z_0} G_n(z) + 2
 \sum_{k=1}^{n-1} \sum_{k'=k+1}^{n} E^{\mathbb{Q}}_{z_0} \big( g(z_{k-1},z)
 E^{\mathbb{Q}}_{z_0} (g(z_{k'-1},z)\mid z_{k-1} ) \big)\\
   & \leq 2 \sup_{z'} E^{\mathbb{Q}}_{z'} G_n(z) E^{\mathbb{Q}}_{z_0}
 G_n(z) + 2 \sum_{k=1}^{n-1} E^{\mathbb{Q}}_{z_0}
\Big( g(z_{k-1},z) E^{\mathbb{Q}}_{z_{k-1}} \Big( \sum_{m=1}^{n-k}
 g(z_{m-1},z) \Big) \Big)\\
& \leq 2 \sup_{z'} E^{\mathbb{Q}}_{z'} G_n(z) E^{\mathbb{Q}}_{z_0}
 G_n(z) + 2 \sup_{z'}E^{\mathbb{Q}}_{z'} \Big( \sum_{m=1}^{n-k}
 g(z_{m-1},z) \Big)
E^{\mathbb{Q}}_{z_0} \Big(\sum_{k=1}^{n-1}
g(z_{k-1},z) \Big)\\
&\leq 4 E^{\mathbb{Q}}_{z_0} \big(G_n(z)\big)
\sup_{z'_0} E^{\mathbb{Q}}_{z'_0} G_n(z),
 \end{align*}
 proving the result for the deterministic time~$n$.
Then we simply let~$n$ go to infinity and use the
monotone convergence theorem.
\end{proof}

The next result provides an estimate on the exponential moments of~$G_{T_\Delta}$,
which is clearly an important ingredient in bounding the right hand side of \eqref{e:couplesingle}. The next theorem imposes some regularity condition on the transition densities $g(\cdot, \cdot)$ (which will be encoded in~$\ell$
and~$\alpha$ below) to help in obtaining such fast decaying bounds. Intuitively speaking, the regularity condition says that if there is a big
accumulation of densities~$g$ in some point~$\hat z$, then there should be
a big
accumulation of densities in a large set~$\Gamma$.

\begin{theorem}
 \label{t:expmoment}
 Given $\hat z \in \Sigma$ and measurable $\Gamma \subset \Sigma$, let
 \begin{equation}
  \begin{split}
  \alpha & = \inf\Big\{\frac{g(z,z')}{g(z,{\hat z})};
  z \in \Sigma, z' \in \Gamma, g(z,{\hat z})>0\Big\},\\
  N(\Gamma) & = \#\{k \leq T_\Delta; z_k \in \Gamma\} ,
\text{ and}\\
  \ell & \geq \; \smash{\sup_{z' \in \Sigma}} \; g(z',\hat z).
  \end{split}
 \end{equation}
 Then, for any $v \geq 2$,
 \begin{equation*}
   \begin{split}
     \mathbb{Q}[G_{T_\Delta} & ({\hat z}) \geq v \ell]\\
     & \leq \mathbb{Q}[G_{T_\Delta}({\hat z}) \geq \ell] \Big( \exp \big\{-\big(\tfrac v2 - 1 \big) \big\} + \sup_{z'}\mathbb{Q}_{z'} \big[ \eta( \Gamma \times [0, \tfrac{1}{2} v \ell\alpha] ) \leq N(\Gamma) \big] \Big),
   \end{split}
 \end{equation*}
(recall the definition of $\eta$ in \eqref{e:M1} and observe that $\eta( \Gamma \times [0, \tfrac{1}{2} v \ell\alpha] )$ is a random variable with distribution $\textnormal{Poisson} \big( \tfrac{1}{2}v \ell \alpha\mu(\Gamma) \big)$).
\end{theorem}

Before proving the above theorem, let us give an idea of what each term in the above bound represents. In order for $G_{T_\Delta}(\hat z)$ to get past $v \ell$, it must first overcome $\ell$, which explains the first term in the above bound. Then the two terms inside the parenthesis above correspond respectively to the overshooting probability and a large deviations term. We can expect the second term to decay fast as $v$ grows, since $N(\Gamma)$ becomes much smaller than the expected value of $\eta(\Gamma \times [0, \tfrac{1}{2}v\ell\alpha])$.

\begin{proof}
 Define the stopping time (with respect to the filtration $\mathcal{F}_n = \sigma(z_k, \xi_k, k \leq n))$
 \begin{equation}
  \label{e:Tell}
  T_\ell = \inf\{ k \geq 1; G_k({\hat z}) \geq \ell\}.
 \end{equation}
 Now, for any $v \geq 2$, we can bound
$\mathbb{Q}[G_{T_\Delta} ({\hat z}) \geq v \ell]$ by
\begin{equation}
\label{e:Qtwoterms}
   \mathbb{Q} \big[T_\ell < \infty, \;
G_{T_\ell}({\hat z}) \geq \tfrac v2 \ell \big] +
\mathbb{Q} \big[T_\ell <  \infty, \; G_{T_\ell}({\hat z}) < \tfrac v2 \ell, \; G_{T_\Delta}({\hat z}) - G_{T_\ell}({\hat z}) > \tfrac v2 \ell \big]
\end{equation}
(observe that $\mathbb{Q}[G_{T_\Delta}({\hat z})\geq \ell]
=\mathbb{Q}[T_\ell<\infty]$).
 We start by estimating the first term in the above sum, which equals
(using the memoryless property of the exponential distribution)
\begin{align}
 \label{e:Qfirstterm}
 \nonumber
  \smash{\sum_{n \geq 1}} & E^\mathbb{Q} \Big( G_{n-1}({\hat z}) < \ell,
\mathbb{Q} \big[ \xi_n g(z_{n-1},{\hat z})
 > \tfrac v2 \ell - G_{n-1} ({\hat z}) \mid z_{n-1}, G_{n-1} \big] \Big)\\
 \nonumber
  & \leq \sum_{n \geq 1} E^\mathbb{Q} \Big( G_{n-1}({\hat z}) < \ell,
\mathbb{Q}[ \xi_1 g(z_{n-1},{\hat z}) > \ell - G_{n-1}] \, \mathbb{Q}
\big[\xi_1 g(z_{n-1},{\hat z}) > \big( \tfrac v2 -1 \big) \ell \big] \Big)\\
  & \leq \mathbb{Q}[T_{\ell} < \infty]
\sup_{z' \in \Sigma} \mathbb{Q}\big[\xi_1 g(z',{\hat z})
> \big( \tfrac v2 -1 \big) \ell \big]\\
 \nonumber
  & \leq \mathbb{Q}[T_{\ell} < \infty]
\exp \big\{-\big(\tfrac v2 - 1 \big) \big\}.
\end{align}
 We now turn to the bound on the second term
in~\eqref{e:Qtwoterms}, which is
 \begin{equation}
 \label{e:Qsecterm}
  \begin{split}
   E^\mathbb{Q} & \big(T_\ell <  \infty, \;
G_{T_\ell}(\hat z) < \tfrac v2 \ell,
 \mathbb{Q} [G_{T_\Delta}({\hat z}) - G_{T_\ell}({\hat z})
> \tfrac v2 \ell \mid G_1, \dots, G_{T_\ell}]\big)\\
   & \leq \mathbb{Q} \big[T_\ell <  \infty \big] \; \smash{\sup_{z'}}
 \; \mathbb{Q}_{z'} [G_{T_\Delta}({\hat z}) > \tfrac v2 \ell ].
  \end{split}
 \end{equation}
 Now using that for any $z' \in \Sigma$
 \begin{equation}
  G_{T_\Delta}(z') = \sum_{k=1}^{T_\Delta} \xi_k g(z_{k-1},z') \geq \sum_{k=1}^{T_\Delta} \alpha \xi_k g(z_{k-1},{\hat z}) \1{\Gamma}(z') = \alpha G_{T_\Delta}({\hat z}) \1{\Gamma}(z').
 \end{equation}
 we obtain that for all~$z'$
 \begin{equation}
 \label{e:GbyPoisson}
  \begin{split}
   \mathbb{Q}_{z'}\big[G_{T_\Delta}({\hat z}) \geq \tfrac v2 \ell \big]
& \leq \mathbb{Q}_{z'} \Big[ G_{T_\Delta}(z) \geq \frac{1}{2 }v \ell\alpha, \text{ for every $z \in \Gamma$} \Big]\\
   & \leq \mathbb{Q}_{z'} \big[ \eta( \Gamma
\times [0, \tfrac{1}{2} v \ell\alpha] ) \leq N(\Gamma) \big].
  \end{split}
 \end{equation}
 Joining \eqref{e:Qtwoterms} with~\eqref{e:Qfirstterm}, \eqref{e:Qsecterm}
and the above we obtain the desired result.
\end{proof}

Unfortunately, the simulation of a single Markov chain will not suffice for our purposes in this work. As suggested by the definition of random interlacements in terms of a collection of random walks (see \eqref{e:Iu}), we will need to apply the above scheme to construct a sequence of independent Markov chains on $\Sigma$ and for this aim, we will make use of the same Poisson point process $\m$. This is done in Proposition~\ref{p:poisson} below, which requires some further definitions.

Suppose that in some probability space $(\M, \mathcal{L}, \mathcal{P})$ we are given a collection of random elements $(Z^j_k)_{j, k \geq 1}$ of $\Sigma$ such that
\begin{align}
\label{e:transit}
 & \partext{for any given $j \geq 1$, the sequence $(Z^j_1, Z^j_2, \dots)$ is a Markov chain on $\Sigma$, \\ characterized by $\mathcal{P}[Z^j_{k} \in dz\mid Z^j_{k-1}] = g(Z^j_{k-1}, z) \mu(dz)$, for $k = 1, 2, \dots$}\\
\label{e:indMarkov}
 & \text{for distinct values of $j$, the above Markov chains are independent.}
\end{align}
Recall that we interpret \eqref{e:transit} for $k = 1$ as a notation for the starting distribution of the chain as we did in \eqref{e:startlaw}. However, we are allowed to impose different starting laws (for distinct values of $j$) by choosing the $Z^j_0$'s. Although they have a possibly different starting distribution, they all evolve independently and under the same transition laws.

Suppose that for each $j \geq 1$,
\begin{equation}
 \text{the hitting time of~$\Delta$ (as below \eqref{e:Expectcount}) is $\mathcal{P}_{\!Z^j_0}$-a.s. finite,}
\end{equation}
where $\mathcal{P}_{z}$ denotes the law of this Markov chain evolution starting from $z$.

In what follows, we are going to use a single Poisson point process $\m$ to simulate all the above Markov chains $(Z^j_k)$ until they hit~$\Delta$. We do this by simply repeating the procedure of Proposition~\ref{p:xiGi} following the lexicographic order $(j,k) \preccurlyeq (j', k')$ if $j < j'$ or $j = j'$ and $k \leq k'$. This construction results in the accumulation of the soft local times of all the chains, which is essential in proving our main theorem.

In the same spirit of the definition \eqref{e:xisingle}, we set $G^1_0 \equiv 0$ and define inductively, for $n = 1, 2, 3,\dots$
\begin{equation}
 \begin{split}
  & \xi^1_n :=  \inf \big\{ t \geq 0; \text{ $\exists (z_\lambda, v_\lambda)
\notin \{(z^1_k, v^1_k)\}_{k=1}^{n-1}$; $G^1_{n-1}(z_\lambda) + t g(z^1_{n-1}, z_\lambda) \geq v_\lambda$}\big\},\\
  & G^1_n(z) = G^1_{n-1}(z) + \xi^1_{n} \; g(z^1_{n-1}, z), \text{ and}\\
  & (z^1_n, v^1_n) \text{ as the unique pair $(z_\lambda, v_\lambda) \notin \{(z^1_k,v^1_k)\}_{k=1}^{n-1}$ with $G^1_n(z_\lambda) = v_{\lambda}$}.
 \end{split}
\end{equation}

We write $T^1_\Delta$ for the hitting time of $\Delta$ by the chain $(z^1_1, z^1_2, z^1_3,\dots)$. Applying Proposition~\ref{p:xiGi}, we obtain that $(z^1_1, \dots, z^1_{T^1_\Delta})$ is distributed as $(Z^1_1, \dots, Z^1_{T_\Delta})$ under the law~$\mathcal{P}$ and that
\[
\m' := \qquad 
\sum\limits_{\mathclap{(z_\lambda, v_\lambda) \notin \{(z^1_n,v^1_n) \}_{n \leq T^1_\Delta}}}
\qquad \delta_{(z_\lambda, v_\lambda - G^1_{T^1_\Delta}(z_\lambda))}
\]
is distributed as~$\m$ and independent of the above.

Now that we are done simulating the first Markov chain up to time $T^1_{\Delta}$ using $\m$, let us continue the above procedure in order to obtain from $\eta'$ the chain $(Z^2_k)_{k \geq 1}$ and so on. Supposing we have concluded the construction up to $m-1$, then let $G^m_0 \equiv 0$ and define for $n = 1, \dots, T^m_\Delta$ ($T^m_\Delta$ stands for the absorption time of the $m$'th chain),
\begin{equation}
 \label{e:xis}
 \begin{split}
   & \xi^m_n :=  \inf \big\{ t \geq 0; \text{ $\exists (z_\lambda, v_\lambda) \notin \{(z^j_{k}, v^j_k\}_{(j,k) \preccurlyeq (m,n-1)}$;}\\
   & \hspace{4cm} \text{${\textstyle \sum_{j=1}^{m-1}}
     G^{j}_{T^{j}_\Delta}(z_\lambda) + G^m_{n-1}(z_\lambda) + t g(z^m_{n-1}, z_\lambda) \geq v_\lambda$}\big\},\\
   & G^m_n(z) = G^m_{n-1}(z) + \xi^m_{n} \; g(z^m_{n-1}, z), \text{ and}\\
   & \text{$(z^m_n, v^m_n) \notin \{(z^j_k,v^j_k)\}_{(j,k)
       \preccurlyeq (m,n-1)}$ with ${\textstyle \sum_{j=1}^{m-1}}
     G^{j}_{T^{j}_\Delta}(z_\lambda) + G^m_n (z_\lambda) =
     v_{\lambda}$}.
 \end{split}
\end{equation}

The following proposition summarizes the main properties of the above construction and its proof is a straightforward consequence of Proposition~\ref{p:xiGi}.
\begin{proposition}
 \label{p:poisson}
 Suppose we are given starting densities $g(Z^j_0, \cdot)$ ($j \geq 1$) and transition densities $g(\cdot, \cdot)$ of a Markov chain as in \eqref{e:transit}. Then, defining $\xi^j_k$, $G^j_k$ and $z^j_k$ for $j = 1, 2, \dots$ and $k = 1, \dots, T^j_\Delta$, as in \eqref{e:xis} one has
\begin{align}
 \label{e:xi11}
  & (\xi^j_k, {j \geq 1, k \leq T^j_\Delta}) \text{ are i.i.d.\ Exp$(1)$-random variables and}\\
 \label{e:Zzs}
  & (z^j_k, j \geq 1, k \leq T^j_\Delta) \overset{\smash{d}}\sim
(Z^j_k, j \geq 1, k \leq T^j_\Delta) \text{ are independent of $\xi^j_k$'s}.
\end{align}
\end{proposition}

The most relevant conclusion of the proposition is~\eqref{e:Zzs}, showing that our method indeed provides a way to simulate a sequence of independent Markov chains.

\section{Construction of random interlacements from a soup of excursions}
\label{s:altern_constr}
In this section we use Proposition~\ref{p:poisson} to construct
random interlacements in an alternative way. The advantage
of this new construction is that it is more ``local'' than the usual
one, i.e., it does not reveal the interlacement configuration
far away from the set of interest; this of course
facilitates the decoupling of the configuration on different sets,
and that is why we consider this construction to be
the key idea of this paper. Note that the canonical construction
of the random interlacements (presented in Section~\ref{ss:ri})
does not have this property of ``localization'',
since it is quite probable that
many walkers would do \emph{long} excursions away from the
set of interest before eventually coming back.

Let us start with a simple decomposition of random interlacements that prepares the ground for the main construction of this section.

\subsection{Decomposition of random interlacements}
\label{s:decompRI}
A crucial ingredient in proving our main result is a decomposition of the interlacement set $\mathcal{I}^u$ that we now describe. For the rest of this section, let $K$ be a fixed finite subset of $\Z^d$.

Consider first the map $s_K:W^*_K \rightarrow W$ defined by
\begin{equation}
s_K(w^*) \text{ is the unique trajectory $w \in W$ with $\pi^*(w) = w^*$ and $H_K(w) = 0$.}
\end{equation}
We also introduce, for $w \in W$, the one-sided trajectories $w^+ = (X_i(w))_{i \geq 0}$ and $w^- = (X_{-i}(w))_{i \geq 0}$ in $W_+$. These can be seen as the future and past of $w$.

Let us define the space of point measures
\begin{equation}
 \label{e:M}
 \begin{split}
  M = \bigg\{ \;\; \chi = \smash{\sum_{i \in I}} \delta_{(w_i,u_i)} \Big| \; \;
  \begin{split}
   & I \subset \mathbb{N}, w_i \in  W_+, u_i \in \mathbb{R}_+ \text{ and }\\
   & \omega(W_+ \times [0,u]) < \infty \text{ for every } u\geq 0
  \end{split} \;\; \bigg\},
 \end{split}
\end{equation}
endowed with the $\sigma$-algebra $\mathcal{M}$ generated by the evaluation maps $\chi \mapsto \chi(D)$ for $D \in \mathcal{W}_+ \otimes \mathcal{B}(\mathbb{R}_+)$. And for $\chi = \sum_i \delta_{(w_i,u_i)}$ we extend the definition in~\eqref{e:Iu} to~$M$ as follows
\begin{equation}
 \label{e:Iuagain}
 \mathcal{I}^u(\chi) = \bigcup_{i; u_i < u} \Range(w_i).
\end{equation}

We can now introduce, for $\omega = \sum_i \delta_{(w^*_i, u_i)} \in \Omega$, the maps $\chi^+_{K}, \chi^-_{K}: \Omega \to M$ by
\begin{equation}
 \label{e:chipm}
  \chi^+_{K}(\omega) = \sum_{i; w^*_i \in W^*_K} \delta_{(s_K(w^*_i)^+, u_i)} \qquad \text{and} \qquad \chi^-_{K}(\omega) = \sum_{i; w^*_i \in W^*_K} \delta_{(s_K(w^*_i)^-, u_i)} \quad \text{in $M$.}
\end{equation}
We also define the analogous point processes $\chi^+_{K,u}$ and $\chi^-_{K,u}$ where the summations are taken only over $u_i \leq u$.

The main observation concerning these point processes is stated in the following proposition, which is a direct consequence of \eqref{e:QK} and \eqref{e:nuQK}.
\begin{proposition}
 \label{p:chiKu}
 For any finite set $K \subset \Z^d$, the law of $(\chi^+_{K}, \chi^-_{K})$ under $\mathbb{P}$ is a Poisson point process on $(M \times M, \mathcal{M} \otimes \mathcal{M})$ with intensity measure characterized by
 \begin{equation}
  \zeta_{K} \big(A \times [a,b] \times B \times [c,d]\big)
= \Delta\big((a,b)\times(c,d)\big) \sum_{x \in K} e_K(x) \; P_x \big[A \big] \; P_x \big[ B \mid \tilde H_K = \infty \big],
 \end{equation}
 for $A, B \in \mathcal{W}_+$ and $a < b,c < d \in \R$. Where $\Delta$ is the Lebesgue measure at the diagonal in $\R^2$ divided by $\sqrt{2}$.
\end{proposition}


A way to rephrase the above proposition is to say that we can simulate the pair $(\chi^+_{K,u}, \chi^-_{K,u})$ as follows:
\begin{itemize}
 \item Let $\Theta^K_u$ be a Poisson($u \capacity(K)$)-distributed random variable,
 \item choose i.i.d.\ points $X^1_0, \dots, X^{\Theta^K_u}_0$ with law $\bar e_K$ and
 \item from each point $X^j_0$, start two trajectories, with laws given respectively by $P_{X^j_0}$ and $P_{X^j_0}[ \; \cdot \; | \tilde H_K = \infty]$.
\end{itemize}

Given a finite set $K \subset \Z^d$, we are going to decompose the interlacement set $\mathcal{I}^u$ as the union of three sets $\mathcal{I}^u_{K,+}$, $\mathcal{I}^u_{K,-}$ and $\widehat{\mathcal{I}}^u_{K}$ given by
\begin{equation}
 \label{e:Ipm}
 \begin{array}{c}
  \mathcal{I}^u_{K,+} (\omega) = \mathcal{I}^u \big( \chi^+_K(\omega) \big),\\
  \mathcal{I}^u_{K,-} (\omega) = \mathcal{I}^u \big( \chi^-_K(\omega) \big),
\text{ and} \\
  \widehat{\mathcal{I}}^u_{K}(\omega) = \mathcal{I}^u
\big( 1\{W^* \setminus W^*_K\} \cdot \omega \big),
 \end{array}
\end{equation}
recall the definitions \eqref{e:Iu} and \eqref{e:Iuagain}.

Roughly speaking, the sets $\mathcal{I}^u_{K,+}$ and $\mathcal{I}^u_{K,-}$ correspond respectively to the future and past of the trajectories of $\mathcal{I}^u$ that hit~$K$, while $\widehat{\mathcal{I}}^u_K$ encompasses the trajectories not hitting~$K$. This decomposition will be crucial for obtaining the decoupling in Theorem~\ref{t_main} and we now present its main properties.
\begin{proposition}
 \label{p:decompose}
 For any finite $K \subset \Z^d$ and $u \geq 0$,
 \begin{align}
 \label{e:unionthree}
  & \mathcal{I}^u = \mathcal{I}^u_{K,+} \cup \mathcal{I}^u_{K,-} \cup \widehat{\mathcal{I}}^u_K, \text{ for every $\omega \in \Omega$,}\\
  & \text{$\mathcal{I}^u \cap K = \mathcal{I}^u_{K,+}$ $\mathbb{P}$-a.s.}\\
 \label{e:independ}
  & \text{$\widehat{\mathcal{I}}_{K,u}$ is independent of $(\mathcal{I}^u_{K,+}, \mathcal{I}^u_{K,-})$.}
 \end{align}
\end{proposition}

\begin{proof}
To prove \eqref{e:unionthree}, one should decompose the union giving $\mathcal{I}^u$ into $W^*_K$ and $W^* \setminus W^*_K$, observing that for each $w^* \in W^*_K$, $\Range(w^*) = \Range(s_K(w^*)^+) \cup \Range(s_K(w^*)^-)$.

To see why the second statement is true, observe first that $\mathcal{I}^u \cap K \subset \mathcal{I}^u_{K,+} \cup \mathcal{I}^u_{K,-}$, since we have \eqref{e:unionthree} and $\widehat{\mathcal{I}}^u_K$ is disjoint from~$K$. Then, observe that $\mathcal{I}^u_{K,-} \cap K$ is $\mathbb{P}$-a.s. contained in $\mathcal{I}^u_{K,+}$, which follows from Proposition~\ref{p:chiKu}, since for every $x \in \supp(e_K)$, $\Range(w) \cap K = \{X_0(w)\}$, $P_x[\; \cdot\mid \tilde H_K = \infty]$-a.s.

Finally, to prove~\eqref{e:independ}, we observe that these two sets are determined by the outcome of the Poisson point process~$\omega$ into the disjoint spaces of trajectories~$W^*$ and~$W^* \setminus W^*_K$. This finishes the proof of Proposition~\ref{p:decompose}
\end{proof}

We observe also that the random variable
\begin{equation}
 \label{e:ThetaKu}
 \Theta^K_u = \chi^+_{K}(W_+ \times [0,u]) = \chi^-_{K}(W_+ \times [0,u]) \text{ is Poisson($u \capacity(K)$)-distributed.}
\end{equation}



\subsection{Chopping into excursions}
\label{s:chopping}
Fix a finite set $\V \subset \Z^d$ and a set $C \subset \Z^d$ such that
\begin{equation}
 \label{e:dCfinite}
 \partial{C} \text{ is finite.}
\end{equation}
The above condition is equivalent to $C$ being either finite or having finite complement, see Figure~\ref{f:CD} below. Suppose also that $C\cap\V = \varnothing$.
Although some of the definitions that follow will depend on both~$\V$ and~$C$, we will keep only the dependence on~$C$ explicit, since the set~$\V$ will be kept unchanged throughout proofs.

We are interested at first in the trace left by $\mathcal{I}^u_{\V,\pm}$
on the set~$C$. The random walks composing~$\mathcal{I}^u_{\V,+}$ (see~\eqref{e:Ipm}) will perform various excursions between~$C$ and~$\V$ until they finally escape to infinity. This decomposition of a random walk trajectory into excursions is crucial to our proofs and we now give the details of its definition. In fact,
one can look at Figure~\ref{f:CD}, to have a feeling of what is going to happen.

\begin{figure}[ht]
\centering \includegraphics[width = 0.95 \textwidth]{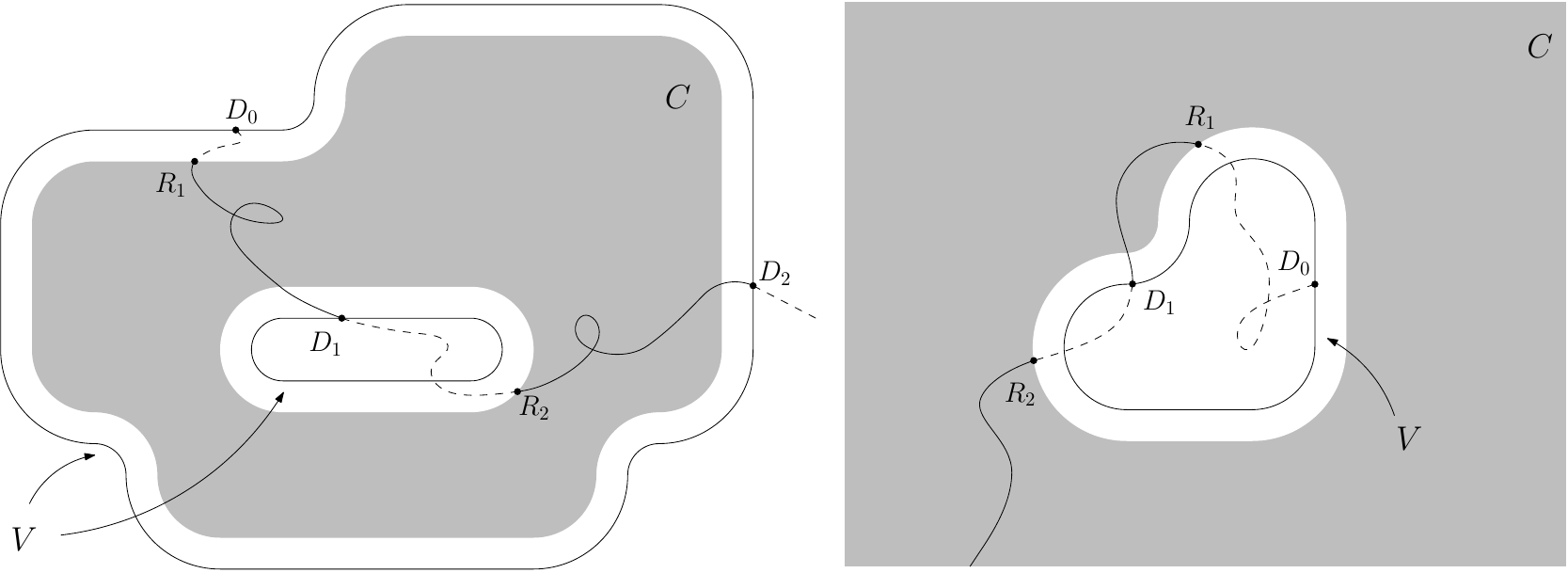}
  \caption{Typical examples of sets $C$ (gray) and $\V$ (closed curves).
On the left $C$ is finite, while on the right it has finite complement.
The stopping times $R_k$ and $D_k$ are also pictured.}
  \label{f:CD}
\end{figure}

Given a trajectory $w_+ \in W_+$ (recall \eqref{e:WW}), let us define its successive return and departure times between $C$ and $\V$:
\begin{align*}
D_0 &= 0, \qquad & & R_1 = H_C,\\
D_1 & = H_{\V} \circ \theta_{R_1} + R_1, \qquad & & R_2 = H_C \circ \theta_{D_1} + D_1,\\
D_2 & = H_{\V} \circ \theta_{R_2} + R_2 \qquad & & \text{and so on, see Figure~\ref{f:CD}.}
\end{align*}
Note that above we have omitted the dependence on $w_+$. Define,
\begin{equation}
 \label{e:tauj}
 T^C = \inf \{k \geq 1; R_k = \infty\},
\end{equation}
which is equal to one plus the random number of excursions performed by $w_+$ until escaping to infinity. Since we assumed the set $\V$ to be finite and the random walk on $\Z^d$ ($d \geq 3$) is transient, $T^C$ is finite $P$-almost surely.

The reason why we define $T^C$ as \emph{one plus} the number of excursions is to guarantee that it coincides with $T_\Delta$ as defined just after \eqref{e:Expectcount} in the construction that follows.


As mentioned before, we are interested in the intersection of $\mathcal{I}^u_{\V,+}$ (recall~\eqref{e:Ipm}) with the set~$C$. Writing $\chi^+_{\V,u} = \sum_{j = 1}^{\Theta^\V_u} \delta_{w_j}$
(where the $w$'s are ordered according to their corresponding~$u$'s),
and abbreviating $T_j^C=T^C(w_j)$, we obtain
\begin{equation}
 \label{e:IuPsi}
 C \cap \mathcal{I}^u_{\V,+} = C \;\; \cap \quad \bigcup_{\mathclap{(w_+,u)
\in \, \supp(\chi^+_{\V,u})}} \;\; \Range(w_+) = C \cap 	\bigcup_{j=1}^{\Theta^\V_u} \bigcup_{k = 1}^{T_j^C-1} \{X_{R_k}(w_j), \dots, X_{D_k}(w_j)\},
\end{equation}
where it may occur that some of the $D_k(w_j)$'s above are infinite.

We are now going to employ the techniques of Section~\ref{s:simul} to simulate the above collection of excursions using a Poisson point process. For this let $\Sigma_C$ denote the following space of paths
\begin{equation}
 \label{e:Sigma}
 \begin{split}
 \Sigma_C = \big\{ \Delta \big\} & \mcup \Big\{
 \begin{array}{c}
  w = (x_1, \dots, x_k) \text{ finite nearest neighbor path, starting}\\
  \text{at $\partial C$ and ending at its first visit to $\V$}
 \end{array}
 \Big\} \\
 & {\textstyle \bigcup} \; \Big\{
 \begin{array}{c}
  w = (x_1, x_2, \dots) \text{ infinite nearest neighbor path,}\\
  \text{starting at $\partial C$ and never visiting $\V$}
 \end{array}
\Big\},
 \end{split}
\end{equation}
where $\Delta$ is a distinguished state that encodes the fact that a given trajectory has already diverged to infinity. Illustrations of finite and infinite paths in~$\Sigma_C$ can be found in Figure~\ref{f:CD}.

Consistently with the previous discussion, we use the
shorthand $X^j_\cdot=X_\cdot^{\vphantom{j}}(w_j)$; in other words, the superscript~$j$
means that we are dealing with the $j$th walk of the construction.
The excursions induced by the random walks will
be encoded as elements of~$\Sigma_C$ as follows
\begin{equation}
 \begin{split}
 & Z^j_k = \big( X_{R_k}^j, \ldots, X^j_{D_k} \big) \in \Sigma_C,
\text{ for $k = 1, \dots, T_j^C-1$, and}\\
 & Z^j_{T_j^C} = \Delta.
 \end{split}
\end{equation}
The reason why we introduce the state $\Delta$ is to recover the description of Section~\ref{s:simul}, indicating that another trajectory is about to start.

In view of \eqref{e:IuPsi}, in order to simulate $C \cap \mathcal{I}^u_{V,+}$, we only need to construct the excursions~$Z^j_k$ with the correct law. For this, we are going to use the construction of the previous section to simulate them from a Poisson point process. In \eqref{e:radon} below, we will prove that for a fixed $j$, the sequence $Z^j_1$, $Z^j_2$, $\dots$ is a Markov chain, as required in~\eqref{e:transit} and~\eqref{e:indMarkov}.

Endow the space of paths $\Sigma_C$ with the $\sigma$-algebra $\mathcal{S}$ generated by the canonical coordinates and with the measure $\mu_C$ given by
\begin{equation}
 \label{e:mu}
 \mu_C(\mathcal{X}) = \sum_{x \in \partial C} P_{x} \big[(X_{0}, X_{1}, \dots, X_{H_\V}) \in \mathcal{X} \big] + \delta_\Delta(\mathcal{X}),
\end{equation}
where $\mathcal{X} \in \mathcal{S}$. Note that $\mu_C$ is finite due to \eqref{e:dCfinite}. We can therefore define a Poisson point process $\m = \sum_i  \delta_{(z_i,v_i)}$ on $\Sigma_C \times \R_+$ with intensity $\mu_C \otimes \d v$ as in~\eqref{e:M1}.


In order to apply Proposition~\ref{p:poisson}, we first observe that for fixed $j \geq 1$, $Z^j_k$ is a Markov chain, due to the Markovian character of the simple random walk. We then define
\begin{equation}
 f^C_y(x) := P_{y} [X_{\tilde H_C} = x]
\end{equation}
and apply the strong Markov property at~$D_{k-1}$, to obtain the Radon-Nikodym derivative
\begin{equation}
 \label{e:radon}
 \frac{\d P[Z^j_{k} \in \cdot \mid Z^j_{k-1} = z]}{\d \mu_C} (z')=
 \begin{cases}
  1, & \text{ if $z = z' = \Delta$,}\\
 & \text{ or $z'=\Delta,D_{k-1}=\infty$,}\\
  f^C_{X^j_{D_{k-1}}(z)} \big( X_0(z') \big), &
\text{ if $z, z' \neq \Delta$, $D_{k-1}<\infty$,}\\
  \smash{P_{X^j_{D_{k-1}}(z)}}[ H_C = \infty ], & \text{ if $z' = \Delta \neq z$, $D_{k-1}<\infty$,}\\
  0, & \text{ otherwise,}
 \end{cases}
\end{equation}
for all $k \geq 2$.

Not only the above shows that the sequence $Z^j_1, Z^j_2, \dots$ is Markovian, but also that the transition density of the chain satisfies
\begin{equation}
 \label{e:gtrans}
 g_C \big( (x_0, \dots, x_l), (y_0, \dots, y_m) \big) = f_{x_l}(y_0)
\end{equation}
($g$ is a density with respect to $\mu_C$, as in~\eqref{e:mu}). We are now left with the starting distributions of the Markov chains $Z^j_k$.

Recall that we are attempting to construct the measure $\chi^+_{\V,u}$, which is not independent of $\chi^-_{\V,u}$. In fact, they are conditionally independent given $\{X_0(w)\}_{w \in \supp(\chi^-_{\V,u})}$. Therefore, conditioning on $\{X^j_0\}_{j = 1, \dots, \Theta^\V_u}$, the starting density of the $j$-th chain (with respect to $\mu_C$) satisfies
\begin{equation}
 g^j_C (x_0, \dots, x_l) = f_{X^j_0}(x_0).
\end{equation}
Finally, we set $Z^j_0 = w$ where $w$ is any trajectory with $X_0(w) = X^j_0 = X^j_{D_0}$, so that \eqref{e:radon} is also satisfied for $k = 1$, in compliance with the notation in \eqref{e:transit} (see also Remark~\ref{rem4.6}).
\begin{figure}
 \centering\includegraphics[width=\textwidth]{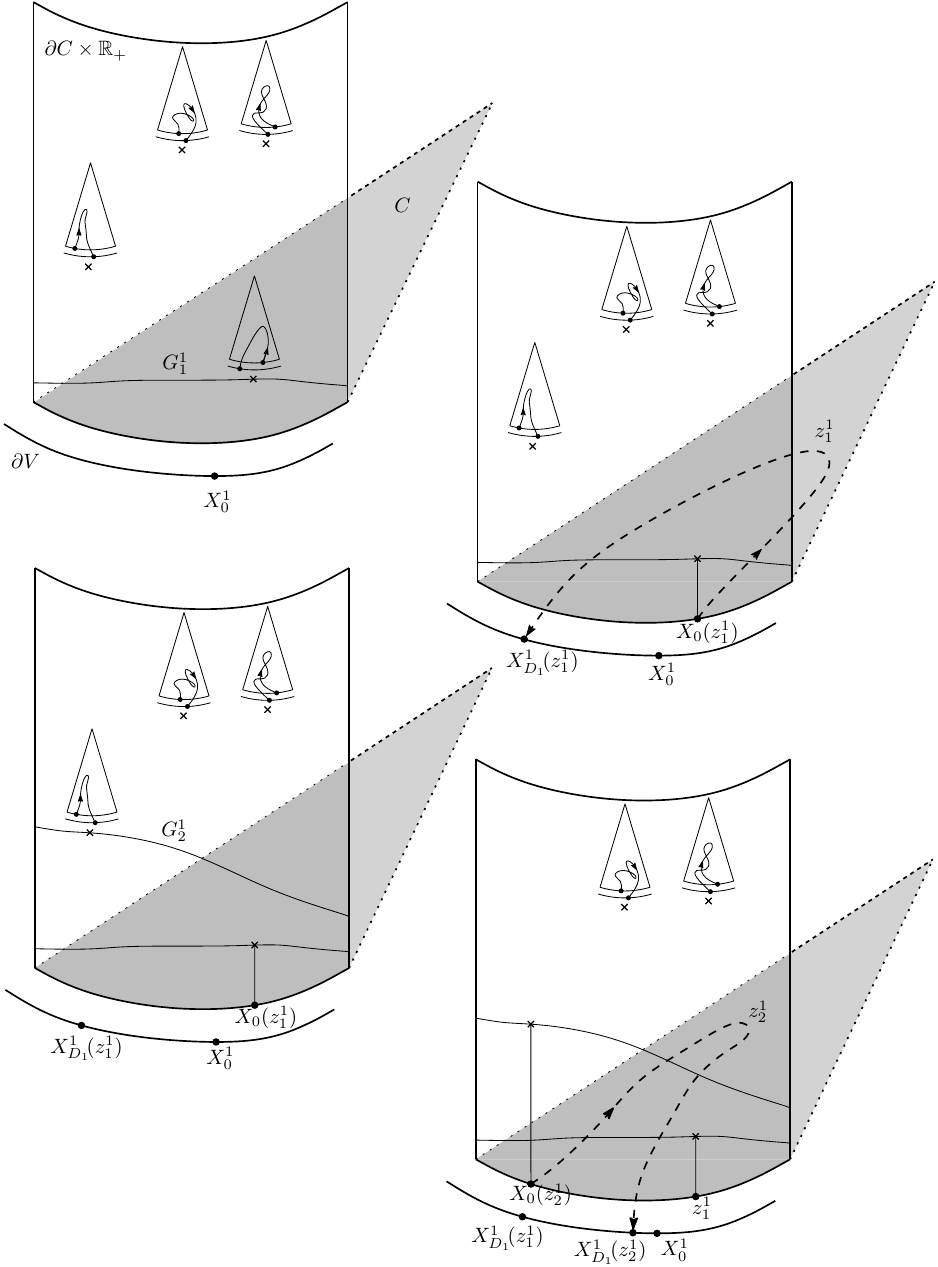}
\caption{On the construction of random interlacements on the set~$C$;
the points of~$\Sigma_C$ are substituted by points
in $\partial C\times\R_+$ with marks representing the corresponding trajectories, and the state~$\Delta$ is not pictured
}
\label{f:constr_trace}
\end{figure}

We can now follow the construction of $\xi^C_{j,k}$ and $G^C_{j,k}$, for $j \geq 1$, $k = 1, \dots, T^C_j$, as in~\eqref{e:xis}. Then, using Proposition~\ref{p:poisson}, we obtain a way to simulate the excursions $Z^j_k$ as promised. In particular, we can show that
\begin{equation}
 \label{e:samelaw}
 C \cap \mathcal{I}^u_{\V,+} \text{ is distributed as $C \cap \bigcup_{j=1}^{\Theta^V_u} \bigcup_{k=1}^{T^C_j} \Range(z^C_{j,k})$ under $\mathbb{Q}$.}
\end{equation}

See on Figure~\ref{f:constr_trace} an illustration of the first two
steps (for the first particle) of the construction of random interlacements on the set~$C$.

We now prove a proposition that relates our main result Theorem~\ref{t_main} with the above construction. To simplify the notation for the soft local time, we abbreviate the accumulated soft local time up to the $\Theta^C_u$-th trajectory
\begin{equation}
 G^{C}_u = G^C_{1,T^C_1} + G^C_{2,T^C_2} + \cdots + G^C_{\Theta^C_u, T^C_{\Theta^C_u}}.
\end{equation}

We can use Theorem~\ref{t:expectsingle} to obtain a short expression
for $E G^C_v(z)$. For this, given $j \geq 1$, we let
\begin{equation}
\label{df_rho}
 \rho^C_j(x) = \sum_{k = 1}^{T_j^C} 1_x(X^j_{R_k})
\end{equation}
count the number of times the $j$-th trajectory starts an excursion through~$x$.

Let us first recall, from~\eqref{e:gtrans}, that~$G^C_v$ depends
on $z=(x_0, x_1, \dots)$ solely through~$x_0$.
Thus, given $z, z' \in \Sigma_C$,
we define $q(z,z') = 1 \{ X_0(z) = X_0(z') \}$ to obtain that
\begin{equation}
\label{e:EGisErho}
\begin{array}{e}
E^{\mathbb{Q}} G^C_{1,T_1^C}(z) & \overset{\eqref{e:gtrans}}=
& \int q(z,z') E^{\mathbb{Q}}  G^C_{1,T_1^C} (z') \, \mu_C(\d z') \\
 & \overset{\mathclap{\text{Theorem~\ref{t:expectsingle}}}}= & \qquad
E^{\mathcal{P}} \Big( \sum_{k=1}^{T^C_j} q(Z^j_k,z)\Big)
 = E^{\mathcal{P}}\rho^C_j(X_0(z)),
\end{array}
\end{equation}
for every $z \in \Sigma_C$.
Clearly, this implies that
\begin{equation}
 \label{e:Elinear}
 E^{\mathbb{Q}} G^C_v(z) = E^{\mathbb{Q}}\Theta^C_v
 \times E^{\mathcal{P}}\rho^C_j(X_0(z)) = v\capacity(\V)
 E^{\mathcal{P}}\rho^C_j(X_0(z)).
\end{equation}

\begin{proposition}
\label{p:couple}
Let $A_1$ and $A_2$ be two disjoint subsets of $\Z^d$ with $A_2$ having finite complement. Now suppose that
\begin{gather}
 \label{e:aPsi}
 \begin{array}{c}
  \text{$\V \subset \Z^d$ is such that any path from $A_1$ to $A_2$ crosses $V$.}
 \end{array}
\end{gather}
Then, for every $u>0$ and $\eps\in(0,1)$ there exists a coupling $\mathbb{Q}$ between $\mathcal{I}^u$ and two independent random interlacements processes, $(\mathcal{I}^u_1)_{u \geq 0}$ and $(\mathcal{I}^u_2)_{u \geq 0}$ such that
\begin{equation}
 \label{e:boundlarge}
 \begin{split}
 \mathbb{Q} \big[ & \mathcal{I}^{u(1-\eps)}_k \cap A_k \; \subseteq \; \mathcal{I}^u \cap A_k \; \subseteq \; \mathcal{I}^{u(1+\eps)}_k, \text{ $k = 1, 2$} \big]\\
 & \;\; \geq 1 - \quad \sum_{\mathclap{\substack{(v,C) = (u(1 \pm \eps), A_1),
\\ (u(1 \pm \eps),A_2), (u,A_1 \cup A_2)}}} \quad  \mathbb{Q} \Big[ \big|G^{C}_v(z)-E^{\mathbb{Q}}G^{C}_v(z)\big| \geq \tfrac{\eps}3 E^{\mathbb{Q}}G^{C}_v(z) \text{ for some $z \in \Sigma_{C}$} \Big],
 \end{split}
\end{equation}
 where the soft local times above are determined in terms of $\V$.
\end{proposition}

We note that the above proposition is an important ingredient for the proof of Theorem~\ref{t_main}, since it relates the success probability of our decoupling with an estimate on the soft local times. In Section~\ref{s:proof}, we will bound the right hand side of \eqref{e:boundlarge} using large deviations. One should not be worried that the set $\Sigma_{C}$ may be uncountable (in case the excursions are infinite). Later we will deal with this using the fact that the soft local time depends on $z$ only through its starting point.

\begin{proof}
We are going to follow the scheme in Section~\ref{s:decompRI} in order to construct the triple $\mathcal{I}^u$, $(\mathcal{I}^u_1)_{u \geq 0}$ and $(\mathcal{I}^u_2)_{u \geq 0}$, distributed as random interlacements on $\Z^d$ as stated in the proposition. However, we will need two independent copies of some of the ingredients appearing in that construction. More precisely,
\begin{align}
\label{e:twochi}
 & \partext{let $\chi^-_{\V,1} = \textstyle \sum_i \delta_{(w^1_i,u^1_i)}$ and
 $\chi^-_{\V,2} = \textstyle \sum_i \delta_{(w^2_i,u^2_i)}$ be two independent  
random variables on $M$ (i.e., Poisson point processes on the space of labelled
trajectories) with the same law as $\chi^-_{\V}$ in \eqref{e:chipm},}\\
\label{e:twoTheta}
 &\partext{let the counting process $\Theta^{V, 1}_u = \chi^-_{\V,1}(W_+ \times [0,u])$ and $\Theta^{V, 2}_u = \chi^-_{\V,2}(W_+ \times [0,u])$ be as in \eqref{e:ThetaKu}, for $u \geq 0$, and finally}\\
\label{e:twohat}
 & \text{define two independent processes $\widehat{\mathcal{I}}^u_{\V,1}$ and $\widehat{\mathcal{I}}^u_{\V,2}$ as in \eqref{e:Ipm}.}
\end{align}

The only missing ingredients in order to construct two independent random interlacements processes following the construction of Section~\ref{s:decompRI} are the random walks composing~$\chi^+_\V$, see~\eqref{e:chipm}. Such construction will be based on Proposition~\ref{p:poisson} and that is where the coupling will take place.

Let us introduce the sets
\begin{equation}
  \begin{array}{c}
    \text{$\Sigma_{A_1 \cup A_2}$, $\Sigma_{A_1}$ and $\Sigma_{A_2}$
      given by~\eqref{e:Sigma} with $\V$ as in \eqref{e:aPsi}.}
  \end{array}
\end{equation}
Note that we have replaced the set $C$ by the three above choices, while keeping $\V$ fixed.

 We also let $\mu_{A_1 \cup A_2}$, $\mu_{A_1}$ and $\mu_{A_2}$ be the
 respective measures on these sets, given by~\eqref{e:mu}.
The first crucial
observation for this proof is the fact that
\begin{equation}
  \label{e:Sigmadisjoint}
  \begin{array}{c}
    \Sigma_{A_1 \cup A_2} \text{ is the disjoint union of $\Sigma_{A_1}$ and $\Sigma_{A_2}$ and $\mu_{A_1 \cup A_2} = \mu_{A_1} + \mu_{A_2}$}.
  \end{array}
\end{equation}
Note that we are duplicating the cemetery on $\Sigma_{A_1 \cup A_2}$, for the above to hold.

We define a Poisson point process $\m$ on $\Sigma_{A_1 \cup A_2} \times \R_+$ with intensity $\mu_{A_1 \cup A_2} \otimes \d v$ as below \eqref{e:mu}. From \eqref{e:Sigmadisjoint} we conclude that,
\begin{equation}
  \label{e:m12}
  \begin{array}{c}
    \text{$\m$ restricted to $\Sigma_{A_1}$ and $\Sigma_{A_2}$ are Poisson point processes with respective}\\
    \text{intensities $\mu_{A_1} \otimes \d v$ and $\mu_{A_2} \otimes \d v$, which are independent of each other.}\\
    \text{Moreover, an excursion $z \in \Sigma_{A_k}$ cannot intersect $A_{k'}$ with $k \neq k'$, see \eqref{e:aPsi}.}
  \end{array}
\end{equation}

We use $\chi^-_{\V,1}$ and $\chi^-_{\V,2}$ in order to define the starting points $\{X^{V,1,j}_0\}_{j = 1, \dots, \Theta^{A_1}_u}$ and $\{X^{V,2,j}_0\}_{j = 1, \dots, \Theta^{A_2}_u}$. Let us finally recall the definitions of $T^C$ from~\eqref{e:tauj}, $G^C_{j,k}$ and $z^C_{j,k}$ from \eqref{e:xis}, where~$C$ can be replaced by either of the three sets $(A_1 \cup A_2)$, $A_1$, or $A_2$. It is important to observe that we use the starting points $X^{V,1,j}_0$ for the case $C = A_1$ and~$X^{V,2,j}_0$
for both $C = A_2$ or $(A_1 \cup A_2)$. We can finally introduce
\begin{equation}
 \mathcal{J}^u_{C} = C \cap \bigcup_{j = 1}^{\Theta^C_u} \bigcup_{k = 1}^{T^C_j} \Range \big( z^C_{j,k} \big), \text{ with $C = (A_1 \cup A_2)$, $A_1$ or $A_2$}
\end{equation}
(note that we use the same Poisson point process to define the three sets above) and
\begin{align*}
 \mathcal{I}^u &= \mathcal{J}^u_{(A_1 \cup A_2)} \cup \mathcal{I}^u(\chi^-_{\V,2}) \cup \widehat{\mathcal{I}}^u_{V,2},\\
 \mathcal{I}^u_1 &= \mathcal{J}^u_{A_1} \cup \mathcal{I}^u(\chi^-_{\V,1}) \cup \widehat{\mathcal{I}}^u_{V,1}, \text{ and}\\
 \mathcal{I}^u_2 &= \mathcal{J}^u_{A_2} \cup \mathcal{I}^u(\chi^-_{\V,2}) \cup \widehat{\mathcal{I}}^u_{V,2}.
\end{align*}
We independently modify the above sets on $(A_1 \cup A_2)^c$ to obtain the correct distributions, although this is immaterial for the statement of the proposition.

To conclude the proof of the proposition, let us observe that
\begin{itemize}
 \item $(\mathcal{J}^u_{C})_{u \geq 0}$ is distributed as $(C \cap \mathcal{I}^u_{\V,+})_{u \geq 0}$, for $C = (A_1 \cup A_2)$, $A_1$ or $A_2$, see \eqref{e:samelaw}, so that $((A_1 \cup A_2) \cap \mathcal{I}^u)_{u \geq 0}$, $(A_1 \cap \mathcal{I}^u_1)_{u \geq 0}$ and $(A_2 \cap \mathcal{I}^u_2)_{u \geq 0}$ have the right distributions as under the random interlacements;
 \item $\mathcal{J}^u_{A_1}$ and $\mathcal{J}^u_{A_2}$ are independent, see \eqref{e:twochi}, \eqref{e:twoTheta} and \eqref{e:m12},
which means that $(A_1 \cap \mathcal{I}^u_1)_{u \geq 0}$ and $(A_2 \cap \mathcal{I}^u_2)_{u \geq 0}$ are also independent.
\end{itemize}

So that, using the definition of $\mathcal{I}^u$, $\mathcal{I}^u_1$ and $\mathcal{I}^u_2$,
\begin{align}
  \nonumber
  \mathbb{Q} \big[ \mathcal{I}^{u(1-\eps)}_k & \cap A_k \; \subseteq \; \mathcal{I}^u \cap A_k \; \subseteq \mathcal{I}^{u(1+\eps)}_k, \text{ $k = 1, 2$} \big]\\
  & \geq \mathbb{Q} \big[ \mathcal{J}^{u(1-\eps)}_{A_k} \subseteq \; \mathcal{J}^u_{A_1 \cup A_2} \cap A_k \; \subseteq \mathcal{J}^{u(1+\eps)}_{A_k}, \text{ $k = 1, 2$} \big]\\
  \nonumber & \geq \mathbb{Q} \big[ G^{A_k}_{u(1-\eps)}(z) \leq G^{A_1
    \cup A_2}_u (z) \leq G^{A_k}_{u(1+\eps)}(z), \text{ for all $z \in
    \Sigma_{A_k}$ and $k = 1,2$} \big].
\end{align}
Now, \eqref{e:aPsi} implies that for $x\in\partial A_k$ we have
$\varphi(x):=E^{\mathcal{P}}\rho^{A_k}_1(x)=E^{\mathcal{P}}\rho^{A_1\cup A_2}_1(x)$.
The conclusion of \eqref{e:boundlarge} is now a simple consequence of the above display and the fact that the expectation of $G^{C}_u$ is linear in~$u$ according to~\eqref{e:Elinear}
(see Figure~\ref{f_2baldes}).
\begin{figure}
 \centering \includegraphics[width=\textwidth]{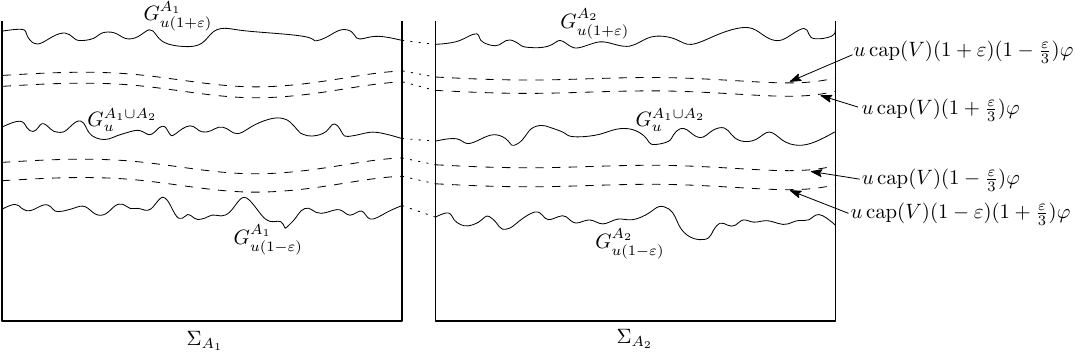}
  \caption{On the proof of Proposition~\ref{p:couple}, $\varphi$ was defined in the last paragraph of the proof
 (observe that $1+\frac{\eps}{3}\leq (1-\frac{\eps}{3})(1+\eps)$ for $\eps\in [0,1]$)}
  \label{f_2baldes}
\end{figure}
\end{proof}

\section{Proof of Theorem~\ref{t_main}}
\label{s:proof}

In this section we will prove our main result, modulo a set of
additional assumptions that will be proved in the next section.

Recall that we use the notation
$\bally(x,r)=\{y\in\Z^d: \|x-y\|\leq r\}$ for discrete balls. Also,
for $A\subset \Z^d$ we write $\bally(A,r)=\bigcup_{x\in A}\bally(x,r)$.

Suppose we are given sets~$A_1$ and~$A_2$
as in Theorem~\ref{t_main} and suppose
without loss of generality that the diameter of~$A_1$ is
not greater than the diameter of~$A_2$.
It is clear that we can assume
that $A_2=\Z^d \setminus \bally(A_1, s)$, since the function~$f_2$
can be seen as a
function in $\{0,1\}^{\Z^d \setminus \bally(A_1, s)}$;
so, from now on we work with this assumption.

The proof of the main theorem will require some estimates on the entrance distribution of a random walk on the sets $A_1$, $A_2$ and $A_1\cup A_2$, which are closely related to the regularity conditions mentioned above Theorem~\ref{t:expmoment}.
However, the problem is that, in general, these estimates need not
be satisfied for \emph{arbitrary} finite set~$A_1$ and
$A_2=\Z^d \setminus \bally(A_1, s)$. So, in order to fix
this problem, we will replace~$A_1$ and $A_2$
by slightly larger sets $A^{(s)}_1$ and $A^{(s)}_2$, using Proposition~\ref{p:fatA} below.
Roughly speaking, these ``fattened'' sets
will have the following properties
(below, $C$ stands for any of the three sets
$A^{(s)}_1, A^{(s)}_2,$ or $A^{(s)}_1\cup A^{(s)}_2$):
\begin{itemize}
 \item the probability that the simple random walk
enters~$C$ through some point~$y$ is at most~$O(s^{-(d-1)})$,
for starting points at distance at least of order~$s$ from~$C$;
 \item this probability should be at least of order~$s^{-(d-1)}$
for ``many'' starting points which are at distance of order~$s$ from~$y$;
 \item the probabilities of entering~$C$ through two near
points~$y$ and~$y'$ in~$\partial C$ can be different by at most a (fixed)
constant factor
(this should be valid as soon as the random walk starts far from $\{y,y'\}$);
 \item finally, we also need some additional geometric properties of $\partial C$.
\end{itemize}
A typical example of a set having these properties is a
discrete ball of radius~$s$;
in fact, we will prove that any set with
``sufficiently smooth boundary'' will do.
More rigorously, the fact that we need is formulated in the following way
(one may find helpful to look at Figure~\ref{f_smoothing}):
\newiconst{regular}
\newiconst{reg2}
\newiconst{reg8}
\newiconst{reg9}
\newiconst{reg10}
\newiconst{reg11}
\begin{proposition}
 \label{p:fatA}
 There exist positive constants
 $\useiconst{regular}\in (0,\frac{1}{10})$, $\useiconst{reg2}$,
$\useiconst{reg8}<\frac{\useiconst{regular}}{2}$,
$\useiconst{reg9},\useiconst{reg10}, \useiconst{reg11}\in (0,1),s_0$
(depending only on dimension)
such that, for any $s \geq s_0$ and any set $A\subset\Z^d$
such that $\Z^d\setminus \bally(A,s)$ is nonempty,
there is a set $A^{(s)}$ with the following properties:
 \begin{gather}
  \label{hip0}
 A \subseteq A^{(s)} \subseteq \bally\Big(A, \frac{s}{5}\Big);
\intertext{for any $y \in \partial A^{(s)}$,}
  \label{hip2}
  \sup_{\substack{x\in \Z^d;\\ \dist(x,y) \geq \useiconst{regular}s/2}}
P_x [X_{H_{A^{(s)}}} = y] \leq \useiconst{reg2} s^{-(d-1)}
\intertext{and there exists a ball $\bar B_y$ of
radius $\useiconst{regular} s$ such that $\dist(\bar B_y, y)
\in [\useiconst{regular}s, 2 \useiconst{regular}s]$
and}
  \label{hip3}
  \inf_{x\in \bar B_y} P_x [X_{H_{A^{(s)}}} = y,
H_{A^{(s)}}<H_{\Z^d\setminus
\bally(y,4\useiconst{regular} s)}]
 \geq \useiconst{reg2}^{-1} s^{-(d-1)}.
\intertext{Moreover, for any $y \in \partial A^{(s)}$}
\label{hip4}
|\{z\in \partial A^{(s)}: \|y - z\|
  \leq \useiconst{reg8}s\}| \geq \useiconst{reg9}s^{d-1}
\intertext{and if $y' \in \partial A^{(s)}$
is such that $\|y - y'\| \leq \useiconst{reg8}s$, then
there exists a set $\hat D$ (depending on $y,y'$) that separates $\{y,y'\}$ from
$\partial \bally(y,\useiconst{regular}s)$ (i.e., any nearest-neighbor path starting at
$\partial \bally(y,\useiconst{regular}s)$ that enters $A^{(s)}$ at $\{y,y'\}$, must pass through $\hat D$)
such that}
  \label{hip1}
  \sup_{\substack{x\in\hat D;\\ P_x [X_{H_{\!A^{(s)}}} = y']>0}}
\frac{P_x [X_{H_{A^{(s)}}} = y]}
{P_x [X_{H_{A^{(s)}}} = y', H_{A^{(s)}}<H_{\Z^d\setminus
\bally(y',5\useiconst{regular} s)}]}
\leq \useiconst{reg10}.
\end{gather}
\end{proposition}

\begin{figure}
\centering \includegraphics{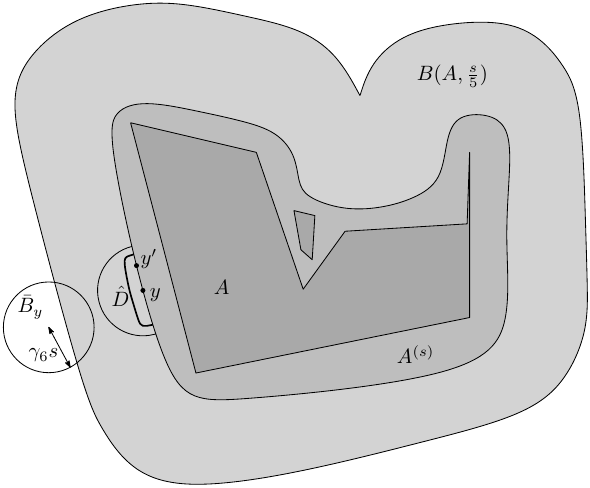}
  \caption{On the sets in Proposition~\ref{p:fatA}}
  \label{f_smoothing}
\end{figure}

The proof of this proposition is postponed to Section~\ref{s:reduction}.
We now are going to use the above result to prove Theorem~\ref{t_main}.

Recall that we define $A_2=\Z^d \setminus \bally(A_1, s)$.
The idea is to use Proposition~\ref{p:couple} for $A_1^{(s)}$
and~$A_2^{(s)}$ provided by Proposition~\ref{p:fatA},
and~$\V$ defined as
\begin{equation}
\label{df_separator}
 \V = \big\{y\in \Z^d : \dist(y,A_1^{(s)}\cup A_2^{(s)})\geq
 \useiconst{regular}s\big\}.
\end{equation}

Let $y,y'\in \partial A_1^{(s)} \cup \partial A_2^{(s)}$ be such that
$\|y - y'\| \leq \useiconst{reg8}s$ (in fact, in this case both~$y$ and~$y'$ must be in the same
set, either $\partial A_1^{(s)}$ or $\partial A_2^{(s)}$).
Let~$\hat D$ be the corresponding separating set, as in~\eqref{hip1} of
Proposition~\ref{p:fatA}.
Now, consider an arbitrary site $x \in \V$,
and write for $C=A_1^{(s)}, A_2^{(s)}, A_1^{(s)}\cup A_2^{(s)}$
\begin{equation}
 \label{decomp_hat_G}
 P_x[X_{H_C}=y]
 = \sum_{z\in \hat D} P_x[X_{H_{C\cup \hat D}}=z] P_z[X_{H_C}=y] \text{ and similarly with $y'$},
\end{equation}
where we used the strong Markov property
at $H_{C\cup \hat D}$ and dropped vanishing terms.
So, by construction, we have
\begin{equation}
\label{hip1'}
\sup_{\substack{x\in\V;\\ P_x [X_{H_C} = y']>0}}
 \frac{P_x [X_{H_C} = y]}{P_x [X_{H_C} = y']} \leq \useiconst{reg10}
\quad \text{ for } C=A_1^{(s)}, A_2^{(s)}, A_1^{(s)}\cup A_2^{(s)},
\text{ when }\|y - y'\| \leq \useiconst{reg8}s.
\end{equation}

With the above, we can now start estimating
the soft local times appearing in~\eqref{e:boundlarge}.

In the rest of this section, $C$ stands for one of the sets
$A_1^{(s)}, A_2^{(s)}, A_1^{(s)}\cup A_2^{(s)}$; we will
obtain the same estimates for all of them. Recalling the definition of $T_\ell$ in \eqref{e:Tell}, we consider $x\in \partial C$
and fix any~$z\in\Sigma_C$ such that $x=X_0(z)$; then we denote
\begin{equation}
\label{df_W}
 F^C_j(x) = G^C_{j,T^C_j}(z)
\end{equation}
to be the contribution of the $j$-th particle to
the soft local time in trajectories starting at~$x$,
in the construction of the corresponding interlacement set for~$C$,
so that $G^C_u(z)=\sum_{j=1}^{\Theta^C_u}F^C_j(x)$.

We also introduce
\begin{equation}
 \pi^C(x) = E [F^C_1(x)], \text{ which also equals
$E^{\mathcal{P}} \rho_1^C(x)$ due to \eqref{e:EGisErho},}
\end{equation}
recall the definition of $\rho^C_j$ from \eqref{df_rho}.


\newiconst{c:L41'}
\newiconst{c:L41''}
\newiconst{c:L41'''}
\begin{lemma}
\label{l_expect_Y}
For $C$ being either $A_1^{(s)}$, $A_2^{(s)}$ or $A_1^{(s)} \cup A_2^{(s)}$
and~$\V$ as in~\eqref{df_separator}, we have for all~$x\in \partial C$
\begin{itemize}
 \item[(i)] $\useiconst{c:L41'}
     s^{-1}\capacity(\V)^{-1} \leq \pi^C(x)
        \leq \useiconst{c:L41''} s^{-1}\capacity(\V)^{-1}
 \vphantom{\displaystyle\int}$;
 \item[(ii)] $E(F_1^C(x))^2\leq \useiconst{c:L41'''}
           s^{-d}\capacity(\V)^{-1}  \vphantom{\displaystyle\sum}$.
\end{itemize}
\end{lemma}

\begin{proof}
Instead of estimating the expected soft local time directly,
we rather work with the
``real'' local time~$\rho_1^C(x)$, with the assistance of
Theorem~\ref{t:expectsingle}.

Consider the discrete sphere~$\tilde\V$ of radius~$3(r+s)$ centered in any fixed point of~$A_1$. Given a trajectory $w^* \in W^*$, the number of excursions $\rho_1^C(x)$ between~$\V$ and~$C$ entering at~$x$ is the same for both $s_\V(w^*)$ and $s_{\tilde\V}(w^*)$. Thus, their expected values are the same and can be written respectively as $u\capacity(\V)\pi^C(x)$ and $u\capacity(\tilde\V){\tilde\pi}^C(x)$, where ${\tilde\pi}^C(x)$ is the expected number of such $(V,x)$-crossings under $P_{\bar e_{\tilde \V}}$.
So,
\[
  \pi^C(x) = \capacity(\V)^{-1}\capacity(\tilde\V){\tilde\pi}^C(x).
\]
We know that $\capacity(\tilde\V) \asymp (r+s)^{d-2}$
(see~\eqref{e:estimatecap}), so, in order to prove the part~(i), it will be enough to obtain
 that
\begin{equation}
\label{e:piasymp}
 {\tilde\pi}^C(x) \asymp s^{-1}(r+s)^{-(d-2)}.
\end{equation}

For $x' \in \Z^d \setminus C$ such that $d(x',x) \geq \useiconst{regular} s$,
we use the Markov property at~$H_{C}$ to obtain
\begin{equation}
 E_{x'} \rho_1^C(x) \leq P_{x'}[X_{H_{C}} = x] + \sup_{y \in \V}
P_y[H_{B(x, \useiconst{regular} s/2)} < \infty] \sup_{z: d(z,x) \geq \useiconst{regular} s/2} E_{z} \rho_1^C(x).
\end{equation}
Then taking the supremum in $x'$ and using~\eqref{hip2}, we get that
\newconstant{c:p24}\newconstant{c:p24'}\newconstant{c:p24''}
\begin{equation}
\label{e:suppp}
 \sup_{x': d(x',x) \geq \useiconst{regular} s/2} E_{x'} \rho_1^C(x)
\leq \frac{\sup_{x': d(x',x) \geq \useiconst{regular} s/2}P_{x'}[X_{H_C} = x]}
{1-\sup_{y \in \V} P_y[H_{B(x, \useiconst{regular} s/2)} < \infty]} \leq \useconstant{c:p24} s^{-(d-1)}.
\end{equation}
So, by Proposition~6.4.2 of~\cite{LL10},
\begin{align*}
  E_{\bar e_{\tilde \V}}\rho_1^C(x) &\leq \sup_{x'\in{\tilde \V}}
 P_{x'}[H_{B(x, \useiconst{regular} s/2)}<\infty]
\sup_{x': d(x',x) \geq \useiconst{regular} s/2} E_{x'} \rho_1^C(x)\\
 &\leq \useconstant{c:p24'} \Big(\frac{s}{s+r}\Big)^{-(d-2)} s^{-(d-1)}
 = \useconstant{c:p24'} s^{-1}(r+s)^{-(d-2)}.
\end{align*}
We are now left with the lower bound
\[
E_{\bar e_{\tilde \V}}\rho_1^C(x) \geq \inf_{x' \in \partial \tilde \V} P_{x'}[H_{\bar B_x} < \infty] \inf_{x'' \in \partial \bar B_x} P_{x''}[X_{H_C} = x] \overset{\eqref{hip3}}\geq \useconstant{c:p24''}
\Big(\frac{s}{s+r}\Big)^{-(d-2)} s^{-(d-1)},
\]
proving~\eqref{e:piasymp} and consequently (i).

The part~(ii) then immediately follows from~\eqref{e:suppp}
and Theorem~\ref{t:2ndmoment} (see also Remark~\ref{rem4.6}).
\end{proof}

Next, we need the following large deviation bound for $F_1^C(x)$:
\newiconst{c:L42'}
\newiconst{c:L42''}
\begin{lemma}
\label{l_LD_Y}
For $C=A_1^{(s)}, A_2^{(s)},A_1^{(s)} \cup A_2^{(s)}$
and~$\V$ as in~\eqref{df_separator}, we have for all~$x\in \partial C$
\begin{equation}
\label{eq_LD_Y}
P[F_1^C(x) > v\useiconst{reg2}s^{-(d-1)}] \leq
\useiconst{c:L42'} s^{d-2}\capacity(\V)^{-1}
 \exp(-\useiconst{c:L42''}v), \text{ for any $v \geq 2$}
\end{equation}
(also, without loss of generality we suppose that $\useiconst{c:L42''}\leq 1$).
\end{lemma}

\begin{proof}
The idea is to apply Theorem~\ref{t:expmoment} for $F_1^C(x)$
and with $\Gamma_x=\{z\in \Sigma_C: \|x - X_0(z)\|
  \leq \useiconst{reg8}s\}$; observe that
$\mu_C(\Gamma_x)\geq \useiconst{reg9}s^{d-1}$ by~\eqref{hip4}.
With the notation of Theorem~\ref{t:expmoment}, we set
\[
 \ell = \useiconst{reg2} s^{-(d-1)}\quad \text { and observe that } \quad
 \alpha \geq \frac{1}{\useiconst{reg10}},
\]
by~\eqref{hip2} and \eqref{hip1'}.

Chebyshev's inequality together with Lemma~\ref{l_expect_Y}~(i)
then imply that
\begin{equation}
\label{chegar_perto}
 P[T_\ell<\infty] = P[F_1^C(x)\geq \useiconst{reg2} s^{-(d-1)}]
 \leq \frac{\pi^C(x)}{\useiconst{reg2}s^{-(d-1)}}
 \leq \useiconst{reg2}^{-1}\useiconst{c:L41''} s^{d-2}\capacity(\V)^{-1}.
\end{equation}

Now, denoting by $N(\Gamma_x)$ the number of crossings between~$\V$
and~$C$ that enter in~$\Gamma_x$ and by~$\eta_x$ the number
of points of the Poisson process (from the construction
in Section~\ref{s:altern_constr})
in $\Gamma_x\times
\big[0,\frac{\useiconst{reg2}}
{2\useiconst{reg10}}vs^{-(d-1)}\big]$, we
write
\[
 \mathbb{Q}_{z'} \big[ \eta_x \leq N(\Gamma_x) \big]
\leq \mathbb{Q}_{z'}\Big[\eta_x\leq \frac{\useiconst{reg2}\useiconst{reg9}}
{4\useiconst{reg10}}v\Big]
+
\mathbb{Q}_{z'}\Big[N(\Gamma_x)\geq\frac{\useiconst{reg2}\useiconst{reg9}}
{4\useiconst{reg10}}v\Big].
\]
To see that both terms in the right-hand side
of the above display are exponentially small in~$v$,
we observe that
\newconstant{c:unif}
\begin{itemize}
 \item $\eta_x$ has Poisson distribution
with parameter at least $\frac{\useiconst{reg2}\useiconst{reg9}}
{2\useiconst{reg10}}v$, and
 \item starting from any~$y\in\V$, with uniformly
positive probability the random walk does not enter in~$\Gamma_x$ (recall that
$\useiconst{reg8}<\frac{\useiconst{regular}}{2}$, which implies that
$P_y[H_{\Gamma_x}<\infty]<\useconstant{c:unif}<1$ uniformly in
$y\in\V$).
Therefore~$N_x$ is dominated by a Geometric($\useconstant{c:unif}$)
random variable having exponential tail as well.
\end{itemize}
Together with~\eqref{chegar_perto} and Theorem~\ref{t:expmoment}, this finishes the proof of Lemma~\ref{l_LD_Y}.
\end{proof}

Now, we are able to finish the proof of our main result.
\begin{proof}[Proof of Theorem~\ref{t_main}.]
For $C=A_1,A_2,A_1\cup A_2$ and $x\in\partial C$,
let $\psi^x_C(\lambda)=Ee^{\lambda F_1^C(x)}$
be the moment generating function of $F_1^C(x)$.
It is elementary to obtain that
$e^t-1\leq t+t^2$ for all $t\in [0,1]$. Using this
observation, we write for
$0\leq \lambda\leq \frac{1}{2}\useiconst{reg2}^{-1}\useiconst{c:L42''}s^{d-1}$
(where $\useiconst{c:L42''}$ is from Lemma~\ref{l_LD_Y})
\newconstant{c:56'}
\newconstant{c:56''}
\newconstant{c:56'''}
\begin{align}
 \psi^x_C(\lambda) & - 1 = E(e^{\lambda F_1^C(x)}-1)
    \1{\lambda F_1^C(x)\leq 1} + E(e^{\lambda F_1^C(x)}-1)
    \1{\lambda F_1^C(x) > 1} \nonumber\\
 &\leq E\big(\lambda F_1^C(x)+\lambda^2 (F_1^C(x))^2\big)
     +Ee^{\lambda F_1^C(x)}\1{F_1^C(x) > \lambda^{-1}}
        \nonumber\\
 &\leq \lambda\pi^C(x) + \useiconst{c:L41'''} \lambda^2 s^{-d}
   \capacity(\V)^{-1} 
  + \lambda\int\limits_{\lambda^{-1}}^\infty e^{\lambda y}
      P[F_1^C(x)>y]\, dy \nonumber\\
 &\leq \lambda\pi^C(x) + \useiconst{c:L41'''} \lambda^2 s^{-d}
  \capacity(\V)^{-1} 
  + \lambda\useiconst{c:L42'}
s^{d-2}\capacity(\V)^{-1} 
\int\limits_{\lambda^{-1}}^\infty
 \exp\Big(-\frac{\useiconst{c:L42''}}{2\useiconst{reg2}}s^{d-1}y\Big)
      \, dy \nonumber\\
 &\leq \lambda\pi^C(x) + \useiconst{c:L41'''} \lambda^2 s^{-d}
 \capacity(\V)^{-1} 
 + \useconstant{c:56'}
 s^{-1}\capacity(\V)^{-1} 
\lambda \exp(-\useconstant{c:56''} \lambda^{-1}s^{d-1})
     \nonumber\\
 &\leq \lambda\pi^C(x) + \useconstant{c:56'''} \lambda^2 s^{-d}
 \capacity(\V)^{-1}, 
\label{est_psi(>0)}
\end{align}
where we used Lemma~\ref{l_expect_Y}~(ii) and Lemma~\ref{l_LD_Y}.
Analogously, since $e^{-t}-1\leq -t+t^2$ for \emph{all} $t>0$,
we obtain for $\lambda\geq 0$
\newconstant{c:57}
\begin{equation}
\label{est_psi(<0)}
\psi^x_C(-\lambda) - 1 \leq -\lambda\pi^C(x)
          + \useconstant{c:57} \lambda^2 s^{-d}
 \capacity(\V)^{-1} 
\end{equation}
(in this case we do not need the large deviation bound
of Lemma~\ref{l_LD_Y}).

Observe that, if $(Y_k, k\geq 1)$ are i.i.d.\ random variables with common
moment generating function~$\psi$ and~$N$ is an independent Poisson random
variable with parameter~$\theta$, then
$E\exp\big(\lambda\sum_{k=1}^N Y_k\big)
     =\exp\big(\theta(\psi(\lambda)-1)\big)$.
So, using \eqref{est_psi(>0)} and Lemma~\ref{l_expect_Y}~(ii),
we write for any $\delta > 0$, $z \in \Sigma$ and $x = X_0(z)$,
\newconstant{c:888}
\newconstant{c:999}
\begin{align*}
 \mathbb{Q}\big[G^C_{\hat u}
\geq (1+\delta) &
 {\hat u} \capacity(\V)
\pi^C(x) \big]
 =\mathbb{Q}\vvviiiggg[\sum_{k=1}^{\Theta_{\hat u}^C} F_k^C(x) \geq
    (1+\delta)
 {\hat u} \capacity(\V) 
\pi^C(x) \vvviiiggg]\\
 &\leq
  \frac{E\exp\vvviiiggg(\lambda\sum_{k=1}^{\Theta^C_{\hat u}}
 F_k^C(x) \vvviiiggg)}
   {\exp\big(\lambda (1+\delta)
{\hat u} \capacity(\V) 
\pi^C(x)\big)}\\
&= \exp\big(-\lambda(1+\delta)
{\hat u} \capacity(\V)   
\pi^C(x)
  + {\hat u}
 \capacity(\V) 
(\psi(\lambda)-1)\big)\\
&\leq \exp\vvviiiggg(-\big(\lambda\delta
{\hat u} \capacity(\V) 
\pi^C(x)
     -\useconstant{c:56'''}\lambda^2
{\hat u} s^{-d}\big)\vvviiiggg)\\
 &\leq
\exp\vvviiiggg(-\big(\useconstant{c:888}\lambda\delta {\hat u}s^{-1}
-\useconstant{c:56'''}\lambda^2{\hat u}s^{-d}\big)\vvviiiggg),
\end{align*}
and, analogously, with~\eqref{est_psi(<0)} instead of~\eqref{est_psi(>0)}
one can obtain
\newconstant{c:888'}
\newconstant{c:999'}
\[
 \mathbb{Q}\vvviiiggg[
G^C_{\hat u}
\leq
    (1-\delta) {\hat u} \capacity(\V) 
\pi^C(x) \vvviiiggg]
    \leq \exp\vvviiiggg(-\big(\useconstant{c:888'}\lambda\delta
{\hat u}s^{-1}-\useconstant{c:999'}\lambda^2{\hat u}s^{-d}\big)\vvviiiggg).
\]
\newconstant{c:577}
Choosing $\lambda=\useconstant{c:577}\delta s^{d-1}$,
with small enough~$\useconstant{c:577}$ depending on
$\useconstant{c:56'''},\useconstant{c:888},
\useconstant{c:888'},\useconstant{c:999'}$
(and such that $\useconstant{c:577}\leq
(\delta^{-1}\frac{\useiconst{c:L42''}}{2})\wedge\frac{\useiconst{c:L42''}}
{3\useiconst{reg2}} $), we thus
obtain using also the union bound (clearly, the cardinality of
$\partial C$ is at most $O((r+s)^d)$)
\newconstant{c:i58'}
\newconstant{c:i58''}
\begin{align}
\lefteqn{
 \mathbb{Q}\vvviiiggg[ (1-\delta){\hat u} \capacity(\V)\pi^C(x)\leq
G^C_{\hat u}
\leq (1+\delta){\hat u} \capacity(\V)\pi^C(x)\text{ for all }x\in\partial C\vvviiiggg]}
\phantom{****************************}\nonumber\\
  &\geq 1 -
\useconstant{c:i58'}(r+s)^d
 \exp\big(-\useconstant{c:i58''}\delta^2{\hat u} s^{d-2}\big).
\label{main_LD_bound}
\end{align}
 Using~\eqref{main_LD_bound} with $\delta=\frac{\eps}{3}$
and $u,(1-\eps)u, (1+\eps)u$ on the place of~$\hat u$
together with Proposition~\ref{p:couple},
we conclude the proof of Theorem~\ref{t_main}.
\end{proof}

\begin{remark}
\label{rem_boundary_size}
As mentioned in the introduction, the factor $(r+s)^d$ before
the exponential in~\eqref{eq_mainmain} can usually be reduced.
Let us observe that this factor
(times a constant) appears in the proof as an upper
bound for the cardinality of $\partial(A_1^{(s)}\cup A_2^{(s)})$.
In the typical situation when~$s$ is smaller than~$r$ and the
sets have a sufficiently regular boundary (e.g., boxes or balls),
one can substitute $(r+s)^d$ by $r^{d-1}$.
\end{remark}

\section{Connectivity decay}
\label{s:connectivity}

\begin{proof}[Proof of Theorem~\ref{t:connect}]
 We start by introducing the renormalization scheme in which the proof will be based.
 Fix $b\in (1,2]$; clearly, one can consider only this
range of the parameter~$b$ for proving~\eqref{e:connect_d3},
and any particular value of~$b\in (1,2]$ (in fact, any $b\in (0,\infty)$) will work for
proving~\eqref{e:connect_d4}.
Given $L_1 \geq 100$, we define the sequence
 \begin{equation}
  L_{k+1} =  2 \Big( 1 + \frac{1}{(k+5)^b} \Big) L_k, \text{ for $k \geq 1$}.
 \end{equation}
 Note that $L_k$ grows roughly as~$2^k$ and it need not be an integer in general. Before moving further, let us first establish some important properties on the rate of growth of this sequence. First, it is obvious that
 \begin{equation}
  \label{e:Lkbounds}
  2 L_k = L_{k+1} - \frac{2}{(k+5)^b} L_k \leq L_{k+1} - \frac{2^{k}L_1}{(k+5)^b} \leq \lfloor L_{k+1} \rfloor - \frac{2^{k-1}L_1}{(k+5)^b}.
 \end{equation}
 for all $k \geq 1$
(here we used that $\frac{50 \cdot 2^{k+1}}{(k+5)^2} > 1$ for every $k\geq 1$).
 Moreover, it is clear that
\begin{equation*}
   \log L_k  = \log L_1 + (k-1) \log 2 + \sum_{j = 1}^{k-1}
 \log (1 + \tfrac{1}{(j+5)^b})
    \leq \log L_1 + (k-1) \log 2 + \sum_{j = 1}^{k-1} \tfrac{1}{(j+5)^b},
\end{equation*}
so
\begin{equation}
  \label{e:Lkgrowth}
  L_1 2^{k-1} \leq L_k \leq e^{\zeta(b)} L_1 2^{k-1}.
\end{equation}

We use the above scale sequence to define boxes entering our re\-nor\-mal\-i\-zation scheme. For $x \in \Z^d$ and $k \geq 1$, let
 \begin{equation}
  \label{e:Ck}
  C^k_x = [0, L_k)^d \cap \Z^d + x \quad \text{and} \quad D^k_x = [-L_k,2L_k) \cap \Z^d + x.
 \end{equation}
(Observe that the $L_k$'s above need not be integers in general.)

 Given $u > u^{**}$, $k \geq 1$ and a point $x \in \Z^d$, we will be interested
in the probability of the following event
 \begin{equation}
  \label{e:Ak}
  A^k_x(u) = \big\{ C^k_x \xleftrightarrow{\mathcal{V}^u} \Z^d \setminus D^k_x \big\},
 \end{equation}
 pictured in Figure~\ref{f:renorm}.
 Our main objective is to bound the probabilities
 \begin{equation}
  \label{e:pk}
  p_k(u) = \sup_{x \in \Z^d} \mathbb{P} \big[ A^k_x(u) \big] \overset{\eqref{e:ergodic}}= \mathbb{P}[A^k_0(u)].
 \end{equation}

\begin{figure}[ht]
\centering \includegraphics{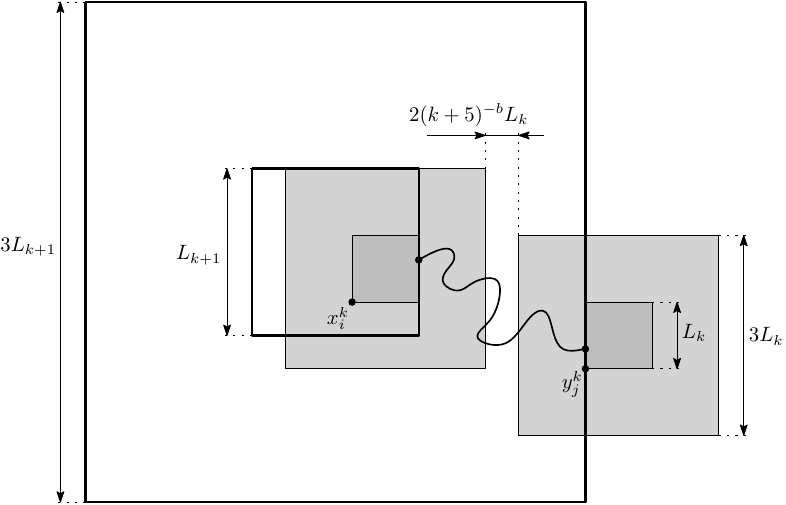}
\caption{An illustration of the event in \eqref{e:Ak} and the inclusion in \eqref{e:Aks}.
}
  \label{f:renorm}
\end{figure}

 In order to employ a renormalization scheme, we will need to relate the events $A^k$ for different scales, as done in the following observation. Given $k \geq 1$,
 \begin{equation}
  \label{e:boxes}
  \begin{array}{c}
  \text{there exist two collections of points $\{x^k_i\}_{i = 1}^{3^d}$ and $\{y^k_j\}_{j = 1}^{2d\cdot 7^{d-1}}$ such that}\\
\begin{array}{l}
  i)  \phantom{ii)} C^{k+1}_0 = \bigcup_{i=1}^{3^d} C^k_{x^k_i},\\
  ii) \phantom{i)}  \bigcup_{j=1}^{2d\cdot 7^{d-1}} C^k_{y^k_j}
\text{ is disjoint from $D^{k+1}_0$
and contains $\partial(\Z^d\setminus D^{k+1}_{0})$,}
\end{array}
  \end{array}
 \end{equation}
 see Figure~\ref{f:renorm}. The above statement is a consequence of \eqref{e:Lkbounds} and the fact that for all $k\geq 1$
we have that $2(1+\frac{1}{(k+5)^b})<3$
and $6(1+\frac{1}{(k+5)^b})<7$. It implies that
 \begin{equation}
  \label{e:Aks}
  A^{k+1}_0 \subset {\textstyle \bigcup
\limits_{\substack{i \leq 3^d\\j \leq 2d\cdot 7^{d-1}}}} A^k_{x^k_i} \cap A^k_{y^k_j},
 \end{equation}
see Figure~\ref{f:renorm}.

It is also important to observe from \eqref{e:boxes} that for any
$i \leq 3^d$ and $j \leq 2d\cdot 7^{d-1}$,
\begin{equation}
\label{bound_dist}
 d \big(D^k_{x^k_i}, D^k_{y^k_j} \big) \geq \lfloor L_{k+1} \rfloor
- 2L_k \overset{\eqref{e:Lkbounds}}\geq \frac{2^{k-1} L_1}{(k+5)^b}.
\end{equation}

Abbreviate $\hat{u}=\frac{u^{**}+u}{2}$; since $u>u^{**}$, we have
$u>\hat{u}>u^{**}$. Then, choose a sufficiently small~$\eps>0$
in such a way that
\[
 \prod_{k=1}^{\infty}\Big(1-\frac{\eps}{k^b}\Big) > \frac{\hat{u}}{u}
\]
and define
\[
 u_k = \frac{\hat{u}}{\prod_{j=1}^{k-1}(1-\eps k^{-b})};
\]
by construction, it holds that $u_k<u$ for all~$k$.
Abbreviate $\kappa_d=\frac{2d\cdot 21^d}{7}$ and let us denote
\begin{equation}
\label{eq_liminf}
 \varrho_d = \liminf_{L\to\infty}\mathbb{P} \big[ [0,L]^d
 \xleftrightarrow{\mathcal{V}^{\hat{u}}} \partial [-L,2L]^d \big];
\end{equation}
then, it holds (recall~\eqref{e:ustarstar3}) that $0 \leq \varrho_d < \kappa_d^{-1}$. The above event stands for the fact that there exists a connecting path between these two sets through $\mathcal{V}^{\hat{u}}$.

Now, we obtain a recursive relation for $p_k(u_k)$, recall~\eqref{e:pk}.

We use~\eqref{main_decr} with $r=3\sqrt{d}L_k$,
$s \geq 2^{k-1}L_1(k+5)^{-b}$
(recall~\eqref{bound_dist}), $u_{k+1}$ on the place of~$u$ and
$\eps k^{-b}$ on the place of~$\eps$ (observe that
$u_k=(1-\eps k^{-b})u_{k+1}$), and use also~\eqref{e:Lkgrowth} and \eqref{e:Aks} to obtain that
\newconstant{c:rec1}
\newconstant{c:rec2}
\begin{equation}
\label{e:recursion}
 p_{k+1}(u_{k+1}) \leq \kappa_d p_k^2(u_k)
     + \useconstant{c:rec1}2^{kd}L_1^d\exp\Big(-\useconstant{c:rec2}
        k^{-2b}\Big(\frac{L_12^k}{(k+5)^b}\Big)^{d-2}\Big),
\end{equation}
where $\useconstant{c:rec2} = \useconstant{c:rec2}(u,b,\eps)$.

Now, let us first consider the case $d\geq 4$ (as mentioned above,
for this case any particular value of $b\in (1,2]$ will do the job,
so in the calculations below
one can assume for definiteness that e.g.\ $b=2$). Let~$h_1>0$ be such that
$\varrho_d < e^{-h_1} < \kappa_d^{-1}$. Choose a sufficiently large
$L_1\geq 100$ in such a way that $p_1(\hat{u}) < e^{-h_1}$ and
\begin{equation}
\label{big_cond_L1}
  \useconstant{c:rec1}2^{kd}L_1^d\exp\Big(-\useconstant{c:rec2}
        k^{-2b}\Big(\frac{L_12^k}{(k+5)^b}\Big)^{d-2}+h_1+ 2^{k+1}\Big)
   < 1-\kappa_d e^{-h_1}
\end{equation}
for all $k\geq 1$ (here we used $d \geq 4$).
Then, we can find small enough~$h_2\in (0,1)$ in such a way that
\begin{equation}
\label{cond_p1_d4}
 p_1(\hat{u})\leq \exp(-h_1-2h_2).
\end{equation}

We then prove by induction that
\begin{equation}
\label{e:induction_to_prove}
 p_k(u_k) \leq \exp(-h_1-h_2 2^{k}).
\end{equation}
Indeed, the base for the induction is provided by~\eqref{cond_p1_d4};
then, we have by~\eqref{e:recursion} that
\[
p_{k+1}(u_{k+1}) \leq \kappa_d \exp(-2h_1-h_2 2^{k+1}) +
\useconstant{c:rec1}2^{kd}L_1^d\exp\Big(-\useconstant{c:rec2}
        k^{-2b}\Big(\frac{L_12^k}{(k+5)^b}\Big)^{d-2}\Big)
\]
so, by~\eqref{big_cond_L1} (recall that $h_2<1$)
\[
 \frac{p_{k+1}(u_{k+1})}{\exp(-h_1-h_2 2^{k+1})}
 \leq \kappa_d e^{-h_1} +
\useconstant{c:rec1}2^{kd}L_1^d\exp\Big(-\useconstant{c:rec2}
        k^{-2b}\Big(\frac{L_1 2^k}{(k+5)^b}\Big)^{d-2}
                   +h_1+ h_2 2^{k+1}\Big),
\]
which is smaller than one, thus proving~\eqref{e:induction_to_prove}.

Observe that for all~$x$
it holds that
\begin{equation}
\label{annulus_path}
\mathbb{P} \big[ 0 \xleftrightarrow{\mathcal{V}^u}x \big] \leq
\mathbb{P} \big[ [-L_k,L_k]^d \xleftrightarrow{\mathcal{V}^u}
\partial [-2L_k,2L_k]^d \big]
\end{equation}
with $k=\max\{m:\frac{3}{2}L_m < \|x\|\}$;
also, $L_k=O(2^k)$ by~\eqref{e:Lkgrowth}.
Since $u_k<u$ for all~$k$,
\eqref{e:induction_to_prove} implies that $p_k(u) \leq \exp(-h_1-h_2 2^{k})$,
we obtain~\eqref{e:connect_d4} from~\eqref{annulus_path}.

Let us now treat the case~$d=3$. Again, let~$h'_1>0$ be such that
$\varrho_3 < e^{-h'_1} < \kappa_3^{-1}$. Choose a sufficiently large
$L_1\geq 100$ in such a way that $p_1(\hat{u}) < e^{-h'_1}$ and
\begin{equation}
\label{big_cond_L1_d=3}
  \useconstant{c:rec1}2^{3k}L_1^3\exp\big(-\useconstant{c:rec2}
        (k+5)^{-3b}L_1 2^{k-1}+h_1\big)
   < 1-\kappa_3 e^{-h'_1}
\end{equation}
for all $k\geq 1$.
Then, we can find small
enough~$h'_2\in \big(0,\frac{1}{4}\useconstant{c:rec2}\big)$
 in such a way that
\begin{equation}
\label{cond_p1_d3}
 p_1(\hat{u})\leq \exp(-h'_1-2h'_2).
\end{equation}

Now, in three dimensions we are going to
 prove by induction that
\begin{equation}
\label{e:induction_to_prove_d=3}
 p_k(u_k) \leq \exp(-h'_1-h'_2 (k+5)^{-3b} 2^{k}).
\end{equation}
Indeed, we have by~\eqref{e:recursion} that
\[
p_{k+1}(u_{k+1}) \leq \kappa_3 \exp(-2h'_1-h'_2 (k+5)^{-3b}2^{k+1}) +
\useconstant{c:rec1}2^{3k}L_1^3\exp\big(-\useconstant{c:rec2}
        (k+5)^{-3b}L_12^k\big)
\]
so, by~\eqref{big_cond_L1_d=3}
(recall that $h_2<\frac{1}{4}\useconstant{c:rec2}$)
\begin{align*}
 \frac{p_{k+1}(u_{k+1})}{\exp(-h'_1-h'_2 (k+6)^{-3b} 2^{k+1})}
 & \leq \kappa_3 e^{-h'_1} \exp\big(-h'_2((k+5)^{-3b}-(k+6)^{-3b})\big) \\
 & \quad {}+
\useconstant{c:rec1}2^{3k}L_1^3\exp\Big(-\useconstant{c:rec2} \frac{L_12^k}{(k+5)^{3b}} +h'_1+ h'_2
  \frac{2^{k+1}}{(k+6)^{3b}} \Big)\\
 & \leq\kappa_3 e^{-h'_1} + \useconstant{c:rec1}2^{3k}L_1^3\exp\big(-\useconstant{c:rec2}
        (k+5)^{-3b}L_12^{k-1}+h'_1\big)\\
 & < 1,
\end{align*}
thus proving~\eqref{e:induction_to_prove_d=3}.
Again, since $u_k<u$ for all~$k$, \eqref{e:induction_to_prove} implies that $p_k(u) \leq \exp(-h'_1-h'_2 (k+5)^{-3b}2^{k})$ for all~$k$, and then we obtain~\eqref{e:connect_d3} with the help of~\eqref{annulus_path} analogously to the case $d\geq 4$.
This concludes the proof of Theorem~\ref{t:connect}.
\end{proof}

\section{Smoothing of discrete sets: proof of Proposition~\ref{p:fatA}}
\label{s:reduction}
In this section we show that any set can be enclosed in
a slightly larger set
with ``smooth enough'' boundaries, and this larger set
has the desired properties (in particular, the
entrance probabilities behave in a good way), as
described in Proposition~\ref{p:fatA}.

To facilitate reading, throughout this section we will
adopt the following convention for denoting
 points and subsets of $\R^d$ which
are not (generally) in~$\Z^d$: they
will be respectively denoted by $\mathsf{x}, \mathsf{y}, \mathsf{z}$ and $\mathsf{A}, \mathsf{B}, \mathsf{D}$, using the \textsf{sans serif} font.
 The usual fonts are reserved to points and subsets of~$\Z^d$.
Also, we use the following (a bit loose but convenient) notation:
if a set $\mathsf{A} \subset \R^d$ was defined, then we denote
by~$A\subset \Z^d$ its
discretization: $A = \mathsf{A} \cap \Z^d$;
conversely, if $A\subset\Z^d$ was a discrete set, then $\mathsf{A}$ just
equals~$A$, but is regarded as a subset of~$\R^d$.

Similarly to the notations in the discrete case, let us write
$\mathsf{B}(\mathsf{x},s)=\{\mathsf{y} \in \R^d: \|\mathsf{x}-\mathsf{y}\|\leq s\}$ for the balls
with radius~$s$, recall that $\|\cdot\|$ stands for the Euclidean norm.
We abbreviate $\mathsf{B}(s) = \mathsf{B}(0,s)$.
It will be convenient to define, for $\mathsf{A} \subseteq \R^d$,
the ball $\mathsf{B}(\mathsf{A}, s) = \bigcup_{\mathsf{x} \in \mathsf{A}} \mathsf{B}(\mathsf{x}, s)$.

\begin{definition}
\label{d:reg}
Let $\mathsf{D} \subset \R^d$ be an open set (not necessarily connected)
with smooth boundary $\partial \mathsf{D}$.
We say that $\mathsf{D}$ is \emph{$s$-regular} if
for any $\mathsf{x} \in \partial \mathsf{D}$ there exist
two balls $\mathsf{B}^\mathsf{x}_{\text{in}} \subset \bar{\mathsf{D}}$ and
$\mathsf{B}^\mathsf{x}_{\text{out}} \subset \R^d \setminus \mathsf{D}$ of radius $s$, such that $\partial \mathsf{D} \cap \mathsf{B}^\mathsf{x}_{\text{out}} = \partial \mathsf{D} \cap
\mathsf{B}^\mathsf{x}_{\text{in}}=\{\mathsf{x}\}$.
Informally speaking, the above definition means that one can touch the boundary
of~$\mathsf{D}$ by spheres of radius~$s$ from inside and outside.
We also adopt the convention that~$\R^d$ is $s$-regular
for any~$s>0$.
\end{definition}

Observe that if $\mathsf{D}$ is an $s$-regular set ,
then for each $\mathsf{x} \in \partial \mathsf{D}$
the balls $\mathsf{B}^\mathsf{x}_{\text{in}}$ and $\mathsf{B}^\mathsf{x}_{\text{out}}$ are unique. Let us denote by $\mathsf{x}^\text{in}$ and $\mathsf{x}^\text{out}$ their respective centers, which lie in the line normal to $\partial \mathsf{D}$ at $\mathsf{x}$.
Also, it is important to keep in mind that if~$\mathsf{D}$
is $s$-regular then it is also $s'$-regular for all $s'\leq s$.


First, we will show that any set $\mathsf{A} \subset \R^d$ can be thickened into a smooth and regular~$\mathsf{A}^{(s)}$ which is
``close'' to~$\mathsf{A}$, see Figure~\ref{f_smoothing}.
This is made precise in the following

\begin{lemma}
\label{l_separate}
There exists a constant \newiconst{c:reg} $\useiconst{c:reg} \in (0,1/5)$ such that, for any set  $\mathsf{A} \subset \R^d$ and $s > 0$, there exist a set $\mathsf{A}^{(s)} \subset \R^d$ with smooth boundary, such that
\begin{enumerate}
\item $\mathsf{A} \subseteq \mathsf{A}^{(s)} \subseteq
\mathsf{B}\big(\mathsf{A},\frac{s}{5}\big)$ and
\item $\mathsf{A}^{(s)}$ is $\useiconst{c:reg} s$-regular in the sense of
Definition~\ref{d:reg}.
\end{enumerate}
\end{lemma}

\begin{proof}
Assume that $\R^d\setminus\mathsf{B}\big(\mathsf{A},\frac{s}{5}\big)$
is nonempty, otherwise the claim is straightforward.
Since we suppose that $\mathsf{A} \subset \R^d$ is arbitrary, we can suppose that $s = 5$ (so that $\frac{s}{5}$=1)
by scaling~$\mathsf{A}$ if necessary.

Let us first tile the space $\R^d$ with compact cubes $\mathsf{K}_m$, of side length $\frac{1}{8 \sqrt{d}}$. More precisely, for $m = (m_1, \dots, m_d) \in \Z^d$, let
\begin{equation}
 \label{e:tile}
 \mathsf{K}_m = \tfrac{1}{8 \sqrt{d}} [m_1, m_1 + 1] \times \dots \times [m_d, m_d + 1].
\end{equation}
With the above definition, $\diam(\mathsf{K}_{m_1} \cup \mathsf{K}_{m_2})\leq\frac{1}{4}$ if
$\mathsf{K}_{m_1}$ and $\mathsf{K}_{m_2}$ have at least one common point.

We first consider the set
\[
 \hat{\mathsf{A}} = \bigcup \mathsf{K}_m,
\]
where the above union is taken over all cubes that either intersect $\mathsf{A}$
or have at least one common point with another cube that intersects $\mathsf{A}$.

Define now the function $\hat f$ to be the convolution of $\1{\hat{\mathsf{A}}}(\cdot)$ with a smooth test function~$\psi \geq 0$, with $\int\psi\, d\mathsf{x} = 1$ and supported
on $\mathsf{B}(\frac{1}{8})$. Clearly, for any $\alpha \in(0,1)$ it holds that
\begin{equation}
 \label{e:Acontained}
 \mathsf{A} \subseteq \{\mathsf{x}: \hat f(\mathsf{x}) > \alpha\},
\end{equation}
so it remains
to show that, for some $\alpha$, the set
$\{\mathsf{x}: \hat f(\mathsf{x}) > \alpha\}$ is $\useiconst{c:reg}$-regular for some
small enough $\useiconst{c:reg}<\frac{1}{5}$ independent of $\mathsf{A}$.

To understand how the above construction depends on the choice of $\mathsf{A}$,
let us scale and recenter the function $\hat f$. More precisely, let $\varphi_{\mathsf{A},m}:\mathsf{B}(0,1) \to \R_+$ 
be the function that associates a point $\mathsf{x} \in \mathsf{B}(0,1)$ 
to $\hat f(\mathsf{x} - m)$. It is important to observe that
\begin{equation}
 \label{e:finitef}
 \begin{array}{c}
   \text{as we vary $\mathsf{A} \subset \Z^d$, 
and $m \in \Z^d$, the functions}\\
   \text{$\varphi_{\mathsf{A},m}$ range over a finite collection of smooth functions,}
 \end{array}
\end{equation}
since it is determined by the finitely many possible configurations of boxes $K_{m'}$ that intersect~$\mathsf{K}_m$ (whether they appear or not in the union defining~$\hat{\mathsf{A}}$).

From the Sard's theorem and the implicit function theorem
one can obtain that
for some $\alpha \in (0,1)$ (in fact, for generic values of $\alpha \in (0,1)$)
the boundary $\{\mathsf{x}: \hat f(\mathsf{x}) = \alpha\}$ is smooth.
Therefore, using \eqref{e:finitef} we can choose $\alpha_o \in (0,1)$ such that $\{\mathsf{x}: \hat f(\mathsf{x}) = \alpha_o\}$ is smooth, independently of the choice of $\mathsf{A}$. We now let $\mathsf{A}' = \{\mathsf{x}: \hat f(\mathsf{x}) > \alpha_o\}$. From \eqref{e:Acontained}, we conclude that $\mathsf{A} \subset \mathsf{A}'$ and from the definition of $\hat f$, we obtain that $\mathsf{A}' \subseteq \mathsf{B}(\mathsf{A}, 1)$. To finish the proof, we should show that $\mathsf{A}'$ is $\useiconst{c:reg}$-regular (with some small
enough constant~$\useiconst{c:reg}$ independent of $\mathsf{A}$).

Since $\partial \mathsf{A}'$ is smooth, we can show that for every $\mathsf{x} \in \partial \mathsf{A}'$, there exist $\mathsf{B}_{\text{in}}$ and $\mathsf{B}_{\text{out}}$ as in Definition~\ref{d:reg}.
Observe that the existence of such balls with radius smaller or equal to $1/4$ only depends on the values of $\hat f$ in $\mathsf{B}(\mathsf{x},1)$. So that the independence of $\useiconst{c:reg}$ of the choice of $\mathsf{A}$ follows from \eqref{e:finitef}.
\end{proof}


At this point,
we can collect the first ingredient for Proposition~\ref{p:fatA}:
we take~$A^{(s)}$ to be the discretization of the set $\mathsf{A}^{(s)}$
 provided by Lemma~\ref{l_separate}.

Now, we prove several geometric properties of regular
sets and their discretizations.
\newiconst{g:til}
\newiconst{g:1/6}
\begin{lemma}
\label{l:reg_spheres}
Abbreviate $\useiconst{g:til}=\frac{1}{200}$
and~$\useiconst{g:1/6}=\frac{1+\sqrt{799}}{200}<\frac{1}{6}$.
Then for any $s$-regular set $\mathsf{A}$ and for any
$\mathsf{v}_1,\mathsf{v}_2\in\partial\mathsf{A}$ such that
$\|\mathsf{v}_1-\mathsf{v}_2\|\leq \useiconst{g:til}s$ it holds that
\begin{equation}
\label{eq:closecenters}
 \|\mathsf{v}_1^\text{\rm out}-\mathsf{v}_2^\text{\rm out}\|
 \leq \useiconst{g:1/6}s
\end{equation}
(observe that, by symmetry, the same holds for
$\mathsf{v}_1^{\text{\rm in}},\mathsf{v}_2^{\text{\rm in}}$).
\end{lemma}

\begin{proof}
 Consider the plane~$\mathcal{L}$ generated by points $\mathsf{v}_1$,
$\mathsf{v}_1^\text{\rm out}$, and
$\mathsf{v}_2^\text{\rm out}+(\mathsf{v}_1-\mathsf{v}_2)$,
see Figure~\ref{f:2spheres} (note that, as indicated on the
picture, $\mathsf{v}_2$ need not lie
on this plane).
\begin{figure}
\centering \includegraphics[width=.6\textwidth]{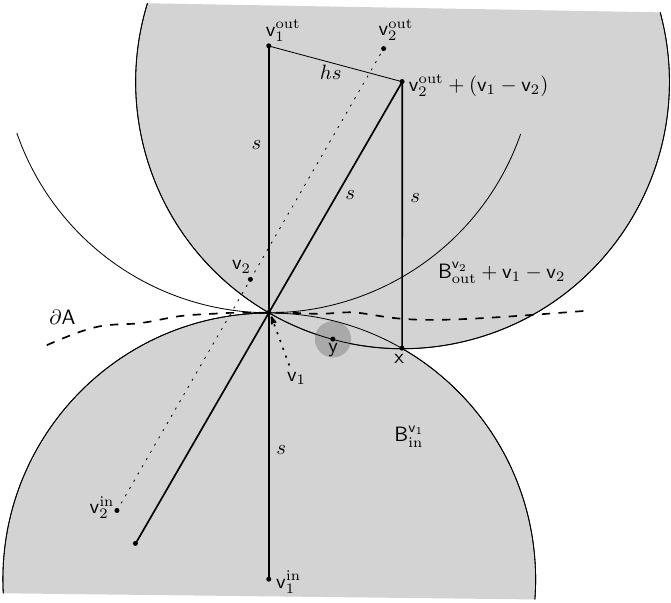}
  \caption{Depiction of the plane $\mathcal{L}$ on the proof of Lemma~\ref{l:reg_spheres}. The
radius of the small gray circle centered in~$\mathsf{y}$ is $as$;
also, on this picture the segment between $\mathsf{v}_2^\text{\rm in}$ and
$\mathsf{v}_2^\text{\rm out}$ (containing also $\mathsf{v}_2$)
does not have an intersection with the plane~$\mathcal{L}$. That is why $v_2$ appears not to intersect $\partial A$.
}
  \label{f:2spheres}
\end{figure}
Let $\mathsf{x}$ be the point that lies on the intersection of
$\partial\mathsf{B}^{\mathsf{v}_1}_\text{in}\cap
 (\partial\mathsf{B}^{\mathsf{v}_2}_\text{out}+{\mathsf{v}_1}-{\mathsf{v}_2})$
with $\mathcal{L}$ and is different from~$\mathsf{v}_1$,
and let~$\mathsf{y}$ be the middle point on the arc of the circle
$(\partial\mathsf{B}^{\mathsf{v}_2}_\text{out}+{\mathsf{v}_1}-{\mathsf{v}_2})
\cap\mathcal{L}$ between~$\mathsf{v}_1$ and~$\mathsf{x}$ (of course, we
mean the arc that lies inside $\mathsf{B}^{\mathsf{v}_1}_\text{in}$).
Abbreviate also
$h=s^{-1}\|\mathsf{v}_2^\text{\rm out}+
(\mathsf{v}_1-\mathsf{v}_2)-\mathsf{v}_1^\text{\rm out}\|$
and $a=s^{-1}d(\mathsf{y}, \partial\mathsf{B}^{\mathsf{v}_1}_\text{in})$;
 with some elementary geometry,
we obtain that
\[
h= 2\sqrt{a-\frac{a^2}{4}}.
\]
But, we must necessarily have
\[
 d(\mathsf{y}, \partial\mathsf{B}^{\mathsf{v}_1}_\text{in}) \leq
\|\mathsf{v}_1-\mathsf{v}_2\|,
\]
because otherwise the point $\mathsf{y}-\mathsf{v}_1+\mathsf{v}_2
\in \mathsf{B}^{\mathsf{v}_2}_\text{out}$ would also
belong to the interior of $\mathsf{B}^{\mathsf{v}_1}_\text{in}$,
a contradiction.
So, we have
\[
 h \leq 2\sqrt{\useiconst{g:til}-\frac{\useiconst{g:til}^2}{4}} =
 \frac{\sqrt{799}}{200},
\]
which means that
\[
 \|\mathsf{v}_1^\text{\rm out}-\mathsf{v}_2^\text{\rm out}\|
 < \Big(\frac{\sqrt{799}}{200}+\useiconst{g:til}\Big)s =
\useiconst{g:1/6}s.
\]
\end{proof}

The next lemma is a consequence of an obvious observation that
the boundary of discretized $s$-regular sets looks locally flat
for large~$s$:
\begin{lemma}
\label{l:locally_flat}
There exist (large enough) $s_0,h_0$ with the following properties.
 Assume that~$\mathsf{A}$ is $s$-regular for some $s \geq s_0$ and
$x,y\notin A$ are such that $\|x-y\|\leq 3\sqrt{d}$. Then
there exists a path between~$x$ and $y$ of length at most~$h_0$ that
does not intersect~$A$.
\end{lemma}
\begin{proof}
 This result is fairly obvious, so we give only a sketch
of the proof (certainly, not the most ``economic'' one).
First, without restricting generality,
one can assume that
$\max( d(x,\mathsf{A}),  d(y,\mathsf{A}))<3\sqrt{d}$
(otherwise, the ball of radius $3\sqrt{d}$ centered in one of the points
does not have intersection with~$A$ and contains the other
point; then, use the fact that this discrete ball is a connected graph).
 Then, let $\mathsf{z}\in\partial\mathsf{A}$ be a point on the boundary
closest to~$x$, $z$ be the point in~$A$ closest to~$\mathsf{z}$,
and consider the cube
\[
 G = \{z'\in\Z^d : \|z'-z\|_\infty\leq \lceil 7\sqrt{d}\rceil\}
\]
(where $\|\cdot\|_\infty$ is the maximum norm).
Assume without loss of generality that the projection of the normal vector to $\partial\mathsf{A}$ at~$\mathsf{z}$ on the first coordinate vector is at least~$\frac{1}{\sqrt{d}}$. Then, the claim of the lemma follows once we prove that for all large enough~$s$
\begin{equation}
 \label{e:GAconn}
 \text{the set $G\setminus \mathsf{A}$ is connected.}
\end{equation}
Indeed, for $s$ large enough $\{v, v+e_1\}$ is not fully inside $\R^d\setminus (\mathsf{B}_\text{in}^{\mathsf{z}} \cup \mathsf{B}_\text{out}^{\mathsf{z}})$
for any~$v$ in~$G$. This implies that $G \setminus \mathsf{A}$ is given by $G \cap \mathsf{B}^z_{\text{out}}$ together with some extra points in the neighborhood of this set, implying \eqref{e:GAconn} and concluding the proof of the lemma.
\end{proof}

Observe that Lemma~\ref{l:locally_flat} implies that for any $x\in\partial A$ and $y\notin A$ such that $\|x-y\|\leq 2\sqrt{d}$, it holds that
\begin{equation}
\label{eq:entr_close}
P_y[X_{H_A}=x] \geq (2d)^{-h_0}.
\end{equation}


\newconstant{c:L34''}
\newconstant{c:L34'''}
\newconstant{c:LL34''}
\newconstant{c:LL34'''}
Next, we need an elementary result about escape probabilities from spheres:
\begin{lemma}
\label{l_hit_spheres}
 There exist positive constants $s_1,\useconstant{c:L34''},
\useconstant{c:L34'''}, \useconstant{c:LL34''},
\useconstant{c:LL34'''}$
(depending only on the dimension) such that
for all $\mathsf{y}\in\R^d$, for all
$s \geq s_1$
and every $x\in \bally(\mathsf{y},2 s) \setminus \bally(\mathsf{y},s)$,
we have
\begin{equation}
\label{eq_hit_spheres}
 \useconstant{c:L34''} \frac{\|x-\mathsf{y}\|-s}{s}
\leq P_x[H_{\bally(\mathsf{y},s)}>H_{\Z^d\setminus \bally(\mathsf{y},2s)}]
 \leq  \useconstant{c:L34'''} \frac{\|x-\mathsf{y}\|-s+1}{s},
\end{equation}
and for all $x\in \bally(\mathsf{y},3 s) \setminus \bally(\mathsf{y},s)$
\begin{equation}
\label{eq_hit_spheres'}
 \useconstant{c:LL34''} \frac{3s-\|x-\mathsf{y}\|}{s}
\leq P_x[H_{\bally(\mathsf{y},s)}<H_{\Z^d\setminus \bally(\mathsf{y},3s)}]
 \leq  \useconstant{c:LL34'''} \frac{3s-\|x-\mathsf{y}\|+1}{s}.
\end{equation}
\end{lemma}

\begin{proof}
By a direct calculation, it is elementary to obtain that,
for large enough~$s$ (not depending on~$\mathsf{y}$), the process
$\|X_{n\wedge H_{\bally(s)}}-\mathsf{y}\|^{-(d-1)}$ is a supermartingale,
and $\|X_{n\wedge H_{\bally(s)}}-\mathsf{y}\|^{-(d-\tfrac{5}{2})}$ is a submartingale,
see e.g.\ the proof of Lemma~1 in~\cite{HMP99}.
 From the Optional Stopping Theorem, we obtain that
\begin{equation}
\label{eq:auxL34}
 \frac{s^{-(d-1)}-\|x-\mathsf{y}\|^{-(d-1)}}{s^{-(d-1)}-(2s+1)^{-(d-1)}}
  \leq P_x[H_{\bally(s)}>H_{\Z^d\setminus \bally(2s)}] \leq \frac{(s-1)^{-d+\tfrac{5}{2}}-\|x-\mathsf{y}\|^{-d+\tfrac{5}{2}}}{(s-1)^{-d+\tfrac{5}{2}}-
(2s)^{-d+\tfrac{5}{2}}},
\end{equation}
where the above balls are centered in $\mathsf{y}$.
Then~\eqref{eq_hit_spheres} follows from~\eqref{eq:auxL34}
with the observation that $0<\frac{\|x - \mathsf{y}\|}{s}\leq 2$
and some elementary calculus. The proof of~\eqref{eq_hit_spheres'}
is completely analogous.

In fact, with some more effort,
one can obtain that $s_1=\frac{\sqrt{d}}{2}$ (observe that
$\bally(\mathsf{y},\frac{\sqrt{d}}{2})$ is nonempty for all
$\mathsf{y}\in\R^d$), but we do not need this stronger
fact for the present paper.
\end{proof}

We now need estimates on the entrance measure of a set in~$\Z^d$ which has been obtained from the discretization of a regular set $\mathsf{D} \subset \R^d$.
For this, we will need the following definitions.
Let $D = \mathsf{D} \cap \Z^d$ and fix $x \in \partial D$, we write $\mathsf{x}$ for the closest point to $x$ in $\partial \mathsf{D}$ (it can be chosen arbitrarily in case of ties) and note that $\lVert x - \mathsf{x} \rVert \leq 1$. We define $x^\text{in}$ and $x^\text{out}$ to be the closest points to $\mathsf{x}^\text{in}$ and $\mathsf{x}^\text{out}$ in $\Z^d$ (again chosen arbitrarily in case of ties). Observe that $\lVert x^\text{out} - \mathsf{x}^\text{out} \rVert$ is at most $\frac{\sqrt d}{2}$ (and the same holds for $x^\text{in}$
and $\mathsf{x}^\text{in}$).

\newconstant{c:L35'}
\newconstant{c:L35''}
\newconstant{c:L35'''}
\begin{lemma}
\label{l_entrance_measure}
\begin{itemize}
 \item[(i)] Suppose that $\mathsf{A}$ is an $s$-regular set for some $s \geq s_0+\sqrt{d}$ and $y \in \partial A,x\in \Z^d\setminus A$ are such that $\|x-y\|\geq 2s$. Then $P_{x}[X_{H_A}= y]\leq \useconstant{c:L35''} s^{-(d-1)}$.
 \item[(ii)] Assume that $\mathsf{A}$ is $s$-regular with $s \geq s_0+\sqrt{d}$ and $y \in \partial A$; then for every $x\in \bally(y^{out},s/2)$, we have $P_{x}[X_{H_A}= y,H_A<H_{\Z^d\setminus \bally(y^{out},s+\sqrt{d})}] \geq \useconstant{c:L35'''} s^{-(d-1)}$.
\end{itemize}
\end{lemma}

\begin{proof}
Given $\mathsf{A}$ and $y \in \partial A$ as above,
recall that~$\mathsf{y}$ stands for the closest point to~$y$
in~$\partial \mathsf{A}$ (chosen arbitrarily in case of ties). By Definition~\ref{d:reg}, we know that the ball $\mathsf{B}^{\mathsf{y}}_{\text{in}} \subset \R^d$ of radius~$s$
lies fully inside $\mathsf{A}$.
 Moreover, since $y^\text{in}$ is at distance at most $\frac{\sqrt{d}}{2}$
  from $\mathsf{y}^\text{in}$, we conclude that
\begin{equation}
 \label{e:BinB}
 B^y_\text{in} := \bally\Big(y^\text{in}, s
- \frac{\sqrt{d}}{2}\Big)
 \subseteq \mathsf{B}^{\mathsf{y}}_{\text{in}},
\text{ and also }
 B^y_\text{out} := \bally\Big(y^\text{out}, s
- \frac{\sqrt{d}}{2}\Big)
\subseteq \mathsf{B}^{\mathsf{y}}_{\text{out}},
\end{equation}
see Figure~\ref{f:discrete}.

\begin{figure}
\centering \includegraphics[width=.5\textwidth]{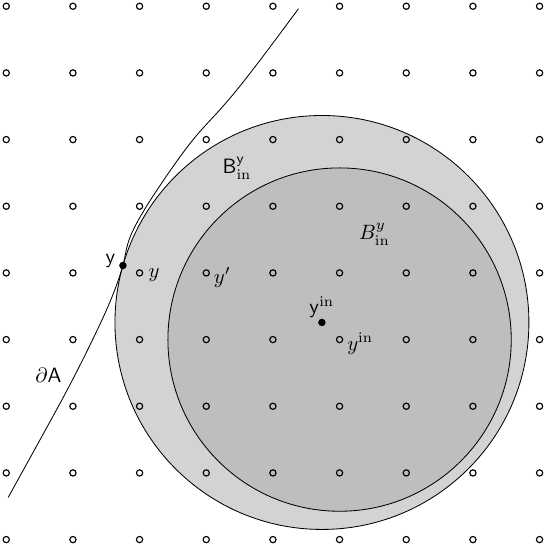}
  \caption{On the proof of Lemma~\ref{l_entrance_measure}}
  \label{f:discrete}
\end{figure}
Let $y'\in \partial B^y_\text{in}$ be the point which is closest to~$y$
(it could happen that $y'$ is $y$ itself).
By construction we have
$ d(y,B^y_\text{in})\leq \sqrt{d}$,
therefore $\|y-y'\|\leq\frac{3}{2}\sqrt{d}$, and
so by Lemma~\ref{l:locally_flat} we have
\newconstant{c:yyprime}
\begin{equation}
 \label{e:yyprime}
 P_{x}[X_{H_B} = y'] \geq \useconstant{c:yyprime} P_{x}[X_{H_A}= y].
\end{equation}
Employing Proposition~6.5.4 of~\cite{LL10}, we obtain that
\begin{equation}
 P_{x}[X_{H_B}= y'] \leq \useconstant{c:L35''} s^{-(d-1)},
\end{equation}
which together with~\eqref{e:yyprime} proves (i).

A discretization argument analogous to the above gives~(ii) for
all $x\in \bally(y^{out},\frac{s}{2}-\frac{\sqrt{d}}{4})$ as a direct consequence of Lemma~6.3.7 of~\cite{LL10}; then, using Lemma~\ref{l:locally_flat}, we obtain the desired
statement for all $x\in \bally(y^{out},\frac{s}{2})$.
\end{proof}

Next, aiming to the proof of~\eqref{hip1}, we formulate and prove the following result:
\newconstant{c:P33'}
\newconstant{c:P33}
\begin{proposition}
\label{p_entrance}
There exist constants $s_0,\useiconst{reg8},\useiconst{reg10} > 0$
such that if $s\geq s_0$,
$\mathsf{A} \subset \R^d$ is $\useiconst{c:reg} s$-regular
and if $y_1,y_2 \in \partial A$
are such that $\|y_1 - y_2\| \leq \useiconst{reg8}s$, then
there exists a set $\hat D$ (depending on $y_1,y_2$) that separates $\{y_1,y_2\}$ from
$\partial \bally(y_1,2\useiconst{c:reg}s)$ (i.e., any nearest-neighbor path starting at
$\partial \bally(y_1,2\useiconst{c:reg}s)$ that enters~$A$ at $\{y_1,y_2\}$, must pass through $\hat D$)
such that
\begin{equation}
\label{eq_entrance}
  \sup_{\substack{x\in\hat D;\\ P_x [X_{H_{A}} = y_1]>0}}
\frac{P_x [X_{H_{A}} = y_2]}
{P_x [X_{H_{A}} = y_1, H_{A}<H_{\Z^d\setminus
\bally(y_1,\frac{5}{2}\useiconst{c:reg} s)}]}
\leq \useiconst{reg10}.
\end{equation}
\end{proposition}
Let us already mention that the constants
$\useiconst{reg8},\useiconst{reg10}$ here are exactly those that we need
in Proposition~\ref{p:fatA}.

\begin{proof}[Proof of Proposition~\ref{p_entrance}]
Define
\begin{equation}
\label{def_s}
 s_2 = \max\big\{\useiconst{c:reg}^{-1}s_0, \;
 36(\useiconst{c:reg}\useiconst{g:til})^{-1}(s_1+\sqrt{d})\big\}
\geq 18(\useiconst{c:reg}\useiconst{g:til})^{-1}.
\end{equation}
Also, we define $\useiconst{reg8}=\frac{1}{3}\useiconst{c:reg}\useiconst{g:til}$.
Given $y_1$ and $y_2$ in $\partial A$ such that $\|y_1-y_2\|<\useiconst{reg8}s$, let us define
\begin{equation}
 D = \Big\{z \in \Z^d \setminus A:
  d(z,\mathsf{A})\leq
\frac{1}{2}\useiconst{reg8}s \text{ and }
    d(z,y_k)\leq 2\useiconst{reg8}s, k=1,2\Big\},
\end{equation}
and
\begin{equation}
  \label{e:Ghat}
  \hat D = \big\{ z \in D: \text{ there exists }v \in \Z^d \setminus
  (A \cup D) \text{ such that $z$ and $v$ are neighbors} \big\},
\end{equation}
see Figure~\ref{f_entrance_prob}. Intuitively speaking, $\hat D$
is the part of the boundary of~$D$ not adjacent to~$A$.

We now claim that
\begin{equation}
 \label{e:hatGfar}
 \begin{array}{c}
  \text{all sites of~$\hat D$ are at distance at
 least
$\displaystyle\frac{1}{2}\useiconst{reg8}s - 1$
 from $\{y_1,y_2\}$.}
 \end{array}
\end{equation}
To see why this is true, observe first that for $z \in \hat D$, the point $v \leftrightarrow z$ as in~\eqref{e:Ghat} is not in~$D$.
Therefore, we either have $ d(v,\{y_1,y_2\}) > 2\useiconst{reg8}s$
or $ d(v, \mathsf{A}) \geq \frac{1}{2}\useiconst{reg8}s$,
in both cases~\eqref{e:hatGfar} holds.

In fact, to prove~\eqref{eq_entrance},
it is enough to prove that for all $z\in\hat D$
\newconstant{c:hitfrom_z'}
\newconstant{c:hitfrom_z''}
\begin{equation}
\label{hit_from_z1}
P_x [X_{H_{A}} = y_1, H_{A}<H_{\Z^d\setminus
\bally(y_1,\frac{5}{2}\useiconst{c:reg} s)}] \geq \useconstant{c:hitfrom_z'}
 \frac{ d(z,A)}{s^d}
\end{equation}
and
\begin{equation}
\label{hit_from_z2}
 P_z[X_{H_A}=y_2]
  \leq \useconstant{c:hitfrom_z''}\frac{ d(z,A)}{s^d}.
\end{equation}
 The idea behind these two bounds is depicted on
 Figure~\ref{f_entrance_prob}, which we now turn into a rigorous proof.
\begin{figure}
\centering \includegraphics[width=0.8\textwidth]{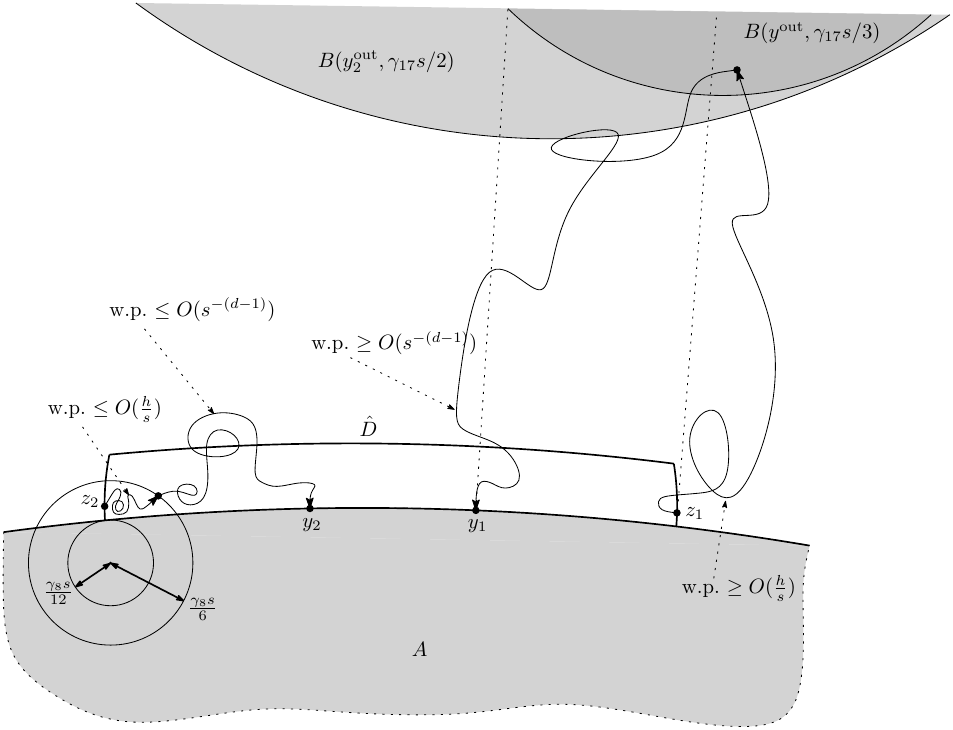}
  \caption{On the proof of Proposition~\ref{p_entrance}: lower
bound for $P_{z_1}[X_{H_A}=y_1, \dots]$ and upper
bound for $P_{z_2}[X_{H_A}=y_2]$; we have $h\simeq  d(z_{1,2},A)$,
and ``w.p.'' stands for ``with probability''.
}
  \label{f_entrance_prob}
\end{figure}
To obtain~\eqref{hit_from_z1},
we proceed in the following way.
Consider some $\mathsf{y} \in \partial A$
such that $ d(z,A)\geq \| z - \mathsf{y}\|$, and observe
that~\eqref{eq_hit_spheres'} implies that
\newconstant{c:dist/s}
\begin{equation}
\label{e:p26}
 P_z[H_{\bally(\mathsf{y}^\text{out},\useiconst{c:reg} s/3)}<H_A]
      \geq \frac{\useconstant{c:dist/s} d(z,A)}{s}.
\end{equation}
Let $\mathsf{y}_k\in \partial A$ be the closest boundary point to~$y_k$;
clearly, we have $\|\mathsf{y}_k-y_k\|\leq 1$.
Then, it holds that $\|y_1-z\|\leq 2\useiconst{reg8}s$ and
$\|z-\mathsf{y}\|\leq \frac{1}{2}\useiconst{reg8}s$,
and thus, by~\eqref{def_s}
\[
 \|\mathsf{y}_1-\mathsf{y}\| \leq \frac{5}{2} \useiconst{reg8}s+1 < 3\useiconst{reg8}s
 = \useiconst{c:reg}\useiconst{g:til}s.
\]
So, by Lemma~\ref{l:reg_spheres} it holds that
$\|\mathsf{y}^{\text{out}}-\mathsf{y}_1^\text{out}\|
\leq \useiconst{c:reg}\useiconst{g:1/6} s < \frac{1}{6}\useiconst{c:reg}s$,
which implies that
$\mathsf{B}(\mathsf{y}^\text{out},\frac{1}{3}\useiconst{c:reg}s)
\subset \mathsf{B}(\mathsf{y}_1^\text{out},\frac{1}{2}\useiconst{c:reg}s)$.
Observing that $\bally(y_1^\text{out},\useiconst{c:reg}s+\sqrt{d})
\subset \bally(y_1,\frac{5}{2}\useiconst{c:reg} s)$
Applying Lemma~\ref{l_entrance_measure}~(ii)
and using~\eqref{e:p26}, we obtain~\eqref{hit_from_z1}.

To prove~\eqref{hit_from_z2}, we proceed in the following way.
Recall that if a set is $r$-regular then it is $r'$-regular
for all $r'\leq r$; so, if $ d(z,A)\geq \frac{1}{3}\useiconst{reg8}s$
then Lemma~\ref{l_entrance_measure}~(i) already implies~\eqref{hit_from_z2}.
Assume now that $z\in\hat D$ is such that
$ d(z,A) < \frac{1}{3}\useiconst{reg8}s$ and
let $\mathsf{y}\in\partial \mathsf{A}$
be such that $ d(z,A)\geq\|z-\mathsf{y}\|$.
Let us show that then we have
$\|\mathsf{y}-y_2\|\geq \frac{1}{2}\useiconst{reg8}s$.
Indeed, by construction of~$\hat D$ there exists $v\notin A\cup G$
with $\|z-v\|=1$ and
such that either $ d(v,A)\geq \frac{1}{2}\useiconst{reg8}s$
or $\min(\|v-y_1\|,\|v-y_2\|)>2\useiconst{reg8}s$.
The first possibility is ruled out since then we would have
$ d(z,A)>\frac{1}{2}\useiconst{reg8}s-1$ which contradicts to
$ d(z,A) < \frac{1}{3}\useiconst{reg8}s$ because
of~\eqref{def_s}. Then, the second alternative implies
that $\|v-y_2\|>\useiconst{reg8}s$, so
$\|z-y_2\|>\useiconst{reg8}s-1$. This means that
\[
 \|\mathsf{y}-y_2\|\geq \useiconst{reg8}s-1 -
   \frac{1}{3}\useiconst{reg8}s \geq \frac{1}{2}\useiconst{reg8}s
\]
again because of~\eqref{def_s}.

Let $\hat{\mathsf{v}}$ be the center of the ball with
radius~$\frac{1}{12}\useiconst{reg8}s$
that touches $\partial \mathsf{A}$ at~$\mathsf{y}$ from inside; we
have by~\eqref{def_s} that
\[
 \inf_{v'\in
\mathsf{B}(\hat{\mathsf{v}},\frac{1}{6}\useiconst{reg8}s)}
\|v'-y_2\|
\geq \frac{1}{2}\useiconst{reg8}s-\frac{1}{4}\useiconst{reg8}s-1
\geq \frac{1}{12}\useiconst{reg8}s+1.
\]
Then, one can write
\newconstant{c:endproof'}
\newconstant{c:endproof''}
\[
 P_z[X_{H_A}=y_2]
\leq
   P_z[H_{\Z^d\setminus \bally(\hat{\mathsf{v}},
  \frac{1}{6}\useiconst{reg8}s)}<
H_{\bally(\hat{\mathsf{v}}, \frac{1}{12}\useiconst{reg8}s)}]
\sup_{z': \|z'-y_2\|\geq \frac{1}{12}\useiconst{reg8}s}
        P_{z'}[X_{H_A}=y_2],
\]
and use Lemma~\ref{l_hit_spheres} to obtain that the
first term in the right-hand side of the above display is at most
$\useconstant{c:endproof'}s^{-1} d(z,A)$.
By Lemma~\ref{l_entrance_measure}~(i),
 the second term is bounded from above by
 $\useconstant{c:endproof''}s^{-(d-1)}$.
This concludes the proof of~\eqref{hit_from_z2} and hence
of Proposition~\ref{p_entrance}.
\end{proof}

We now collect the ingredients necessary
for the proof of Proposition~\ref{p:fatA}:
\begin{itemize}
 \item as already mentioned just before Lemma~\ref{l:reg_spheres},
the sets $A^{(s)}_{1,2}$ are the discretizations of
the sets $\mathsf{A}^{(s)}_{1,2}$
provided by Lemma~\ref{l_separate};
 \item we take the same $s_2$ provided by~\eqref{def_s} and
define $\useiconst{regular} = \frac{1}{2}\useiconst{c:reg}$;
 \item existence of $\useiconst{reg2}$ suitable for~\eqref{hip2}
and~\eqref{hip3} follows then from Lemma~\ref{l_entrance_measure};
 \item the claim~\eqref{hip1} follows from Proposition~\ref{p_entrance},
with the right constants $\useiconst{reg8},\useiconst{reg10}$,
as we already mentioned.
\end{itemize}
So, the only unattended item in Proposition~\ref{p:fatA}
is~\eqref{hip4}.
But it is straightforward to obtain~\eqref{hip4} from a
projection argument: let $\mathsf{y}\in\partial\mathsf{A}_k^{(s)}$
be a closest point to $y\in\partial A_k^{(s)}$ and assume without
lost of generality that the projection of the normal vector at~$\mathsf{y}$
to the first coordinate is at least $\frac{1}{\sqrt{d}}$.
Then, the intersection of projections of
$\mathsf{B}^{\mathsf{y}}_\text{in}\cap \mathsf{B}(\mathsf{y},
\useiconst{reg8}s-1)$ and
$\mathsf{B}^{\mathsf{y}}_\text{out}\cap \mathsf{B}(\mathsf{y},
\useiconst{reg8}s-1)$ along the first coordinate axis contains
a ($(d-1)$-dimensional) ball of radius~$O(s)$, and this proves~\eqref{hip4}
(since on the preimage of each integer point which lies within
 this intersection there
should be at least one point of $\partial A_k^{(s)}\cap
\bally(y,\useiconst{reg8}s)$).
This concludes the
proof of Proposition~\ref{p:fatA}.
\hfill $\Box$

\bigskip
\footnotesize
\noindent\textit{Acknowledgments.}
The authors are very grateful to Caio Alves, Pierre-François Rodriguez and the anonymous referee, who carefully read the manuscript, making many useful suggestions and corrections.
S.P.\ was partially supported by FAPESP (grant 2009/52379--8) and CNPq (grant 300886/2008--0).
A.T.\ received support from CNPq (grant 306348/2012--3).
The authors thank IMPA and IMECC--UNICAMP for financial support and hospitality during their mutual visits.

\bibliographystyle{plain}
\bibliography{all}

\def\cprime{$'$}
\begin{thebibliography}{10}

\bibitem{Bel11}
David Belius.
\newblock {Cover levels and random interlacements.}
\newblock {\em Ann. Appl. Probab.}, 22(2):522--540, 2012.

\bibitem{Bel12}
David {Belius}.
\newblock {Gumbel fluctuations for cover times in the discrete torus.}
\newblock {\em {Probab. Theory Relat. Fields}}, 157(3-4):635--689, 2013.

\bibitem{CTW10}
Ji{\v{r}}{\'{\i}} {\v{C}}ern{\'y}, Augusto Teixeira, and David Windisch.
\newblock Giant vacant component left by a random walk in a random
  {$d$}-regular graph.
\newblock {\em Ann. Inst. Henri Poincar\'e Probab. Stat.}, 47(4):929--968,
  2011.

\bibitem{CP12}
Ji\v{r}\'\i{} {\v{C}}ern\'y and Serguei Popov.
\newblock On the internal distance in the interlacement set.
\newblock {\em Electron. J. Probab.}, 17:no. 29, 1--25, 2012.

\bibitem{HMP99}
F.~den Hollander, M.V. Menshikov, and S.Yu. Popov.
\newblock A note on transience versus recurrence for a branching random walk in
  random environment.
\newblock {\em J. Statist. Phys.}, 95(3-4):587--614, 1999.

\bibitem{DRS12}
Alexander {Drewitz}, Bal\'azs {R\'ath}, and Art\"em {Sapozhnikov}.
\newblock {Local percolative properties of the vacant set of random
  interlacements with small intensity.}
\newblock {\em {Ann. Inst. Henri Poincar\'e, Probab. Stat.}}, 50(4):1165--1197,
  2014.

\bibitem{FP99}
P.~J. Fitzsimmons and Jim Pitman.
\newblock Kac's moment formula and the {F}eynman-{K}ac formula for additive
  functionals of a {M}arkov process.
\newblock {\em Stochastic Process. Appl.}, 79(1):117--134, 1999.

\bibitem{Gri99}
Geoffrey Grimmett.
\newblock {\em Percolation}, volume 321 of {\em Grundlehren der Mathematischen
  Wissenschaften [Fundamental Principles of Mathematical Sciences]}.
\newblock Springer-Verlag, Berlin, second edition, 1999.

\bibitem{L91}
Gregory~F. Lawler.
\newblock {\em Intersections of random walks}.
\newblock Probability and its Applications. Birkh{\"a}user Boston Inc., Boston,
  MA, 1991.

\bibitem{LL10}
Gregory~F. Lawler and Vlada Limic.
\newblock {\em Random walk: a modern introduction}, volume 123 of {\em
  Cambridge Studies in Advanced Mathematics}.
\newblock Cambridge University Press, Cambridge, 2010.

\bibitem{Men86}
M.~V. Men{\cprime}shikov.
\newblock Coincidence of critical points in percolation problems.
\newblock {\em Dokl. Akad. Nauk SSSR}, 288(6):1308--1311, 1986.

\bibitem{MP13}
Mikhail {Menshikov} and Serguei {Popov}.
\newblock {On range and local time of many-dimensional submartingales.}
\newblock {\em {J. Theor. Probab.}}, 27(2):601--617, 2014.

\bibitem{PSSS11}
Yuval Peres, Alistair Sinclair, Perla Sousi, and Alexandre Stauffer.
\newblock Mobile geometric graphs: detection, coverage and percolation.
\newblock In {\em Proceedings of the {T}wenty-{S}econd {A}nnual {ACM}-{SIAM}
  {S}ymposium on {D}iscrete {A}lgorithms}, pages 412--428, Philadelphia, PA,
  2011. SIAM.

\bibitem{RS12}
Bal{\'a}zs R{\'a}th and Art{\"e}m Sapozhnikov.
\newblock On the transience of random interlacements.
\newblock {\em Electron. Commun. Probab.}, 16:379--391, 2011.

\bibitem{RS12b}
Bal\'azs {R\'ath} and Art\"em {Sapozhnikov}.
\newblock {The effect of small quenched noise on connectivity properties of
  random interlacements.}
\newblock {\em {Electron. J. Probab.}}, 18:20, 2013.

\bibitem{R08}
Sidney~I. Resnick.
\newblock {\em Extreme values, regular variation and point processes}.
\newblock Springer Series in Operations Research and Financial Engineering.
  Springer, New York, 2008.
\newblock Reprint of the 1987 original.

\bibitem{Ros82}
Yu.~A. Rozanov.
\newblock {\em Markov random fields}.
\newblock Applications of Mathematics. Springer-Verlag, New York, 1982.
\newblock Translated from Russian by Constance M. Elson.

\bibitem{SS09}
Vladas Sidoravicius and Alain-Sol Sznitman.
\newblock Percolation for the vacant set of random interlacements.
\newblock {\em Comm. Pure Appl. Math.}, 62(6):831--858, 2009.

\bibitem{SS10}
Vladas Sidoravicius and Alain-Sol Sznitman.
\newblock Connectivity bounds for the vacant set of random interlacements.
\newblock {\em Ann. Inst. Henri Poincar\'e Probab. Stat.}, 46(4):976--990,
  2010.

\bibitem{S09b}
Alain-Sol Sznitman.
\newblock On the domination of random walk on a discrete cylinder by random
  interlacements.
\newblock {\em Electron. J. Probab.}, 14:no. 56, 1670--1704, 2009.

\bibitem{Szn09b}
Alain-Sol Sznitman.
\newblock Random walks on discrete cylinders and random interlacements.
\newblock {\em Probab. Theory Related Fields}, 145(1-2):143--174, 2009.

\bibitem{S09}
Alain-Sol Sznitman.
\newblock Upper bound on the disconnection time of discrete cylinders and
  random interlacements.
\newblock {\em Ann. Probab.}, 37(5):1715--1746, 2009.

\bibitem{Szn09}
Alain-Sol Sznitman.
\newblock Vacant set of random interlacements and percolation.
\newblock {\em Ann. of Math. (2)}, 171(3):2039--2087, 2010.

\bibitem{Szn09c}
Alain-Sol Sznitman.
\newblock A lower bound on the critical parameter of interlacement percolation
  in high dimension.
\newblock {\em Probab. Theory Related Fields}, 150(3-4):575--611, 2011.

\bibitem{Szn11}
Alain-Sol Sznitman.
\newblock On the critical parameter of interlacement percolation in high
  dimension.
\newblock {\em Ann. Probab.}, 39(1):70--103, 2011.

\bibitem{Szn12}
Alain-Sol Sznitman.
\newblock Decoupling inequalities and interlacement percolation on
  $\mathcal{G}\times \mathbb{Z}$.
\newblock {\em Inventiones Mathematicae}, 187:645--706, 2012.

\bibitem{Szn12b}
Alain-Sol Sznitman.
\newblock An isomorphism theorem for random interlacements.
\newblock {\em Electron. Commun. Probab.}, 17:no. 9, 1--9, 2012.

\bibitem{Szn12a}
Alain-Sol {Sznitman}.
\newblock {On ${(\mathbb Z / N\mathbb Z)}^2$-occupation times, the Gaussian
  free field, and random interlacements.}
\newblock {\em {Bull. Inst. Math., Acad. Sin. (N.S.)}}, 7(4):565--602, 2012.

\bibitem{Tei09}
A.~Teixeira.
\newblock Interlacement percolation on transient weighted graphs.
\newblock {\em Electron. J. Probab.}, 14:no. 54, 1604--1628, 2009.

\bibitem{Tei09b}
Augusto Teixeira.
\newblock On the uniqueness of the infinite cluster of the vacant set of random
  interlacements.
\newblock {\em Ann. Appl. Probab.}, 19(1):454--466, 2009.

\bibitem{T10}
Augusto Teixeira.
\newblock On the size of a finite vacant cluster of random interlacements with
  small intensity.
\newblock {\em Probab. Theory Related Fields}, 150(3-4):529--574, 2011.

\bibitem{TW10}
Augusto Teixeira and David Windisch.
\newblock On the fragmentation of a torus by random walk.
\newblock {\em Comm. Pure Appl. Math.}, 64(12):1599--1646, 2011.

\bibitem{EJP309}
Boris Tsirelson.
\newblock Brownian local minima, random dense countable sets and random
  equivalence classes.
\newblock {\em Electron. J. Probab.}, 11:no. 7, 162--198, 2006.

\bibitem{Win08}
David Windisch.
\newblock Random walk on a discrete torus and random interlacements.
\newblock {\em Electron. Commun. Probab.}, 13:140--150, 2008.

\end{thebibliography}

\end{document}